\documentclass{amsart}
\usepackage{basic}

\title{A comparison problem for abelian surfaces and descent for symplectic orbital integrals}
\author{Thomas R\"ud}
\address{Department of Mathematics, Massachusetts Institute of Technology\\Cambridge, MA, 02139}
\email{rud@mit.edu}
\date{}

\hypersetup{
    colorlinks=true,
    linkcolor=red,
    filecolor=magenta,      
    urlcolor=cyan,
    }
\begin{document}

\begin{abstract}
To answer a question about the distribution of products of elliptic curves in isogeny classes of abelian surfaces defined over finite fields, we compute specific orbital integrals in the group $\mathrm{GSp}_4$. More precisely, we compute integrals over the orbits of elements in the subgroup $\mathrm{GL}_2\times_{\det} \mathrm{GL}_2$. As a first step towards a complete solution of the problem, this article contains explicit computations for arbitrary orbital integrals of spherical functions over this subgroup, and also compute orbital integrals over $\mathrm{GSp}_4$ in a large number of cases. 
\end{abstract}

\maketitle

\tableofcontents

\section{Introduction}

A big focus of modern representation theory of $p$-adic groups is the study of several trace formulae (relative, stable, \dots). The 
geometric side of the trace formula is expressible in terms of orbital integrals. 

Given an algebraic group, such objects integrate test functions over some conjugation orbits. Despite their central role, explicit computations 
are still largely not explored for groups beyond $\mathrm{GL}_2$.

In recent explorations of mass formulae for abelian varieties over finite fields (\cite{ag-gsp}), the authors 
worked with orbital integrals appearing in Langlands--Kottwitz formula to measure the size of isogeny classes of
abelian varieties in terms of similitude groups. 

Jeff Achter followed this exploration with a question: ``Given a product of two elliptic curves, how many isomorphism classes
in its isogeny class as a principally polarized abelian variety also decompose as products of elliptic curves?''

At first glance, this problem may appear straightforward, as comparing orbital integrals is often simpler than computing them and has been addressed using tools like parabolic descent and endoscopy, but none of those tools can be applied to solve the problem fully.

While we do not fully answer the question in this paper, we introduce the computational tools to tackle explicit computations of orbital integrals and will be instrumental in the next steps. 
This article aims to illustrate the techniques and use them in new nontrivial examples. \\

Firstly, \S \ref{sec:prelim} introduces the groups that we will study, the terminology to state the problem.
Then, we dedicate \S \ref{sec:kottmethod} to introduce Kottwitz's approach, relating orbital integrals to point count on buildings, using $\mathrm{SL}_2$ as an example. In addition to provide an illustrated guide to this method, the computations will be useful in \S \ref{sec:exporbit} to fully determine orbital integrals over   $\mathrm{GL}_2\times_{\det}\mathrm{GL}_2$. This is applicable to count the number of isomorphism classes inside an isogeny 
class of products of elliptic curves (as products of elliptic curves). Theorems \ref{th:shalikhyper}, \ref{th:shalik-elliptic} and \ref{th:mixedshalika} give explicit Shalika germ expansions for spherical functions over hyperbolix, elliptic, and mixed orbits respectively.

When we are dealing with non-elliptic orbit, the computation in $\mathrm{GSp}_4$ can be reduced to an integral over the subgroup explicitly, using several instances of parabolic descent. We cover 
this in \S \ref{sec:parabdescent}. In particular, given a product of $n$ nonisogenous elliptic curves we give a formula for the descent of integrals in the Langlands--Kottwitz formula which holds for 
$1-2^{-n}$ proportion of integrals. This gives an approximation of the proportion of principally polarized abelian varieties decomposing as products of elliptic curves, described in Remarks \ref{rem:idealworld} and \ref{rem:idealworld2}

Finally, \S \ref{sec:equivalued} is dedicated explore elliptic orbits and give preliminary results. We compute orbital integrals of regular semisimple elements under a strong assumption of being \emph{equivalued} so that we can exploit the notion of  \emph{purity} of the corresponding  Springer fiber and apply results of \cite{gkmpure}. Although this fails to give a complete answer for elliptic orbits, this highlights the lack of straightforward descent formula such as the one in \S \ref{sec:parabdescent}.

We hope to achieve similar results to non-equivalued elements in the future and hence obtain an entirely explicit formula for the number of isomorphism classes in an isogeny class of a product of two elliptic curves (isogenies as principally polarized abelian varieties) in all possible cases, as well as answer Achter's question. \\

There are three appendices to this paper. The first one contains all the combinatorial facts about counting points on infinite regular trees, which we use heavily in \S \ref{sec:kottmethod} and \S \ref{sec:exporbit}. The second appendix 
is handling computations of the number of orbits within a stable orbit. Finally, the last appendix shows an example where Kottwitz's technique
 is applied to relative orbital integrals. Relative orbital integrals are growing in importance with the rise of the relative Langlands programme and the relative trace formula. 
 Such integrals are associated with \emph{spherical pairs} $(H,G)$ where $H$ is a subgroup of $G$, and correspond to the integral of a test function over the $H$-orbit of an element of $G$. We compute integrals 
 in the case of the pair  $(\mathrm{GL}_2,\mathrm{GL}_2\times\mathrm{GL}_2)$ where the first group embeds diagonally. 
 Similar methods could be used to compute integrals for other spherical pairs such as the Friedberg--Jacquet pair $(\mathrm{GL}_2\times \mathrm{GL}_2,\mathrm{GL}_4)$ or the Flicker--Rallis 
 pair $(\mathrm{GL}_2, R_{E/F}\mathrm{GL}_2)$ where $E/F$ is a quadratic extension.  

\subsection*{Acknowledgements} We thank Jeff Achter for presenting the problem, as well as Julia Gordon,  
and Niranjan Ramachandran for their help, discussions, and comments. We want to thank Cheng-Chiang Tsai in particular for many meetings
where he patiently taught the author the theory behind the computations in \S \ref{sec:equivalued}.
	``\section{Preliminaries}\label{sec:prelim}

\subsection{A distinguished subgroup of the similitude group}

Let $F$ be a field and $(V , \langle\ ,\ \rangle)$ be a nondegenerate symplectic space of dimension $2n$ over $F$.

Assume that $(V,\langle\ ,\ \rangle)$ splits as a direct sum $(W_1,\ \langle\ ,\ \rangle_1)\oplus (W_1,\ \langle\ ,\ \rangle_2)$. Write $\lambda : \mathrm{GSp}(V)\rightarrow F^\times$ for the multiplier of similitudes.

Define $H = \mathrm{GSp}(W_1)\times_{\lambda}\mathrm{GSp}(W_2)$, defined as the pullback 
\[\begin{tikzcd}
H \ar[d]\ar[r]& \mathrm{GSp}(W_1)\ar[d,"\lambda"]\\
\mathrm{GSp}(W_2)\ar[r,"\lambda"]& F^\times
\end{tikzcd}.\]
 
 Then we have $H = \mathrm{GSp}(W_1)\times\mathrm{GSp}(W_2)\cap \mathrm{GSp}(V)$ inside $\mathrm{GL}(V)$. We are interested in a particular case of such a pullback, when $W_1,W_2$ are quadratic spaces. \\

Fix a Witt basis $\{e_{1}, \cdots, e_n,f_1,\dots, f_n\}$ such that $\langle e_i, f_i \rangle = \delta_{i,j}$. For $1\leq i\leq n$, write $(V_i, \langle\ ,\ \rangle_i)$ the two-dimensional symplectic space generated by $e_i, f_i$ and $\langle\ ,\ \rangle_i$ is defined by restriction of $\langle\ ,\ \rangle$. We have $V = \bigoplus_{i=1}^nV_i$.  This yields an embedding 
\[\mathrm{Sp}(V_1, \langle\ ,\ \rangle_1)\times \cdots \times \mathrm{Sp}(V_n, \langle\ ,\ \rangle_n)\xhookrightarrow{} \mathrm{Sp}(V,\langle\ ,\ \rangle).\]

If $g\in \mathrm{GSp}(V, \langle\ ,\ \rangle)$ (resp. $\mathrm{GSp}_{i}(V_i, \langle\ ,\ \rangle_i)$), denote by $\lambda(g)$ (resp. $\lambda_i(g)$) its multiplier. 
Note that for $g\in \mathrm{GSp}_{i}(V_i, \langle\ ,\ \rangle_i) = \mathrm{GL}(V_i)$, we have $\lambda_i(g)  =\det(g)$.\\

Given $g = (g_1,\cdots, g_n) \in \mathrm{GSp}(V_1, \langle\ ,\ \rangle_1)\times \cdots \times \mathrm{GSp}(V_n, \langle\ ,\ \rangle_n)$, the action of $g$ on $V = \bigoplus_{i=1}^n$ belongs to $\mathrm{GSp}(V,\langle\ ,\ \rangle)$ whenever $\lambda(g_i) = \lambda(g_j)$ for all $1\leq i,j\leq n$.\\

Define $\mathbf{G} := \left(\mathrm{GL}_2\times\cdots\times \mathrm{GL}_2\right)^\circ$ so that 
\[G = \mathbf{G}(F) = \left\lbrace(g_1,\cdots, g_n)\in \prod_{i=1}^n \mathrm{GL}(V_i) : \ \det(g_i)= \det(g_j)\ \forall i,j\right\rbrace .\]

\[\begin{tikzcd}
    \mathrm{Sp}(V,\langle\ ,\ \rangle)\arrow[hookrightarrow]{r}& \mathrm{GSp}(V,\langle\ ,\ \rangle)\arrow[hookrightarrow]{r}&\mathrm{GL}(V)\\
     \prod_{i=1}^n\mathrm{SL}(V_i)\arrow[hookrightarrow]{u}\arrow[hookrightarrow]{r}& G\arrow[hookrightarrow]{u}\arrow[hookrightarrow]{r}& \prod_{i=1}^n\mathrm{GL}(V_i)\arrow[hookrightarrow]{u}\\
\end{tikzcd}\]
Working in  the Witt basis, we can write the inclusion of $G$ into $\mathrm{GSp}(V,\langle\ ,\ \rangle)= \mathrm{GSp}_{2n}$ as

\[\left(\begin{pmatrix}
    a_1&b_1\\ c_1&d_1
\end{pmatrix}, \cdots, \begin{pmatrix}
    a_n&b_n\\ c_n&d_n
\end{pmatrix}\right)\mapsto \left(\begin{array}{ccc|ccc}
a_1&      &   &b_1&      &   \\
   &\ddots&   &   &\ddots&   \\
   &      &a_n&   &      &b_n \\\hline
   c_1&      &   &d_1&      &   \\
   &\ddots&   &   &\ddots&   \\
   &      &c_n&   &      &d_n \\
\end{array}\right).\]

\begin{rem}
We will mostly focus on the case $n=2$, and will therefore use the more explicit notation
$\mathbf{G} = \mathrm{GL}_2\times_{\det}\mathrm{GL}_2\subset \mathrm{GSp}_4$.
\end{rem}

On the level of Lie algebras, $\mathrm{Lie}(G)$ is the subalgebra of $\mathfrak{gsp}_{2n}$ spanned by the long roots. In particular, making a choice of positive roots (see \S \ref{sec:roots} for a description of the root system), the only face of the positive Weyl chamber perpendicular to all the long roots is the trivial one. Consequently, we have the following. 

\begin{prop}\label{prop:nonpara}
The group $\mathbf{G}(F)$  is not contained in any parabolic subgroup of $\mathrm{GSp}_{2n}(F)$.
\end{prop}

Another important fact is that the rank of $\mathrm{GSp}_{2n}$ is $n+1$, which is equal to the rank of $\mathbf{G}$, proving the following.
\begin{lemma}\label{lem:maxismax}
Maximal tori of $G$ are also maximal in $\mathrm{GSp}_{2n}$.
\end{lemma}
\subsection{Statement of the problem.}

Now and for the rest of this paper, the field $F$ is a local nonarchimedean field unless stated otherwise. Let $\mathcal{O}_F$ denote the ring of integers 
of $F$ and $\varpi\in F$ a choice of uniformizer.

Given an algebraic group $\mathbf{H}$ defined over $F$ and  $\gamma\in \mathbf{H}(F)$, we recall the following terms:

\begin{itemize}
\item $\mathbf{H}_\gamma$ is the algebraic group corresponding to the centralizer of $\gamma$.
\item $\gamma$ is \emph{regular semisimple} if its centralizer $\mathbf{H}_\gamma(F)$ is the group of points of a maximal torus of $\mathbf{H}(F)$.
\item $\gamma$ is \emph{unipotent} if $\gamma-1$ is nilpotent. 
\item $\gamma$ is \emph{hyperbolic} if $\mathbf{H}_\gamma$ is a split torus.
\item $\gamma$ is \emph{elliptic} if $\mathbf{H}_\gamma(F)$ is compact modulo center.
\end{itemize}

In appendix \ref{sec:storbits} we introduce the notion of orbits and stable orbits of $\gamma$. Simply put, two 
elements are in the same stable orbit if they are conjugate by a group element (possibly defined over an extension of $F$); if this element is defined over $F$ then 
we say that those elements are in the same (ordinary/rational) orbit.

\begin{rem} Our definition of ``stably conjugate'' is sometimes different, and our definition corresponds to the term ``$\bar{F}$-conjugate'' instead. However, the two notions are  
equal when the derived subgroup of $\mathbf{H}$ is simply connected, which is always the case in this article. 
\end{rem}

As algebraic groups, the stable orbit can be identified with $(\mathbf{H}_\gamma\backslash\mathbf{H})(F)$, whereas the ordinary orbit is 
identified with the subset $\mathbf{H}_\gamma(F)\backslash\mathbf{H}(F)$.

Given $\gamma\in \mathbf{H}(F)$, and a function $f\in\mathcal{C}_c^\infty(\mathbf{H}(F))$, we define orbital integrals and \emph{stable} orbital integrals, as 

\[\mathrm{Orb}_{\mathbf{H}}(\gamma,f):= \int_{\mathbf{H}_\gamma(F)\backslash \mathbf{H}(F)} f(g^{-1}\gamma g)\ \mathrm{d}g/\mathrm{d}g_{\gamma}\]
and
\[\mathrm{Orb}^{\mathrm{st}}_{\mathbf{H}}(\gamma,f):= \int_{(\mathbf{H}\backslash \mathbf{H})(F)} f(g^{-1}\gamma g)\ \mathrm{d}g/\mathrm{d}g_{\gamma}\]
respectively. We may omit the subscript later, when there is no ambiguity in the group.  Such integrals are dependent of the choice of Haar measures on $\mathbf{H}(F)$ and $\mathbf{H}_\gamma(F)$. 

However, we will endeavor to pick \emph{compatible measures}, in the sense that if $\gamma_1,\dots, \gamma_r$ are representatives of 
the different orbits within the stable orbit of $\gamma$, we want 
\[\mathrm{Orb}^{\mathrm{st}}_{\mathbf{H}}(\gamma, f) = \sum_{i=1}^r \mathrm{Orb}_{\mathbf{H}}(\gamma_i,f).\] 
\begin{lemma}
Given $\gamma\in \mathbf{G}(F)$ a regular semisimple element, it is also a regular semisimple of $\mathrm{Gsp}_4$ and 
its centralizer with respect to $G$ and $\mathrm{GSp}_{2n}$ are equal. 
Moreover, the number of $\mathbf{G}$-orbits within its $\mathbf{G}$-stable orbits is always equal to the number of $\mathrm{GSp}_{2n}$ orbits 
within its $\mathrm{GSp}_{2n}$-stable orbit. When $n=2$, such a number can be found in Proposition \ref{prop:orbitsgl22}.
\end{lemma}
\begin{proof}
The first part comes from Lemma \ref{lem:maxismax}. Indeed, 
the centralizer $G_\gamma$ inside $G$ is also a maximal torus of $\mathrm{GSp}_{2n}(F)$
and therefore the two centralizers are equal. 

Since $H^1(F, \mathrm{GSp}_{2n}(\bar{F})) \cong H^{1}(F,\mathbf{G}(\bar{F})) = 1$, then the number 
of orbits within a stable orbit for either group is equal to $|H^1(F, \mathbf{G}_\gamma(\bar{F}))|$.
\end{proof}

\begin{problem}\label{prob:main}
Let $\gamma\in \mathbf{G}(F)$ be regular semisimple element and let $f$ the characteristic function of the maximal compact $\mathrm{GSp}_{2n}(\mathcal{O}_F)$. Fix Haar measures on 
$G=\mathbf{G}(F), G_\gamma=\mathbf{G}_\gamma(F)$, and $\mathrm{GSp}_{2n}(F)$, find $R(\gamma), R^{\mathrm{st}}(\gamma)$ so that 
\[\int_{G_\gamma\backslash \mathrm{GSp}_{2n}(F)} f(h^{-1}\gamma h)\ \mathrm{d}h/\mathrm{d}g_\gamma = R(\gamma)\int_{G_\gamma\backslash G} f(g^{-1}\gamma g)\ \mathrm{d}g/\mathrm{d}g_\gamma ,\]
and 
\[\int_{(\mathbf{G}_\gamma\backslash \mathrm{GSp}_{2n})(F)} f(h^{-1}\gamma h)\ \mathrm{d}h/\mathrm{d}g_\gamma = R^{\mathrm{st}}(\gamma)\int_{(\mathbf{G}_\gamma\backslash \mathbf{G})(F)} f(g^{-1}\gamma g)\ \mathrm{d}g/\mathrm{d}g_\gamma .\]

\end{problem}
\begin{problem}\label{prob:spheric}
More generally, what happens when we replace $f$ by an arbitrary bi-$\mathrm{GSp}_{2n}(\mathcal{O}_F)$ invariant function on $\mathrm{GSp}_{2n}(F)$.
\end{problem}

\subsection{Relation to elliptic curves.}
The work of Gekeler \cite{gekeler}, followed by \cite{ag-gl2} and \cite{ag-gsp} aim to 
give concrete formulas to count points on Shimura varieties corresponding to isogeny classes
of principally polarized abelian varieties over some finite field $\mathbb{F}_q$.

The formulas in \cite{ag-gl2,ag-gsp} are obtained by changing the measure on the Langlands--Kottwitz formula below.

\begin{theorem}[Langlands--Kottwitz formula, {\cite{kottformula}}]
Let $[X,\lambda]$ be a principally polarized abelian variety of dimension $g$ defined over a finite field $\mathbb{F}_q$. We have 
\[\sum_{[Y,\mu]\sim[X,\lambda]}\frac{1}{|\mathrm{Aut([Y,\mu])|}} = \mathrm{vol}(\mathbf{T}(\mathbb{A}_f)/\mathbf{T}(\mathbb{Q}))\mathrm{TO}(\gamma, f_p)\prod_{\ell\neq p}\mathrm{Orb}_{\mathrm{GSp}_{2g}}(\gamma, f_\ell)\]
where: 
\begin{itemize}
\item $\gamma$ denotes the Frobenius element of $X$, 
\item $\mathbf{T}$ is the centralizer of $\gamma$ inside $\mathrm{GSp}_{2g}$,
\item $p=\mathrm{char}(\mathbb{F}_q)$,
\item $f_\ell = \mathds{1}_{\mathrm{GSp}_{2g}(\mathbb{Z}_\ell)}$ if $\ell\neq p$,
\item $\mathrm{TO}(\gamma, f_p)$ is a twisted orbital integral of a more complicated function that we do not define here.  
\end{itemize}
\end{theorem}

The formula above still works when looking at different types of Shimura varieties. Specifically, Achter asked the following question: 

\begin{problem}\label{prob:absurf}
Given $E_1,E_2$ elliptic curves defined over $\mathbb{F}_q$, what  proportion of the isomorphism classes of abelian surfaces 
isogenous to $E_1\times E_2$ as principally polarized abelian varieties, are also isogenous as products of elliptic curves?\\
\end{problem}

When $q$ is prime or when $E_1\times E_2$ is ordinary, one can 
replace the twisted orbital integral in the Langlands--Kottwitz formula by an orbital integral of an explicit spherical function. 
Therefore, answering Problems \ref{prob:main} and \ref{prob:spheric} will also give an answer to Problem \ref{prob:absurf}.

More generally, one can be interested in estimating the proportion of points on a Shimura variety of a certain type. This would correspond to comparison of 
orbital integrals over various groups.

In \cite{milnecount}, the author counts the number of extensions of arbitrary abelian varieties. Although this does not answer
our problem, it will be interesting to compare the two results. 

\begin{theorem}[{\cite[Theorem 3, p.77]{milnecount}}] \label{th:milne}
If $A$ and $B$ are abelian varieties defined over a finite field $\mathbb{F}_q$, then 
\[\left|\mathrm{Ext}^1_{\mathbb{F}_q}(A,B)\right| = \frac{q^{d(A)d(B)}}{\left|\det(\langle \alpha_i,\beta_j \rangle)\right|}\prod_{a_i\neq b_j}\left(1-\frac{a_i}{b_j}\right),\]
where
\begin{itemize}
\item $d(A), d(B)$ are the dimensions of $A,B$ respectively, 
\item $(a_i)_{1\leq i\leq 2d(A)}, (b_j)_{1\leq j\leq 2d(B)}$ are the eigenvalues of the Frobenius 
elements of $A$ and $B$ respectively, 
\item $(\alpha_{i})_{1\leq i\leq r} : A\rightarrow B$ and $(\beta_{i})_{1\leq i\leq r}: B\rightarrow A$ 
are bases for $\mathrm{Hom}_{\mathbb{F}_q}(A,B)$ and $\mathrm{Hom}_{\mathbb{F}_q}(B,A)$,
\item $\langle \alpha_i,\beta_j\rangle$ is the trace of the endomorphism $\beta_j\alpha_i$.
\end{itemize}
\end{theorem}

In particular, when $E_1,E_2$ are nonisogenous elliptic curves defined over $\mathbb{F}_q$, 
then $\mathrm{Hom}_{\mathbb{F}_q}(E_1,E_2)=0$ and $a_1a_2 = b_1b_2 = q$ but $a_1+a_2 \neq b_1+b_2$ so 
eigenvalues of their respective Frobenius elements are distinct.

Therefore,  
\begin{align*}
\left|\mathrm{Ext}^1_{\mathbb{F}_q}(E_1,E_2)\right| 
&= q\left(1- \frac{a_1}{b_1}\right)\left(1- \frac{a_1}{b_2}\right)\left(1- \frac{a_2}{b_1}\right)\left(1- \frac{a_2}{b_2}\right) \\ 
&= \left(1-\frac{a_1}{b_1}\right)\left(1-\frac{b_1}{a_1}\right)\left(1-\frac{a_1b_1}{q}\right)\left(1-\frac{q}{a_1b_1}\right)\\ 
&= \frac{\left(b_1-a_1\right)^2\left(q-a_1b_1\right)^2}{qa_1^2b_1^2}.\end{align*}
is the number of abelian surfaces arising as extensions of $E_1$ by $E_2$ and therefore isogenous to $E_1\times E_2$. However, not every surface isogenous to $E_1\times E_2$ is an extension of $E_1$ by $E_2$.

\begin{rem}
The problem of decomposing an abelian variety into a product of elliptic curves boils down to decomposing a lattice in a CM-extension of $\mathbb{Q}$, which, according to class number considerations, is not unique. Therefore, a given abelian variety may be decomposed in many ways as a product of elliptic curves, even when fixing a polarization.
\end{rem}

\subsection{Approaches to the problem.}

A first approach is to use a method of ``parabolic descent'' for orbital integrals. 

Indeed, when $n=2$ then $G$ is the centralizer of the element $\pmat{1&&&\\&-1&&\\&&1&\\&&&-1}$ in $\mathrm{GSp}_{4}(F)$ but it is, however, not
a Levi subgroup. Moreover, Proposition \ref{prop:nonpara} tells us that $G$ is not contained in any parabolic subgroup. 

Parabolic descent can be used when $\gamma$ is not elliptic, as we will see in \S \ref{sec:parabdescent}. In that case,
the centralizer $G_\gamma$ is contained in a common Levi subgroup of $G$ and $\mathrm{GSp}_4$ so we descend both integrals to integrals over this subgroup. \\

Another appealing method is to use the fundamental lemma for endoscopy, but again, $\mathbf{G}$ is not an endoscopic group for $\mathrm{GSp}_{2n}$. In the case $n=2$, the group $\mathrm{GSp}_4$ is self-dual, hence  
$\mathbf{G}$ is \emph{dual} to an endoscopic group, so one could hope for a nice descent formula.

As mentioned in Appendix \ref{sec:storbits}, the orbits within a stable orbit are in correspondence with the 1-cocycles in $H^1(F,G_\gamma)$. The fundamental lemma asserts that 
 
\[\mathrm{Orb}^{\mathrm{st}}_{\mathbf{H}(F)}(\gamma_H, \mathds{1}_{K}) = \Delta(\gamma_H,\gamma) \sum_{\gamma'\in H^1(F,G_\gamma)}\kappa(\gamma') \mathrm{Orb}_{\mathrm{GSp}_{2n}}(\gamma, \mathds{1}_{\mathrm{GSp}_{2n}(\mathcal{O}_F)}),\]
where $\mathbf{H}$ is an endoscopic group, and for some choice of $\gamma_H\in \mathbf{H}(F)$, a maximal compact subgroup $K\subset \mathrm{H}(F)$  and some specific functions $\kappa$, $\Delta$.

If all $\kappa(\gamma')$ or $H^1(F,G_\gamma)$ were trivial, then we would obtain an equality of stable orbital integrals, but in those cases, we may use parabolic descent as mentioned above. We are still left to 
study the case of elliptic elements. \\

Orbital integrals on the distinguished subgroup can be evaluated explicitly using a method of Kottwitz, and we will do so in \ref{sec:kottmethod}. Such orbital integrals and their corresponding Shalika germs arise in the form of polynomials in $\mathbb{Z}[q^{-1}]$, but as shown in \cite{halesgsp6}, when $n\geq 3$ then subregular Shalika germs on $\mathrm{GSp}_{2n}$ can be 
interpreted as a point counts on hyperelliptic curves that we know not to be expressible as polynomials. This is another reason for which we will restrict our problem to $\mathrm{GSp}_4$ in this paper.

Shalika germs for $\mathrm{GSp}_4$ are computed in \cite{halesgsp4} but are stated in a different way to our explicit computations, which makes
direct comparison hard. We will instead study \emph{equivalued} orbital integrals in $\mathrm{GSp}_4$ explicitly using the study of Springer fibers from \cite{gkmpure}.
This will occupy \S \ref{sec:equivalued}.

\section{Kottwitz's approach and integrals over \texorpdfstring{$\mathrm{SL}_2$}{SL2}.}
\label{sec:kottmethod}
\subsection{Orbital integrals as point count on the building}

Let $\mathbf{G}$ be an algebraic group defined over $F$ acting on itself via conjugation. Let $G= \mathbf{G}(F)$. The Hecke algebra $\mathcal{H}(G)$ is the ring (under convolution product) of locally constant compactly-supported complex valued function on $G$

Letting $K$ be a maximal compact open subgroup of $G$, we consider the \emph{spherical} algebra $\mathcal{H}_K(G)= e_K \mathcal{H}(G)e_K$ of bi-$K$-invariant functions, where $e_K = \mathrm{vol}(K)^{-1}\mathds{1}_{K}$ is an averaging operator. These averaging operators, varying the compact-open subgroup, give $\mathcal{H}(G)$ the structure of an idempotented algebra.\\

Let $\mathcal{B}(\mathbf{G},F)$ be the (reduced) building associated to $G$. It comes 
with an action of $G$ and maximal compact subgroups of $G$ correspond to stabilizers of vertices of $\mathcal{B}(\mathbf{G},F)$. The group $G$ acts transitively on the pairs $(\mathcal{C},\mathcal{A})$, where $\mathcal{A}$ is an apartment containing the chamber $\mathcal{C}$, however note that $G$ does not necessarily act transitively on vertices, but one can color vertices so that the action of $G$ is color-preserving and transitive on vertices of a given color.



Let $\gamma \in G$. We fix a Haar measure $\mathrm{d}\dot{g}= \frac{\mathrm{d}g}{\mathrm{d}g_\gamma}$ on $\mathrm{Orb}_G(\gamma)\cong G_\gamma \backslash G$. Pick a vertex $x\in \mathcal{B}(\mathbf{G},F)$ fixed by  $K$.\\

Let $f = \mathds{1}_{KaK}\in\mathcal{H}_K(G)$ where $a\in G$. Then the map $g\mapsto g^{-1}\gamma g$ is right-$K$-invariant hence
\begin{align}\mathrm{Orb}_{\mathbf{G}}(\gamma,  \mathds{1}_{KaK}) &= \int_{G_\gamma \backslash G} \mathds{1}_{KaK}(g^{-1}\gamma g)\ \mathrm{d}\dot{g}\\
&= \mathrm{vol}_{\mathrm{d}g}(K) \int_{G_\gamma \backslash G/K} \mathds{1}_{KaK}(g^{-1}\gamma g)\ \mathrm{d}\dot{g}\\ 
&\label{eq:orbitsum}=\sum_{\underset{Kg^{-1}\gamma g K = KaK}{g\in G_\gamma \backslash G/K} }\frac{\mathrm{vol}_{\mathrm{d}g}(K)}{\mathrm{vol}_{\mathrm{d}g_\gamma}(G_\gamma\cap gKg^{-1})}.\end{align}
In particular, when $a=1$ we have 
\[\mathrm{Orb}_{\mathbf{G}}(\gamma,  \mathds{1}_{K})=\sum_{\underset{\gamma y = y}{y\in G_\gamma\backslash \mathrm{Orb}_G(x)}}\frac{\mathrm{vol}_{\mathrm{d}g}(K)}{\mathrm{vol}_{\mathrm{d}g_\gamma}(\mathrm{Stab}_{G_\gamma}(y))}.\]
Here, we sum over the $G$-orbit of $x\in \mathcal{B}(\mathbf{G},F)$, up to the action of the centralizer $G_\gamma$, which explains why one usually talks about orbital integrals as a point count on buildings. It bears emphasizing, however, that the point count is not on the whole building, but rather on the 
orbit of a vertex, which is called an \emph{affine Springer fiber}.

\subsection{Fixed points on the tree of \texorpdfstring{$\mathrm{SL}_2$}{SL2}}
Let $X = \mathcal{B}(\mathrm{GL}_2,F)$ be the building associated to our groups. 
The set of vertices of ${X}$ is $F^\times \backslash \mathrm{GL}_2(F)/(\mathrm{GL}_2(\mathcal{O}_F))$ which is the set of lattices inside $F^2$ up to scaling.  Two vertices $x\neq x'$ are neighbors if there are $\Lambda\in x, \Lambda'\in x'$ so that 
\[\varpi\Lambda'\subsetneq \Lambda \subsetneq \Lambda'.\]

The building is an infinite $q+1$-regular tree, the neigbours of $\Lambda$ corresponding to a choice of a line in $\Lambda/\varpi\Lambda\cong \mathbb{F}_q$ hence an element of $\mathbb{P}^1(\mathbb{F}_q)$. \\

\textbf{Distance.} Let $x,x'$ be two vertices. Pick $\Lambda\in x, \Lambda'\in x'$ so that $\Lambda\subset \Lambda'$. Write $\Lambda'/\Lambda \cong \mathcal{O}_F/\varpi^{e_1}\mathcal{O}_F\oplus \mathcal{O}_F/\varpi^{e_2}\mathcal{O}_F$. The distance $\mathrm{d}(x,x')$ is defined by  $|e_1-e_2|$.\\

\begin{rem}
 Let $K = \mathrm{GL}_2(\mathcal{O}_F)$ and $T$ the diagonal torus. By Cartan decomposition we have 
\[F^\times\backslash\mathrm{GL}_2(F) = F^\times\backslash KTK\cong \mathbb{Z}_{\geq 0} \]
via the map $K\begin{pmatrix}
    \varpi^a&\\ &\varpi^b
\end{pmatrix}\mapsto |a-b|$.

Writing $x = gx_0$ and $x' = g'x_0$ then Kottwitz define the invariant function 
\[\mathrm{inv}: \mathrm{GL}_2(F)/K\times \mathrm{GL}_2(F)/K\mapsto F^\times\backslash K\mathrm{GL}_2(F)K\cong \mathbb{Z},\]
by $inv(gK, g'K) = Kg^{-1}gK$. We can see that 
\[\rd(x,x') = \mathrm{inv}(gK,g'K).\]
\end{rem}

Let $x_i$ be the vertex corresponding to $\Lambda_i = \varpi^i\mathcal{O}_F\oplus \mathcal{O}_F $ where $i\in\mathbb{Z}$. We call $x_0$ the fundamental point, and  the  geodesic containing all $x_i$'s is the fundamental apartment $\mathcal{A}_0$.  Clearly we have $\mathcal{A}_0  = Tx_0$. Also note that $x_{-i}$ can be represented by $\mathcal{O}_F\oplus \varpi^i\mathcal{O}_F$

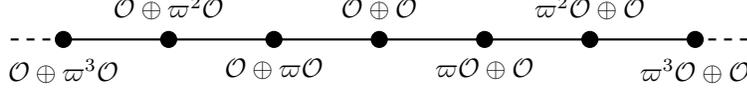
\[\begin{tikzpicture}
\begin{scope}[scale=0.7]
\draw[thick] (-6,0) -- (6,0);
\draw[thick, dashed] (-7,0)--(-6,0);
\draw[thick, dashed] (7,0)--(6,0);
\foreach \i in {-4,0,...,4}{
\node[draw,circle,inner sep=2pt,fill=black, thick] at (\i,0) {};} 
\foreach \i in {-6,-2,...,6}{
\node[draw,circle,inner sep=2pt,fill=black, thick] at (\i,0) {};} 
\node[above] at (0,0.2) {$\mathcal{O}\oplus \mathcal{O}$};
\node[above] at (4,0.2) {$\varpi^2\mathcal{O}\oplus\mathcal{O}$};
\node[above] at (-4,0.2) {$\mathcal{O}\oplus\varpi^2\mathcal{O}$};
\node[below] at (2,-0.2) {$\varpi\mathcal{O}\oplus\mathcal{O}$};
\node[below] at (6,-0.2) {$\varpi^3\mathcal{O}\oplus\mathcal{O}$};
\node[below] at (-6,-0.2) {$\mathcal{O}\oplus\varpi^3\mathcal{O}$};
\node[below] at (-2,-0.2) {$\mathcal{O}\oplus\varpi\mathcal{O}$};
\end{scope}
\end{tikzpicture}\]

The neighbors of $x_0$ are $x_1$, $x_{-1}$ and $\left\lbrace\begin{pmatrix}1&n\\ 0&1\end{pmatrix} x_1 \right\rbrace$ where $n\in \mathcal{O}_F$ spans over a set of representatives of the residue field.

For $\gamma\in \mathrm{GL}_2(F)$ we let $X^\gamma$ denote the set of fixed points of $\gamma$.

\begin{prop}\label{prop:distgx}
    Let $x\in X$ be a vertex and let $\gamma\in \mathrm{GL}_2(F)$. We have 
    \[\rd(\gamma x, x) = 2\rd(x, X^\gamma).\]
\end{prop}
\begin{proof}
    Let $v\in X^\gamma$ be the closest vertex to $x$. Write 
    \[[v_0 = x, v_1,\cdots, v_{n-1}, v_n = v]\]
    for the unique shortest path from $x$ to $v$, where $\rd(x, v) = n$. The only vertex of this path fixed by $\gamma$ is $v$ by definition of $v$. Therefore, the path  
    \[[\gamma v_0, \gamma v_1,\cdots, \gamma v_{n-1}, v_n, v_{n-1}, \cdots, v_1,v_0]\] has no repeated vector, and is the unique shortest path from $\gamma x$ to $x$. We conclude 
    \[\rd(\gamma x, x) = 2n = 2\rd(x, v) = 2\rd(x,X^\gamma).\]
\end{proof}

\begin{lemma}\label{lem:unipfixed} Let $N$ be the group of unipotent matrices of the form $\begin{pmatrix}
    1&x\\ 0&1
\end{pmatrix}$. 
    The set of vertices of the fundamental apartment form a set of representatives for the $N$-orbit of vertices of $X$. Moreover, for each $i\in \mathbb{Z}$ we have 
    \[\mathrm{Stab}_N(x_i) = \left\lbrace\begin{pmatrix}
    1&\varpi^{i}\mathcal{O}\\ 0&1
\end{pmatrix}\right\rbrace\]
\end{lemma}
\begin{proof}
Iwasawa decomposition $\mathrm{GL}_2(F) = NTK$  tells us that the vertices in the apartment corresponding to $T$ form a set of representatives for $N\backslash X$.

    Fix a basis $e,f$ so that $\Lambda_i$ (representing $x_i$) is spanned by $\varpi^i e$ and $f$. We have that $\begin{pmatrix}
    1&x\\ 0&1
\end{pmatrix} \Lambda_i$ is spanned by $\varpi^i e$ and $xe+ f$. We can deduce that $\Lambda_i$ is fixed if and only if $\nu_{\varpi}(x)\geq i$, otherwise $xe+f\notin \Lambda_i$. Therefore, we need $x\in \varpi^{i}\mathcal{O}_F$ as desired.
\end{proof}
\label{sec:fixedpoints}

\textbf{Regular semisimple elements.}
Assume that $\gamma \in \mathrm{GL}_2(\mathcal{O}_F)$ has two distinct eigenvalues $a,b\in F(\gamma)$ where $F(\gamma)$ to be the smallest field
extension containing these eigenvalues. 

Let
\[d_\gamma = \frac{\mathrm{val}_{F(\gamma)}\left(1-\frac{a}{b}\right)}{e_{F(\gamma)}}\in \frac{1}{e_{F[\gamma]}}\mathbb{Z}.\]
The role of the ramification index $e_{F(\gamma)}$ is to take into 
account that the extension of the building to a ramified extension adds vertices at each midpoint of every edges, 
and so the distance metric on $\mathcal{B}(\mathrm{GL}_2, F(\gamma))$ is half the one on the rational building.

\begin{lemma}
The set $X^\gamma$ is equal to the set of vertices of $X$ at distance at most $d_\gamma$ from 
the apartment (in $\mathcal{B}(\mathrm{GL}_2,F(\gamma))$) corresponding to the centralizer of $\gamma$.
\end{lemma}
\begin{proof}
Let us consider the extension of the building over $F(\gamma)$. We may pick a basis so that 
$\gamma=\pmat{a&0\\0&b}$. Therefore, 
 $\gamma$ fixes the vertex $v_0 = \mathcal{O}_{F(\gamma)}^{\oplus 2}$ and $G_\gamma$ corresponds 
 to the fundamental apartment. 
 
 Let $v$ be a vertex in the building over $F(\gamma)$. 
 By transitivity of the action of $\mathrm{GL}_{2}(F(\gamma))$, we can write $v=gv_0$. 
 Therefore, $\gamma$ fixes $v$ if and only if $g^{-1}\gamma g$ fixes $v_0$, in other words, 
 $g^{-1}\gamma g\in \mathrm{GL}_2(\mathcal{O}_{F(\gamma)})$.

 Write $K = \mathrm{GL}_2(\mathcal{O}_{F(\gamma)})$.
 We can write $\mathrm{GL}_2(F(\gamma)) = TNK$,
 where $T$ is the centralizer of $\gamma$ and $TN$ is the standard Borel of upper-triangular matrices. 

Decompose $g = tnk$. We need to check when $n^{-1}\underset{ = \gamma}{\underbrace{t^{-1}\gamma t}}n \in K$ 
Since $\gamma\in K$, this is equivalent to $n^{-1}\gamma n \gamma^{-1}\in K$, and writing $n = \pmat{1&x\\0&1}$
this becomes $\pmat{1&(ab^{-1}-1)x\\ 0&1}\in K$, or equivalently, $x\in v(x)\geq v(1-ab^{-1})$, where $v$ is the valuation map 
on $F(\gamma)$.
 
  We can conclude by noting that this is equivalent to $v$ being at distance at most $v(1-a/b)$ from the fundamental apartment, since $n$ fixes $v_i$ if $i\geq -v(1-a/b)$, and 
  the distance from $ nv_0$ to $v_{i}$ is equal to the distance of $v_0$ to $v_i$ (and $T$ just 
  translates vertices on the fundamental apartment).
\end{proof}
We will call the set described in the Lemma above a \emph{tube}. Thusly, to compute integrals 
we need to count $\mathrm{Gal}(F(\gamma)/F)$-fixed vertices in the tube of radius $d_\gamma$ around 
the apartment corresponding to $G_\gamma$.

\subsection{The case of  \texorpdfstring{$\mathrm{SL}_2$}{SL2}}\label{sec:sl2}
In \cite{kott_bible}, Kottwitz applies the point count technique to $\mathrm{GL}_2$. Although 
very similar, we give a detailed computation of the $\mathrm{SL}_2$ due to a major difference. That is, 
rational orbits may no longer equal the corresponding stable orbits. 

\subsubsection{Setup}

Unlike $\mathrm{GL}_2(F)$, the group $\mathrm{SL}_2(F)$  no longer acts transitively on
 vertices of the building of $\mathrm{GL}_2$. 
 However, by coloring our tree with alternating colors, 
 the action of $\mathrm{SL}_2$ is color-preserving and transitive
  on the set of vertices of the same color. 
\begin{figure} 
\centering
\begin{tikzpicture}[scale=0.8]

\draw[very thick] (0,0) -- (0:3.33)  -- ++(-60:2) -- ++(0:1.33)   -- ++(-60:0.8)  -- ++(0:0.5) -- ++(-60:0.32);
\draw[very thick] (0,0)    -- (120:3.33)   -- ++(180:2)   -- ++(120:1.33)   -- ++(180:0.8) -- ++(120:0.5) -- ++(180:0.32);
\draw[very thick] (0,0) -- (0:3.33);
\draw[very thick] (0,0) -- (120:3.33);
\draw[very thick] (0,0) -- (240:3.33);
\node[draw,circle,inner sep=3pt,fill=black, thick] at (0,0) {};
\foreach \i in {0,120,...,240}{
    \begin{scope}[shift=(\i:3.33),rotate=\i]
        \draw[very thick] (0,0) -- (60:2);
        \draw[very thick] (0,0) -- (-60:2);
        \node[draw,circle,inner sep=3pt,fill=white, thick] at (0,0) {};
        \foreach \j in {-60,60,...,60}{
            \begin{scope}[shift=(\j:2),rotate=\j]
                \draw[very thick] (0,0) -- (60:1.33);
                \draw[very thick] (0,0) -- (-60:1.33);
                \node[draw,circle,inner sep=3pt,fill=black, thick] at (0,0) {};
                \foreach \k in {-60,60,...,60}{
                    \begin{scope}[shift=(\k:1.33),rotate=\k]
                        \draw[very thick] (0,0) -- (60:0.8);
                        \draw[very thick] (0,0) -- (-60:0.8);
                        \node[draw,circle,inner sep=3pt,fill=white, thick] at (0,0) {};
                        \foreach \l in {-60,60,...,60}{
                            \begin{scope}[shift=(\l:0.8),rotate=\l]
                                \draw[very thick] (0,0) -- (60:0.5);
                                \draw[very thick] (0,0) -- (-60:0.5);
                                \node[draw,circle,inner sep=3pt,fill=black, thick] at (0,0) {};
                                \foreach \h in {-60,60,...,60}{
                                    \begin{scope}[shift=(\h:0.5),rotate=\h]
                                        \draw[dashed,very thick] (0,0) -- (60:0.32);
                                        \draw[dashed, very thick] (0,0) -- (-60:0.32);
                                    \end{scope}
                                }
                            \end{scope}
                        }
                    \end{scope}
                }
            \end{scope}
        }
    \end{scope}
}


\end{tikzpicture}
\caption{Building of $\mathrm{GL}_2(\mathbb{Q}_2)$ with coloring, preserved by the action of $\mathrm{SL}_2(\mathbb{Q}_2)$}
\end{figure}
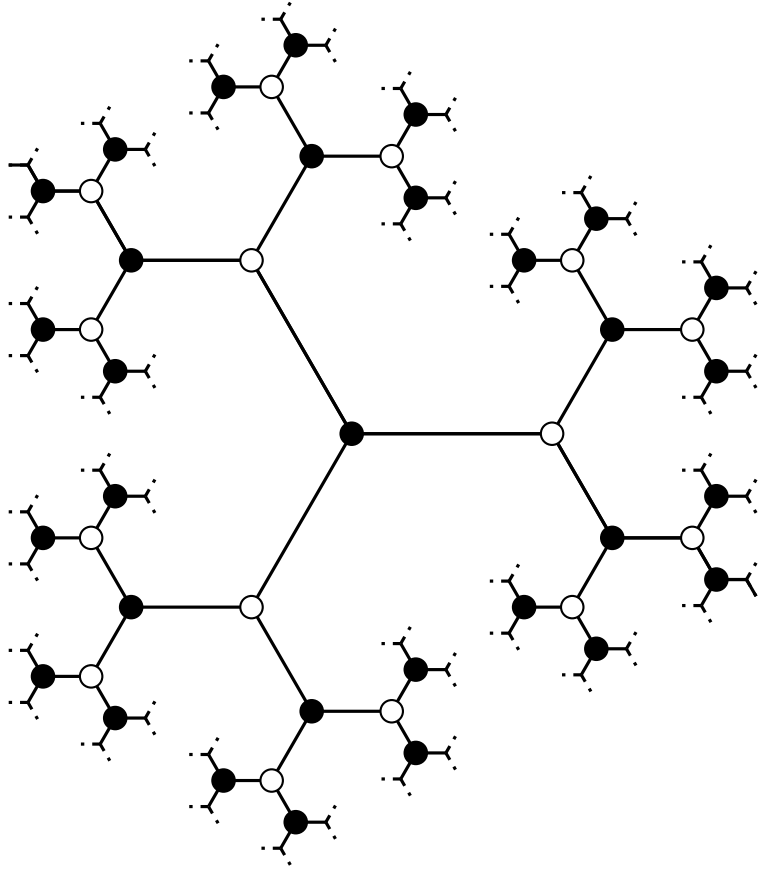

Let $K=\mathrm{SL}_2(\mathcal{O}_F)$. We will want to integrate a spherical function, i.e. bi-$K$-invariant function. Such functions are finite sums of functions of the form $\mathds{1}_{KaK}$ where $a\in \mathrm{SL}_2(F)$ is a diagonal matrix. Therefore, it is sufficient to compute the integrals
\[O_n(\gamma)  :=\mathrm{Orb}_{\mathrm{SL}_2}(\gamma, \mathds{1}_{K\mathrm{diag}(\varpi^n,\varpi^{-n})K}) ,\text{and } O_n^{\mathrm{st}}:=\mathrm{Orb}^{\mathrm{st}}_{\mathrm{SL}_2}(\gamma,\mathds{1}_{K\mathrm{diag}(\varpi^n,\varpi^{-n})K})),\]
where $n\in\mathbb{Z}_{\geq0}$.

Note that in this section we will always pick a Haar measure on $\mathrm{SL}_2(F)$ giving $K$ measure $1$. The measure on $\mathrm{Stab}_{\mathrm{SL}_2(F)}(\gamma)$ will change depending on the case.

Write $t_n = \mathrm{diag}(\varpi^n,\varpi^{-n})$. We have that $g^{-1}\gamma g\in Kt_nK$ if and only if 
\[\mathrm{inv}(gK, \gamma gK) = Kg^{-1}\gamma gK = Kt_nK,\]
in other words, letting $x_g = gx_0$, this means 
\[g^{-1}\gamma g\in Kt_nK \Leftrightarrow\rd(\gamma x_g, x_g) = 2n.\]

\begin{lemma}
    We have $g^{-1}\gamma g \in Kt_nK$ if and only if 
    \[\rd(x_g,X^\gamma) = n.\] 
\end{lemma}
\begin{proof} This is an immediate application of Proposition \ref{prop:distgx}.
\end{proof}
In particular, when $n=0$ the condition is $x_g\in X^\gamma$.

We color $X$ by picking the fundamental point $x_0$ to be black, and its neighbors white. We can therefore write $X = X_{\scalebox{0.5}{\CIRCLE}}\sqcup X_{\scalebox{0.5}{\Circle}}$ where $X_{\scalebox{0.5}{\CIRCLE}}$ and  $X_{\scalebox{0.5}{\Circle}}$ are the set of vertices of black and white color respectively. Since $\mathrm{SL}_2(F)$ acts transitively on each class of vertices, we can reformulate eq \ref{eq:orbitsum} as 

\[O_n(\gamma) = \sum_{v}\frac{1}{\mathrm{meas}_{\mathrm{Stab}_{\mathrm{SL}_2(F)}(\gamma)}(v)},\]
where $v$ spans over $\mathrm{Stab}_{\mathrm{SL}_2(F)}(\gamma)-$orbits of vertices in $X_{\scalebox{0.5}{\CIRCLE}}$ such that $\rd(v, X^\gamma) = n$.

\subsubsection{Unipotent elements}

Assume that $\gamma\in K$ is unipotent. Write $B = TN$ the standard Borel where $T$ is the diagonal torus, and $N$ the group of upper-triangular unipotent matrices.

Up to conjugation by $\mathrm{SL}_2(F)$, there 5 conjugacy classes of unipotent elements, represented by elements 
\[\begin{pmatrix}1 & 0\\0&1\end{pmatrix},\ \begin{pmatrix}1 & 1\\0&1\end{pmatrix},\ \begin{pmatrix}1 & u\\0&1\end{pmatrix},\ \begin{pmatrix}1 & \varpi\\0&1\end{pmatrix},\ \begin{pmatrix}1 & u\varpi\\0&1\end{pmatrix},\]
where $u\mathcal{O}_F^\times$ is a non-square unit.

Write $\alpha:F\rightarrow N$ the isomorphism $x\mapsto \begin{pmatrix}
    1&x\\ 0&1
\end{pmatrix}$.

\begin{rem}
Note that if we were looking at $\mathrm{GL}_2(F)$-conjugacy, we would only have two conjugacy classes represented by $\alpha(0)$ and $\alpha(1)$ since if $x\in F^\times$ then 
\[\begin{pmatrix}
    1&0\\0&x
\end{pmatrix}^{-1}\alpha(1)\begin{pmatrix}
    1&0\\0&x
\end{pmatrix} = \alpha(x).\]
\end{rem}

\begin{lemma}
    We have 
    \[O_n(\alpha(0)) = \left\lbrace \begin{array}{ll}
        1 & n=0  \\
        0 & n>0 
    \end{array}\right. .\]
\end{lemma}
\begin{proof} Since $\alpha(0)$ is in the center, 
    \[O_n(\alpha(0))  = \int_{\mathrm{SL}_2(F)\backslash \mathrm{SL}_2(F)}\mathds{1}_{Kt_nK}(1)\ \mathrm{d}g/\mathrm{d}g=\mathds{1}_{Kt_nK}(1).\]
\end{proof}

We now deal with the remaining orbits. The centralizer of $\alpha(x)$ where $n\neq 0$ is $\{\pm1\}\times N\subset B$. For $x\in \{1,u,\varpi, u\varpi\}$ we will compute $O_n(\alpha(x))$ by using the Haar measure on $N\cong F$ giving $\alpha(x \mathcal{O}_F)$ measure $2$. The $2$ is explained because conjugating an element of $N\cong F$ by a diagonal element of $\mathrm{SL}_2$ multiplies it by a square, hence the orbits are of the form $x{\mathcal{Q}_F^\times}$ where ${\mathcal{Q}_F}$ is the set of squares in $\mathcal{O}_F$. We therefore want to pick the Haar measure giving $x{\mathcal{Q}_F^\times}$ measure $1$.

We get that 
\[O_n(\alpha(x)) =\sum \frac{1}{\mathrm{meas}(\mathrm{Stab}_N(v))},\]
where $v$ runs over a set of representatives of $N$-orbits in $X_{\scalebox{0.5}{\CIRCLE}}$ such that $\rd(v, X^\gamma) = n$.  As previously, we let $x_i$ where $i\in \mathbb{Z}$ denote vertices  of the fundamental apartment, where $x_i$ is represented by the lattice $\varpi^i\mathcal{O}_F\oplus \mathcal{O}_F$. We get that 
\[\mathrm{meas}(\mathrm{Stab}_N(x_i))= \mathrm{meas}(\alpha(\varpi^i\mathcal{O}_F)),\]
as well as the fact that the set $\{x_{2i}: i\in \mathbb{Z}\}$ is a set of representatives of $N$-orbits in $X_{\scalebox{0.5}{\CIRCLE}}$.
 Moreover, $\alpha(1), \alpha(u)$ will fix $x_i$ as long as $0 = \nu_\varpi(1)= \nu_\varpi(u)\geq i$. Similarly, $\alpha(\varpi)$ and $\alpha(u\varpi)$ fix all $x_i$ for $i\leq 1$.

This gives us the following. 

\begin{prop}
    With our choice of measures, we have
    \[O_0(\alpha(1)) = O_0(\alpha(u)) = \frac{1}{2}\frac{1}{1-q^{-2}},\]
    \[O_0(\alpha(\varpi)) = O_0(\alpha(u\varpi)) = \frac{1}{2}\frac{q^{-1}}{1-q^{-2}},\]
    and if $n>0$ then
    \[O_n(\alpha(1)) = O_n(\alpha(u))=\left\lbrace \begin{array}{ll}q^n/2&n\equiv0\pmod2\\0&n\equiv1\pmod2\end{array}\right. ,\]
    \[O_n(\alpha(\varpi)) = O_n(\alpha(u\varpi)) =\left\lbrace \begin{array}{ll}0&n\equiv0\pmod2\\q^n/2&n\equiv1\pmod2\end{array}\right. .\]
\end{prop}
\begin{proof}
    If $x\in \{1,u\}$ then the measure we pick on $F$ is just twice the usual Haar measure.
    \[O_0(\alpha(x)) = \sum_{2i\geq 0}\frac{1}{\mathrm{meas}(\varpi^{-2i}\mathcal{O}_F)} = \sum_{i\geq 0}\frac{1}{2q^{2i}} =\frac{1}{2}\frac{1}{1-q^{-2}}. \]

To compute $O_n$ with $n>0$ we need to find how many $x_i$'s are at distance $n$ from the set of fixed points of $\alpha(1)$, which is exactly $\{x_{i}: i\leq 0\}$. Therefore, $\rd(x_k, \{x_{i}: i\leq 0\}) = k$ if $k>0$ and $0$ else. We get that there is only one such vertex, namely $x_n$. If $n$ is odd then $x_n$ does not have the correct coloring. 

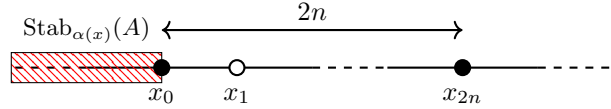
\[\begin{tikzpicture}
    \draw[thick] (0,0) -- (3,0);
    \draw[thick] (4,0) -- (6,0);
    \draw[dashed, thick] (3,0) -- (4,0);
    \draw[dashed, thick] (6,0) -- (7,0);
    \draw[dashed, thick] (0,0) -- (-1,0);
    \draw[pattern=north west lines, pattern color=red] (-1,-.2) rectangle (1,.2);
     \node[draw,circle,inner sep=2pt,fill=black, thick] at (1,0) {};
      \node[draw,circle,inner sep=2pt,fill=white, thick] at (2,0) {};
    \node[draw,circle,inner sep=2pt,fill=black, thick] at (5,0) {};
    \node[below=0.15] at (1,0) {$x_0$};
    \node[below=0.15] at (2,0) {$x_1$};
    \node[below=0.15] at (5,0) {$x_{2n}$};
    \draw[<->, thick] (1,0.5) -- (5,0.5)node[midway, above] {$2n$};
    \node at (0,0.5) {\small $\mathrm{Stab}_{\alpha(x)}(A)$};
\end{tikzpicture}\]
We get 
\[O_n(\alpha(x)) =\left\lbrace \begin{array}{ll}\frac{1}{\mathrm{meas}(\varpi^{n}\mathcal{O}_F)}&n\equiv0\pmod2\\0&n\equiv1\pmod2\end{array}\right.=\left\lbrace \begin{array}{ll}q^n/2&n\equiv0\pmod2\\0&n\equiv1\pmod2\end{array}\right..\]

    Now assume $x\in\{\varpi, u\varpi\}$. This time around, the measure is not exactly twice the Haar measure, now $\mathrm{meas}(\varpi^i \mathcal{O}) = 2q^{-(i-1)}$. We get  
     \[O_0(\alpha(x)) = \sum_{2i\geq 1}\frac{1}{\mathrm{meas}(\varpi^{-2i}\mathcal{O}_F)} = \sum_{i\geq 0}\frac{q^{-2}}{2q^{2i-1}} =\frac{1}{2}\frac{q^{-1}}{1-q^{-2}}. \]

This time the points fixed by $\alpha(x)$ are $x_i$ with $i\leq 1$. We have the following situation.
  
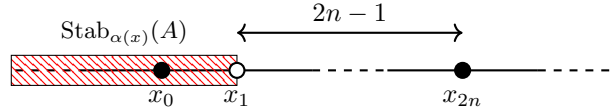
\[\begin{tikzpicture}
    \draw[thick] (0,0) -- (3,0);
    \draw[thick] (4,0) -- (6,0);
    \draw[dashed, thick] (3,0) -- (4,0);
    \draw[dashed, thick] (6,0) -- (7,0);
    \draw[dashed, thick] (0,0) -- (-1,0);
    \draw[pattern=north west lines, pattern color=red] (-1,-.2) rectangle (2,.2);
     \node[draw,circle,inner sep=2pt,fill=black, thick] at (1,0) {};
      \node[draw,circle,inner sep=2pt,fill=white, thick] at (2,0) {};
    \node[draw,circle,inner sep=2pt,fill=black, thick] at (5,0) {};
    \node[below=0.15] at (1,0) {$x_0$};
    \node[below=0.15] at (2,0) {$x_1$};
    \node[below=0.15] at (5,0) {$x_{2n}$};
    \draw[<->, thick] (2,0.5) -- (5,0.5)node[midway, above] {$2n-1$};
    \node at (0.5,0.5) {\small $\mathrm{Stab}_{\alpha(x)}(A)$};
\end{tikzpicture}\]
There is one black point at distance $n$ from the set of fixed point if $n$ is odd, and no such point when $n$ is even. If $n = 2k-1$ with $k\geq 1$, we have 
\[O_n(\alpha(x)) = \frac{1}{\mathrm{meas}(\mathrm{Stab}_N(x_{2k}))} = \frac{1}{2 q^{2k-1}} = q^{2k-1}/2.\] We have 
\[O_n(\alpha(x)) =\left\lbrace \begin{array}{ll}0&n\equiv0\pmod2\\q^n/2&n\equiv1\pmod2\end{array}\right. .\]
\end{proof}

\begin{cor}
    Decomposing the $\mathrm{GL}_2(F)$-orbit of $\alpha(1)$ into its four $\mathrm{SL}_2(F)$-orbits we get 
    \[\mathrm{Orb}_{\mathrm{GL}_2}(\alpha(1), \mathds{1}_{Kt_nK}) =\mathrm{Orb}^{\mathrm{st}}_{\mathrm{SL}_2}(\alpha(1), \mathds{1}_{Kt_nK}) \]
    where the left hand-side is computed in \cite[p.414]{kott_bible} and the right hand-side can be written as \[O_n(\alpha(1))+O_n(\alpha(u))+O_n(\alpha(\varpi))+O_n(\alpha(u\varpi)).\]
\end{cor}
\begin{proof}
    This is immediate from computing 
    \begin{align}
\mathrm{Orb}^{\mathrm{st}}_{\mathrm{SL}_2}(\alpha(1), \mathds{1}_{K})&=O_0(\alpha(1))+O_0(\alpha(u))+O_0(\alpha(\varpi))+O_0(\alpha(u\varpi))\\
&= \frac{1}{2}\frac{q^{-1}}{1-q^{-1}}+\frac{1}{2}\frac{q^{-1}}{1-q^{-1}}+\frac{1}{2}\frac{q^{-1}}{1-q^{-2}}+\frac{1}{2}\frac{q^{-1}}{1-q^{-2}}\\ 
&= \frac{1+q^{-1}}{1-q^{-2}} = \frac{1}{1-q^{-1}},
    \end{align}
    which corresponds to the $\mathrm{GL}_2$-orbital integral computed in \cite{kott_bible}.\\

    Similarly, if $n>0$ we have 
        \begin{align}
O^{\mathrm{st}}(\mathrm{SL}_2, \alpha(1), \mathds{1}_{Kt_nK})&=O_n(\alpha(1))+O_n(\alpha(u))+O_n(\alpha(\varpi))+O_n(\alpha(u\varpi))\\
&=\left\lbrace \begin{array}{ll}q^n/2+q^n/2&n\equiv 0\pmod2\\q^n/2+q^n/2&n\equiv1\pmod2\end{array}\right.\\ 
&= q^n,
    \end{align}
    which again matches the description of \cite{kott_bible}.
\end{proof}

\subsubsection{Regular semisimple elements: the hyperbolic case}
Consider first the case where $\gamma\in \mathrm{SL}_2(F)$ is a hyperbolic regular semisimple element. We know that $X^\gamma$ is a tube of width $d_\gamma$ around the apartment corresponding to the centralizer of $\gamma$. We can always take a conjugate of $\gamma$ so that this apartment is the fundamental one. In other words, we can assume that $\gamma = \begin{pmatrix}
    a&0\\0&a^{-1}
\end{pmatrix}$.\\

Under this assumption, the centralizer of $\gamma$ is the diagonal torus $T$. Identifying vertices of the fundamental apartment with $\mathbb{Z}$ where $0$ is the fundamental point, $T$ acts by translations by $2\mathbb{Z}$ since $\begin{pmatrix}
    u\varpi^n &0\\0&u^{-1}\varpi^{-n}
\end{pmatrix}x_i = x_{i+2n}$, where $u\in \mathcal{O}_F^\times$. 

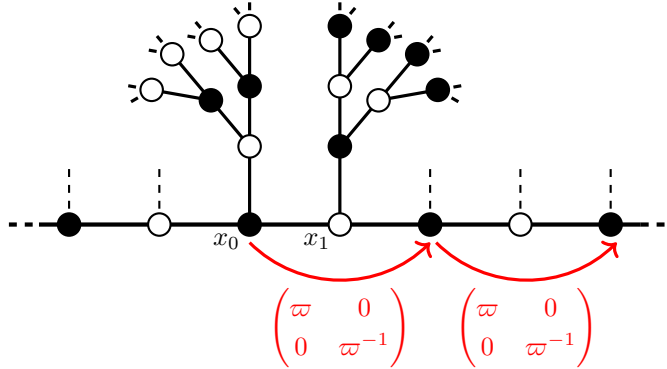
\[\begin{tikzpicture}[scale=0.8]
\draw[ultra thick] (-3.5,0) -- (6.5,0);
\draw[dashed, ultra thick] (-4,0)--(-3.5,0);
\draw[dashed, ultra thick] (6.5,0)--(7,0);

\draw[very thick] (0,0) -- ++(90:1.5);
\begin{scope}[shift=(90:1.3)]
    \draw[very thick] (0,0) -- ++(90:1);
    \draw[very thick] (0,0) -- ++(130:1);
    \begin{scope}[shift=(90:1)]
        \draw[very thick] (0,0) -- ++(90:1);
        \draw[very thick] (0,0) -- ++(130:1);
        \node[draw,circle,inner sep=3pt,fill=black, thick] at (0,0) {};
        \begin{scope}[shift=(90:1)]
            \draw[very thick, dashed] (0,0) -- ++(90:.5);
            \draw[very thick, dashed] (0,0) -- ++(130:.5);
            \node[draw,circle,inner sep=3pt,fill=white, thick] at (0,0) {};
        \end{scope}
        \begin{scope}[shift=(130:1), rotate=40]
            \draw[very thick, dashed] (0,0) -- ++(90:.5);
            \draw[very thick, dashed] (0,0) -- ++(130:.5);
            \node[draw,circle,inner sep=3pt,fill=white, thick] at (0,0) {};
        \end{scope}
    \end{scope}
    \begin{scope}[shift=(130:1), rotate=40]
        \draw[very thick] (0,0) -- ++(90:1);
        \draw[very thick] (0,0) -- ++(130:1);
        \node[draw,circle,inner sep=3pt,fill=black, thick] at (0,0) {};
        \begin{scope}[shift=(90:1)]
            \draw[very thick, dashed] (0,0) -- ++(90:.5);
            \draw[very thick, dashed] (0,0) -- ++(130:.5);
            \node[draw,circle,inner sep=3pt,fill=white, thick] at (0,0) {};
        \end{scope}
        \begin{scope}[shift=(130:1), rotate=40]
            \draw[very thick, dashed] (0,0) -- ++(90:.5);
            \draw[very thick, dashed] (0,0) -- ++(130:.5);
            \node[draw,circle,inner sep=3pt,fill=white, thick] at (0,0) {};
        \end{scope}
    \end{scope}
    \node[draw,circle,inner sep=3pt,fill=white, thick] at (0,0) {};
\end{scope}

\begin{scope}[shift=(0:1.5)]
\draw[very thick] (0,0) -- ++(90:1.5);
\begin{scope}[shift=(90:1.3)]
    \draw[very thick] (0,0) -- ++(90:1);
    \draw[very thick] (0,0) -- ++(50:1);
    \begin{scope}[shift=(90:1)]
        \draw[very thick] (0,0) -- ++(90:1);
        \draw[very thick] (0,0) -- ++(50:1);
        \node[draw,circle,inner sep=3pt,fill=white, thick] at (0,0) {};
        \begin{scope}[shift=(90:1)]
            \draw[very thick, dashed] (0,0) -- ++(90:.5);
            \draw[very thick, dashed] (0,0) -- ++(50:.5);
            \node[draw,circle,inner sep=3pt,fill=black, thick] at (0,0) {};
        \end{scope}
        \begin{scope}[shift=(50:1), rotate=-40]
            \draw[very thick, dashed] (0,0) -- ++(90:.5);
            \draw[very thick, dashed] (0,0) -- ++(50:.5);
            \node[draw,circle,inner sep=3pt,fill=black, thick] at (0,0) {};
        \end{scope}
    \end{scope}
    \begin{scope}[shift=(50:1), rotate=-40]
        \draw[very thick] (0,0) -- ++(90:1);
        \draw[very thick] (0,0) -- ++(50:1);
        \node[draw,circle,inner sep=3pt,fill=white, thick] at (0,0) {};
        \begin{scope}[shift=(90:1)]
            \draw[very thick, dashed] (0,0) -- ++(90:.5);
            \draw[very thick, dashed] (0,0) -- ++(50:.5);
            \node[draw,circle,inner sep=3pt,fill=black, thick] at (0,0) {};
        \end{scope}
        \begin{scope}[shift=(50:1), rotate=-40]
            \draw[very thick, dashed] (0,0) -- ++(90:.5);
            \draw[very thick, dashed] (0,0) -- ++(50:.5);
            \node[draw,circle,inner sep=3pt,fill=black, thick] at (0,0) {};
        \end{scope}
    \end{scope}
    \node[draw,circle,inner sep=3pt,fill=black, thick] at (0,0) {};
\end{scope}
\end{scope}

\foreach \i in {-3,-1.5,3,4.5,6}{
\draw[dashed, thick] (\i,0) -- ++(90:1);
}
\foreach \i in {-3,0,...,6}{
\node[draw,circle,inner sep=3pt,fill=black, thick] at (\i,0) {};} 
\foreach \i in {-1.5,1.5,...,4.5}{
\node[draw,circle,inner sep=3pt,fill=white, thick] at (\i,0) {};} 
\draw[->, color=red, very thick] (0,-.3) to[out=-45,in=225] (3,-.3); \node[below, color=red] at (1.5,-.9) {$\begin{pmatrix}
    \varpi&0\\0&\varpi^{-1}
\end{pmatrix}$};

\begin{scope}[shift=(0:3.1)]
    \draw[->, color=red, very thick] (0,-.3) to[out=-45,in=225] (3,-.3); \node[below, color=red] at (1.5,-.9) {$\begin{pmatrix}
    \varpi&0\\0&\varpi^{-1}
\end{pmatrix}$};
\end{scope}
\node[below left] at (0,0) {$x_0$};
\node[below left] at (1.5,0) {$x_1$};
\end{tikzpicture}\]

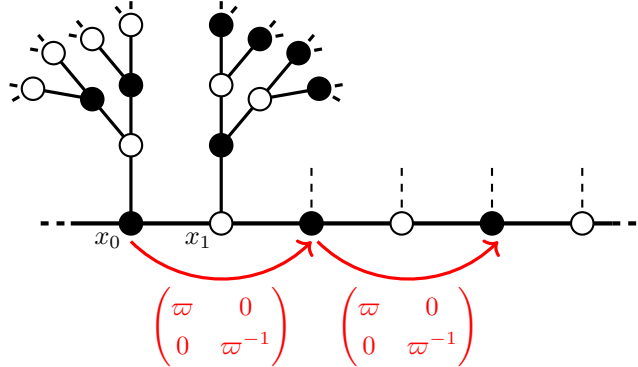
\[\begin{tikzpicture}[scale=0.8]
\draw[ultra thick] (-1,0) -- (8,0);
\draw[dashed, ultra thick] (-1.5,0)--(-1,0);
\draw[dashed, ultra thick] (8,0)--(8.5,0);

\draw[very thick] (0,0) -- ++(90:1.5);
\begin{scope}[shift=(90:1.3)]
    \draw[very thick] (0,0) -- ++(90:1);
    \draw[very thick] (0,0) -- ++(130:1);
    \begin{scope}[shift=(90:1)]
        \draw[very thick] (0,0) -- ++(90:1);
        \draw[very thick] (0,0) -- ++(130:1);
        \node[draw,circle,inner sep=3pt,fill=black, thick] at (0,0) {};
        \begin{scope}[shift=(90:1)]
            \draw[very thick, dashed] (0,0) -- ++(90:.5);
            \draw[very thick, dashed] (0,0) -- ++(130:.5);
            \node[draw,circle,inner sep=3pt,fill=white, thick] at (0,0) {};
        \end{scope}
        \begin{scope}[shift=(130:1), rotate=40]
            \draw[very thick, dashed] (0,0) -- ++(90:.5);
            \draw[very thick, dashed] (0,0) -- ++(130:.5);
            \node[draw,circle,inner sep=3pt,fill=white, thick] at (0,0) {};
        \end{scope}
    \end{scope}
    \begin{scope}[shift=(130:1), rotate=40]
        \draw[very thick] (0,0) -- ++(90:1);
        \draw[very thick] (0,0) -- ++(130:1);
        \node[draw,circle,inner sep=3pt,fill=black, thick] at (0,0) {};
        \begin{scope}[shift=(90:1)]
            \draw[very thick, dashed] (0,0) -- ++(90:.5);
            \draw[very thick, dashed] (0,0) -- ++(130:.5);
            \node[draw,circle,inner sep=3pt,fill=white, thick] at (0,0) {};
        \end{scope}
        \begin{scope}[shift=(130:1), rotate=40]
            \draw[very thick, dashed] (0,0) -- ++(90:.5);
            \draw[very thick, dashed] (0,0) -- ++(130:.5);
            \node[draw,circle,inner sep=3pt,fill=white, thick] at (0,0) {};
        \end{scope}
    \end{scope}
    \node[draw,circle,inner sep=3pt,fill=white, thick] at (0,0) {};
\end{scope}

\begin{scope}[shift=(0:1.5)]
\draw[very thick] (0,0) -- ++(90:1.5);
\begin{scope}[shift=(90:1.3)]
    \draw[very thick] (0,0) -- ++(90:1);
    \draw[very thick] (0,0) -- ++(50:1);
    \begin{scope}[shift=(90:1)]
        \draw[very thick] (0,0) -- ++(90:1);
        \draw[very thick] (0,0) -- ++(50:1);
        \node[draw,circle,inner sep=3pt,fill=white, thick] at (0,0) {};
        \begin{scope}[shift=(90:1)]
            \draw[very thick, dashed] (0,0) -- ++(90:.5);
            \draw[very thick, dashed] (0,0) -- ++(50:.5);
            \node[draw,circle,inner sep=3pt,fill=black, thick] at (0,0) {};
        \end{scope}
        \begin{scope}[shift=(50:1), rotate=-40]
            \draw[very thick, dashed] (0,0) -- ++(90:.5);
            \draw[very thick, dashed] (0,0) -- ++(50:.5);
            \node[draw,circle,inner sep=3pt,fill=black, thick] at (0,0) {};
        \end{scope}
    \end{scope}
    \begin{scope}[shift=(50:1), rotate=-40]
        \draw[very thick] (0,0) -- ++(90:1);
        \draw[very thick] (0,0) -- ++(50:1);
        \node[draw,circle,inner sep=3pt,fill=white, thick] at (0,0) {};
        \begin{scope}[shift=(90:1)]
            \draw[very thick, dashed] (0,0) -- ++(90:.5);
            \draw[very thick, dashed] (0,0) -- ++(50:.5);
            \node[draw,circle,inner sep=3pt,fill=black, thick] at (0,0) {};
        \end{scope}
        \begin{scope}[shift=(50:1), rotate=-40]
            \draw[very thick, dashed] (0,0) -- ++(90:.5);
            \draw[very thick, dashed] (0,0) -- ++(50:.5);
            \node[draw,circle,inner sep=3pt,fill=black, thick] at (0,0) {};
        \end{scope}
    \end{scope}
    \node[draw,circle,inner sep=3pt,fill=black, thick] at (0,0) {};
\end{scope}
\end{scope}

\foreach \i in {3,4.5,...,7.5}{
\draw[dashed, thick] (\i,0) -- ++(90:1);
}
\foreach \i in {0,3,...,6}{
\node[draw,circle,inner sep=3pt,fill=black, thick] at (\i,0) {};} 
\foreach \i in {1.5,4.5,...,7.5}{
\node[draw,circle,inner sep=3pt,fill=white, thick] at (\i,0) {};} 
\draw[->, color=red, very thick] (0,-.3) to[out=-45,in=225] (3,-.3); \node[below, color=red] at (1.5,-.9) {$\begin{pmatrix}
    \varpi&0\\0&\varpi^{-1}
\end{pmatrix}$};

\begin{scope}[shift=(0:3.1)]
    \draw[->, color=red, very thick] (0,-.3) to[out=-45,in=225] (3,-.3); \node[below, color=red] at (1.5,-.9) {$\begin{pmatrix}
    \varpi&0\\0&\varpi^{-1}
\end{pmatrix}$};
\end{scope}
\node[below left] at (0,0) {$x_0$};
\node[below left] at (1.5,0) {$x_1$};
\end{tikzpicture}\]

We want to count vertices in $X^{\gamma}$ up to action of $T$. Given such a vertex $v$, we can apply a translation so that the closest vertex to $v$ on the fundamental apartment is either $x_0$ or $x_1$.

To compute $O_0(\gamma)$ we therefore need to compute the number of black vertices at distance at most $d_\gamma$ from $x_0$ or $x_1$.

Observe on the picture that counting the white vertices closest to $x_0$ at distance at most $d_\gamma$ correspond to black vertices closest to $x_1$. We reduce the count to counting all vertices at distance $\leq d_\gamma$ whose closest vertex on the apartment is $x_0$. We can conclude using Proposition \ref{prop:vertballoutofapt}
\[O_0(\gamma) = q^{d_\gamma},\]
and for $n>0$ we have
\[O_n(\gamma) = q^{n+d_\gamma}-q^{n+d_\gamma-1},\]
where we count vertices at distance $n$ from $X^\gamma$ as the total vertices at distance $n+d_\gamma$ from $x_0$ and remove the ones at distance $n+d_\gamma-1$.\\

Note that it matches exactly the $\mathrm{GL}_2$-orbital integral of $\gamma$ described in \cite[p.415]{kott_bible}. This is because $T$ is split, hence the orbit of $\gamma$ is equal to its stable orbit.

\begin{rem}Before starting computations for elliptic elements. Let us note that for any $\gamma\in \mathrm{SL}_2(F)$, we have that its centralizer is the norm one torus $R_{F[\gamma]/F}^{(1)}\mathbb{G}_m$. Let us note that for hyperbolic $\gamma$, we get $F[\gamma] = F\times F$ and therefore $R_{F[\gamma]/F}^{(1)}\mathbb{G}_m\cong \mathbb{G}_m$ 
 and $H^1(F, R_{F[\gamma]/F}^{(1)}\mathbb{G}_m) = 1$ by Hilbert Theorem $90$. 

 For elliptic elements, $F[\gamma]/F$ is a quadratic field extension and we get  \[H^1(F, R_{F[\gamma]/F}^{(1)}\mathbb{G}_m) \cong \mathbb{Z}/2\mathbb{Z}\] 
 using Proposition \ref{lem:orbits:GLn}.

Therefore, if $\gamma$ is hyperbolic then 
\[O^\mathrm{st}_n(\gamma) = O_n(\gamma).\]
However this will not be the case in the following sections, we expect the stable orbital integrals to split in two terms, one for each rational orbit.
\end{rem}

\subsubsection{Regular semisimple elements: Unramified elliptic case}

\label{sl2unram}

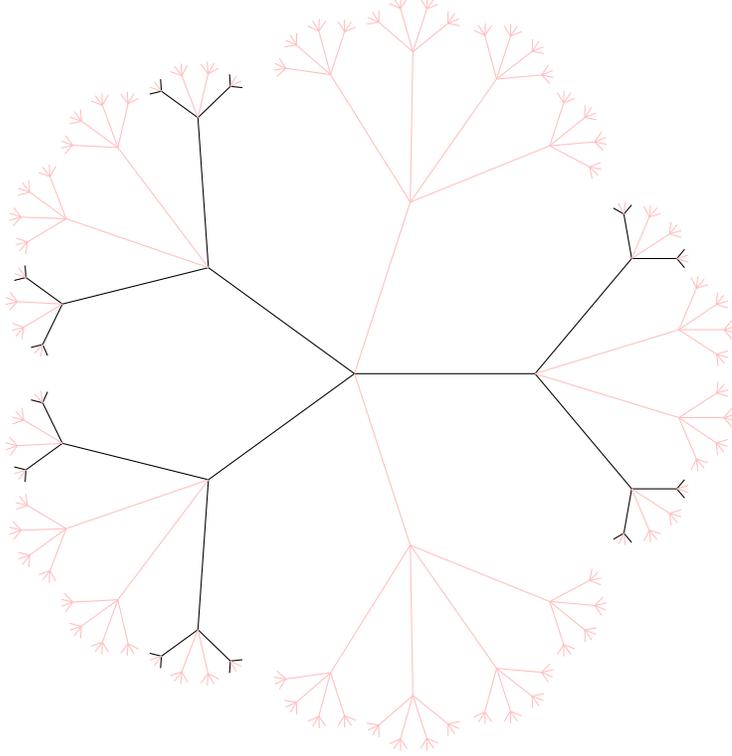
\begin{figure}[H]
\centering
\begin{tikzpicture}[scale=.8]
\draw (0,0) -- (0:3);
\draw[color=pink] (0,0) -- (72:3);
\draw (0,0) -- (144:3);
\draw (0,0) -- (-144:3);
\draw[color=pink] (0,0) -- (-72:3);

\foreach \i in {72,144,...,360}{
    \begin{scope}[shift=(\i:3),rotate=\i]
        \ifthenelse{\i=72 \OR \i=288}{\def\col{pink}}{\def\col{black}}
        \draw[color=\col] (0,0) -- (50:2.5);
        \draw[color=\col] (0,0) -- (-50:2.5);
        \draw[color=pink] (0,0) -- (17:2.5);
         \draw[color=pink] (0,0) -- (-17:2.5);
        \foreach \j in {-17,17,...,17}{
            \begin{scope}[shift=(\j:2.5),rotate=\j]
                \draw[color=pink] (0,0) -- (50:0.75);
                \draw[color=pink] (0,0) -- (-50:0.75);
                \draw[color=pink] (0,0) -- (17:0.75);
                \draw[color=pink] (0,0) -- (-17:0.75);
                \foreach \k in {-17,17,...,17}{
                    \begin{scope}[shift=(\k:0.75),rotate=\k]
                        \draw[color=pink] (0,0) -- (50:0.2);
                        \draw[color=pink] (0,0) -- (-50:0.2);
                        \draw[color=pink] (0,0) -- (17:0.2);
                        \draw[color=pink] (0,0) -- (-17:0.2);
                    \end{scope}
                }
                \foreach \k in {-50,50,...,50}{
                    \begin{scope}[shift=(\k:0.75),rotate=\k]
                        \draw[color=pink] (0,0) -- (50:0.2);
                        \draw[color=pink] (0,0) -- (-50:0.2);
                        \draw[color=pink] (0,0) -- (17:0.2);
                        \draw[color=pink] (0,0) -- (-17:0.2);
                    \end{scope}
                }
            \end{scope}
        }
        \foreach \j in {-50,50,...,50}{
            \begin{scope}[shift=(\j:2.5),rotate=\j]
                \draw[color=\col] (0,0) -- (50:0.75);
                \draw[color=\col] (0,0) -- (-50:0.75);
                \draw[color=pink] (0,0) -- (17:0.75);
                \draw[color=pink] (0,0) -- (-17:0.75);
                \foreach \k in {-17,17,...,17}{
                    \begin{scope}[shift=(\k:0.75),rotate=\k]
                        \draw[color=pink] (0,0) -- (50:0.2);
                        \draw[color=pink] (0,0) -- (-50:0.2);
                        \draw[color=pink] (0,0) -- (17:0.2);
                        \draw[color=pink] (0,0) -- (-17:0.2);
                    \end{scope}
                }
                \foreach \k in {-50,50,...,50}{
                    \begin{scope}[shift=(\k:0.75),rotate=\k]
                        \draw[color=\col] (0,0) -- (50:0.2);
                        \draw[color=\col] (0,0) -- (-50:0.2);
                        \draw[color=pink] (0,0) -- (17:0.2);
                        \draw[color=pink] (0,0) -- (-17:0.2);
                    \end{scope}
                }
            \end{scope}
        }
    \end{scope}
}

\end{tikzpicture}
\caption{Unramified quadratic extension of the building of  $\mathrm{SL}_2$ over $\mathbb{Q}_2$}
\end{figure}
Assume that  $F[\gamma]$ is an unramified quadratic field extension of $F$.  We have that $X^\gamma$ is a ball centered at a vertex, hence we have two cases: the center of the ball is either black or white. 

\begin{itemize}
    \item Assume that the center is black. The number of black vertices in the ball is 
  \[1+(q+1)\left[q+q^3+\cdots +q^{2\lfloor d_\gamma/2\rfloor-1}\right] = \frac{q^{2\lfloor d_\gamma/2\rfloor+1}-1}{q-1}.\]
  Therefore,
  \[O_0(\gamma) =\left\lbrace\begin{array}{ll}
 \frac{q^{d_\gamma+1}-1}{q-1}    &  d_\gamma\equiv0\pmod2 \\ &\\
   \frac{q^{d_\gamma}-1}{q-1}    & d_\gamma\equiv1\pmod2
\end{array}\right. .\]

We will use Proposition \ref{prop:pointfromconvex} to compute $O_n(\gamma)$ for $n>0$. Note that if $d_\gamma$ is even then vertices at distance $n$ from $X^\gamma$ are all black if $n$ is even or all white if $n$ is odd.  We get that 
\begin{align*}O_n(\gamma) 
&=\left\lbrace\begin{array}{ll}
 q^{n-1}\left((q-1)|X^\gamma|+2\right)    &  d_\gamma +n\equiv 0\pmod2 \\ & \\
  0   &  d_\gamma +n\equiv 1\pmod2 \\ 
\end{array}\right..\end{align*}
We plug in $|X^\gamma| = 1+(q+1)\frac{q^{d_\gamma}-1}{q-1}$ and simplify to get 
\[O_n(\gamma) =\left\lbrace\begin{array}{ll}
q^{n+d_\gamma}(1+q^{-1})    &  d_\gamma +n\equiv 0\pmod2 \\ & \\
  0   &  d_\gamma +n\equiv 1\pmod2 \\ 
\end{array}\right. .\]
\item Now assume the center is white. Instead of doing the same work, we may assume that the center is black but we restrict the count to white vertices. We subtract previous results from the total number of vertices in $X^\gamma$ and we get

\[O_0(\gamma) =\left\lbrace\begin{array}{ll}
\frac{q^{d_\gamma}-1}{q-1}     &  d_\gamma\equiv0\pmod2 \\ &\\
   \frac{q^{d_\gamma+1}-1}{q-1}    & d_\gamma\equiv1\pmod2
\end{array}\right. .\] and
\[O_n(\gamma) =\left\lbrace\begin{array}{ll}
0  &  d_\gamma +n\equiv 0\pmod2 \\ & \\
q^{n+d_\gamma}(1+q^{-1})    &  d_\gamma +n\equiv 1\pmod2 \\ 
\end{array}\right. .\]

Computing the stable orbital integral amounts to disregarding the color (since all the colors are the same after extending the building to a quadratic ramified extension). We get 
\[O^{\mathrm{st}}_0(\gamma) = \frac{q^{d_\gamma}-1}{q-1} + \frac{q^{d_\gamma+1}-1}{q-1}  = 1+(q+1)\frac{q^{d_\gamma}-1}{q-1}\]
and 
\[O^{\mathrm{st}}_n(\gamma) =q^{n+d_\gamma}(1+q^{-1})  \]
as in the $\mathrm{GL}_2$-orbital integral in \cite{kott_bible}.
\end{itemize}

\begin{rem}
    Note that we do not really need the power of Proposition \ref{prop:pointfromconvex} to compute $O_n$ for $n>0$. We have that vertices at distance $n$ from $X^\gamma$ are exactly vertices at distance $n+d_\gamma$ from the center. There are exactly \[(q+1)q^{n+d_\gamma-1} = q^{n+d_\gamma}(1+q^{-1})\] such vertices, and these vertices are all the same color. We get the formulae above directly. 
\end{rem}
\subsubsection{Regular semisimple elements: Ramified elliptic case}
\label{sl2ram}

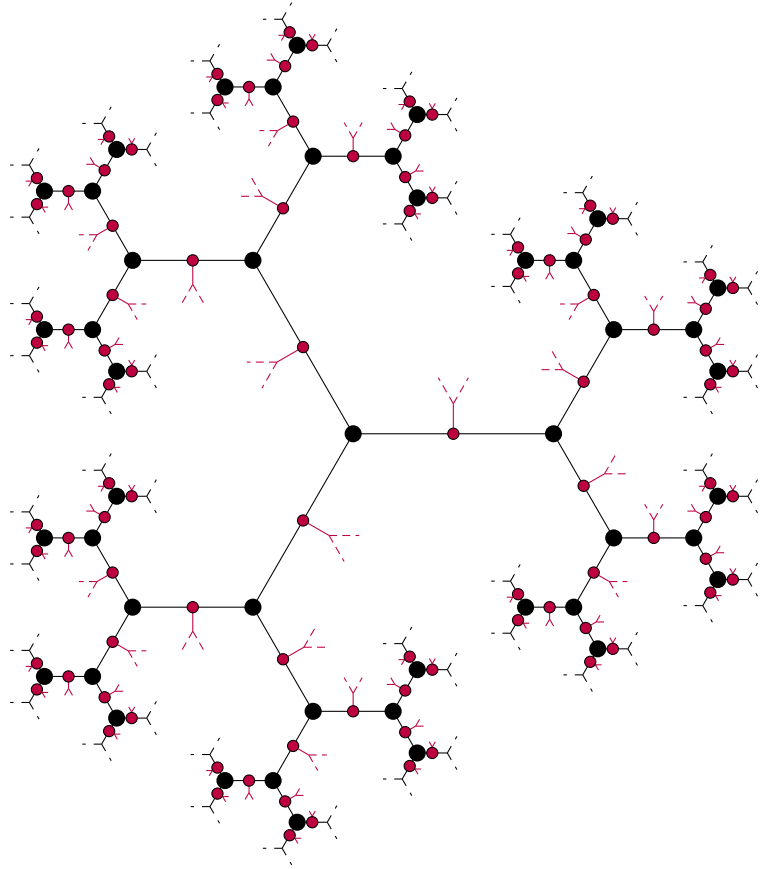
\begin{figure}[H]
\centering
\begin{tikzpicture}[scale=.8]
\begin{scope}[shift=(0:1.665), scale=0.5]
            \draw[color=purple] (0,0)--(0,1);
            \draw[color=purple,densely dashed] (0,1)-- ++(60:1);
            \draw[color=purple,densely dashed] (0,1)-- ++(120:1);
\end{scope}
\begin{scope}[shift=(120:1.665), scale=0.5, rotate=120]
            \draw[color=purple] (0,0)--(0,1);
            \draw[color=purple,densely dashed] (0,1)-- ++(60:1);
            \draw[color=purple,densely dashed] (0,1)-- ++(120:1);
\end{scope}
\begin{scope}[shift=(-120:1.665), scale=0.5, rotate=-120]
            \draw[color=purple] (0,0)--(0,1);
            \draw[color=purple,densely dashed] (0,1)-- ++(60:1);
            \draw[color=purple,densely dashed] (0,1)-- ++(120:1);
\end{scope}
\draw (0,0) --node[draw, circle, inner sep=1.5pt, pos=.5,fill=purple]{} (0:3.33);
\draw (0,0) --node[draw, circle, inner sep=1.5pt, pos=.5,fill=purple]{} (120:3.33);
\draw (0,0) --node[draw, circle, inner sep=1.5pt, pos=.5,fill=purple]{} (240:3.33);
\node[draw,circle,inner sep=2pt,fill=black, thick] at (0,0) {};
\foreach \i in {0,120,...,240}{
    \begin{scope}[shift=(\i:3.33),rotate=\i]
        \begin{scope}[shift=(60:1), scale=0.4, rotate=60]
            \draw[color=purple] (0,0)--(0,1);
            \draw[color=purple,densely dashed] (0,1)-- ++(60:1);
            \draw[color=purple,densely dashed] (0,1)-- ++(120:1);
        \end{scope}
        \begin{scope}[shift=(-60:1), scale=0.4, rotate=-60]
            \draw[color=purple] (0,0)--(0,1);
            \draw[color=purple,densely dashed] (0,1)-- ++(60:1);
            \draw[color=purple,densely dashed] (0,1)-- ++(120:1);
        \end{scope}
        \draw (0,0) --node[draw, circle, inner sep=1.5pt, pos=.5,fill=purple]{} (60:2);
        \draw (0,0) --node[draw, circle, inner sep=1.5pt, pos=.5,fill=purple]{} (-60:2);
        \node[draw,circle,inner sep=2pt,fill=black, thick] at (0,0) {};
        \foreach \j in {-60,60,...,60}{
            \begin{scope}[shift=(\j:2),rotate=\j]
                \begin{scope}[shift=(60:0.665), scale=0.3, rotate=60]
                    \draw[color=purple] (0,0)--(0,1);
                    \draw[color=purple,densely dashed] (0,1)-- ++(60:1);
                    \draw[color=purple,densely dashed] (0,1)-- ++(120:1);
                \end{scope}
                \begin{scope}[shift=(-60:0.665), scale=0.3, rotate=-60]
                    \draw[color=purple] (0,0)--(0,1);
                    \draw[color=purple,densely dashed] (0,1)-- ++(60:1);
                    \draw[color=purple,densely dashed] (0,1)-- ++(120:1);
                \end{scope}
                \draw (0,0) --node[draw, circle, inner sep=1.5pt, pos=.5,fill=purple]{} (60:1.33);
                \draw (0,0) --node[draw, circle, inner sep=1.5pt, pos=.5,fill=purple]{} (-60:1.33);
                \node[draw,circle,inner sep=2pt,fill=black, thick] at (0,0) {};
                \foreach \k in {-60,60,...,60}{
                    \begin{scope}[shift=(\k:1.33),rotate=\k]
                    \begin{scope}[shift=(60:0.4), scale=0.2, rotate=60]
                    \draw[color=purple] (0,0)--(0,1);
                    \draw[color=purple,densely dashed] (0,1)-- ++(60:1);
                    \draw[color=purple,densely dashed] (0,1)-- ++(120:1);
                \end{scope}
                \begin{scope}[shift=(-60:0.4), scale=0.2, rotate=-60]
                    \draw[color=purple] (0,0)--(0,1);
                    \draw[color=purple,densely dashed] (0,1)-- ++(60:1);
                    \draw[color=purple,densely dashed] (0,1)-- ++(120:1);
                \end{scope}
                        \draw (0,0) --node[draw, circle, inner sep=1.5pt, pos=.5,fill=purple]{} (60:0.8);
                        \draw (0,0) --node[draw, circle, inner sep=1.5pt, pos=.5,fill=purple]{} (-60:0.8);
                        \node[draw,circle,inner sep=2pt,fill=black, thick] at (0,0) {};
                        \foreach \l in {-60,60,...,60}{
                            \begin{scope}[shift=(\l:0.8),rotate=\l]
                            \begin{scope}[shift=(60:0.25), scale=0.1, rotate=60]
                                \draw[color=purple] (0,0)--(0,1);
                                \draw[color=purple] (0,1)-- ++(60:1);
                                \draw[color=purple] (0,1)-- ++(120:1);
                            \end{scope}
                            \begin{scope}[shift=(-60:0.25), scale=0.1, rotate=-60]
                                \draw[color=purple] (0,0)--(0,1);300
                                \draw[color=purple] (0,1)-- ++(60:1);
                                \draw[color=purple] (0,1)-- ++(120:1);
                            \end{scope}
                                \draw (0,0) --node[draw, circle, inner sep=1.5pt, pos=.5,fill=purple]{} (60:0.5);
                                \draw (0,0) --node[draw, circle, inner sep=1.5pt, pos=.5,fill=purple]{} (-60:0.5);
                                \node[draw,circle,inner sep=2pt,fill=black, thick] at (0,0) {};
                                \foreach \h in {-60,60,...,60}{
                                    \begin{scope}[shift=(\h:0.5),rotate=\h]
                                        \draw[dashed] (0,0) -- (60:0.32);
                                        \draw[dashed] (0,0) -- (-60:0.32);
                                    \end{scope}
                                }
                            \end{scope}
                        }
                    \end{scope}
                }
            \end{scope}
        }
    \end{scope}
}

\end{tikzpicture}
\caption{Ramified quadratic extension of the building of $\mathrm{SL}_2$ over $\mathbb{Q}_2$}
\end{figure}

We need to count black vertices contained in a ball around a midpoint $v$. The situation simplifies since if vertices at distance $d$ from $v$ are a white on one side, they are black on the other side.

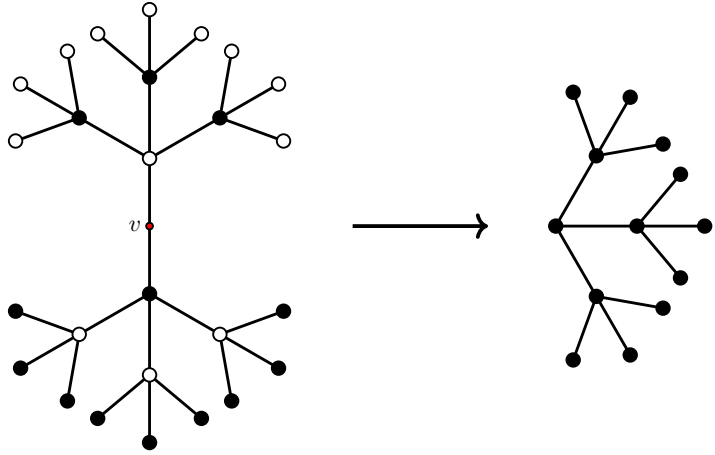
\begin{figure}[H]
\centering
\scalebox{0.9}{
\begin{tikzpicture}
\draw[very thick] (0,0) -- (90:1);
\draw[very thick] (0,0) -- (270:1);
\node[draw,circle,inner sep=1pt,fill=red, thick] at (0,0) {};
\node[left] at (0,0) {$v$};

\begin{scope}[shift=(90:1),rotate=90]
\foreach \i in{-60,0,60}{
\draw[very thick] (0,0) -- (\i:1.2);

\begin{scope}[shift=(\i:1.2), rotate=\i]
\foreach \j in {-50,0,50}{
\draw[very thick] (0,0) -- (\j:1);
\node[draw,circle,inner sep=2pt,fill=white, thick] at (\j:1) {};
}
\node[draw,circle,inner sep=2pt,fill=black, thick] at (0,0) {};
\end{scope}}
\node[draw,circle,inner sep=2pt,fill=white, thick] at (0,0) {};
\end{scope}

\begin{scope}[shift=(-90:1),rotate=-90]
\foreach \i in{-60,0,60}{
\draw[very thick] (0,0) -- (\i:1.2);

\begin{scope}[shift=(\i:1.2), rotate=\i]
\foreach \j in {-50,0,50}{
\draw[very thick] (0,0) -- (\j:1);
\node[draw,circle,inner sep=2pt,fill=black, thick] at (\j:1) {};
}
\node[draw,circle,inner sep=2pt,fill=white, thick] at (0,0) {};
\end{scope}}
\node[draw,circle,inner sep=2pt,fill=black, thick] at (0,0) {};
\end{scope}

\draw[ultra thick,->] (3,0)--(5,0);

\begin{scope}[shift=(0:6),rotate=0]
\foreach \i in{-60,0,60}{
\draw[very thick] (0,0) -- (\i:1.2);

\begin{scope}[shift=(\i:1.2), rotate=\i]
\foreach \j in {-50,0,50}{
\draw[very thick] (0,0) -- (\j:1);
\node[draw,circle,inner sep=2pt,fill=black, thick] at (\j:1) {};
}
\node[draw,circle,inner sep=2pt,fill=black, thick] at (0,0) {};
\end{scope}}
\node[draw,circle,inner sep=2pt,fill=black, thick] at (0,0) {};
\end{scope}
\end{tikzpicture}}
\end{figure}

We have that $d_\gamma$ is a half-integer now. We get that the number of black vertices at distance at most $d_\gamma$ is

\[1+q+q^2+\cdots + q^{d_{\gamma}-1/2} = \frac{q^{d_\gamma+1/2}-1}{q-1}.\]

\begin{prop}
    If $\gamma$ is ramified elliptic, we have 
    \[O_0(\gamma) =\frac{q^{d_\gamma+1/2}-1}{q-1},\]
    and  if $n>0$ 
    \[O_n(\gamma) =q^{n+d_\gamma-1/2}. \]
\end{prop}
\begin{proof}
    We get the second part directly from Proposition \ref{prop:pointfromconvex} but we should only count black vertices. We have $|X^\gamma| = 2\frac{q^{d_\gamma+1/2}-1}{q-1}$, however half of the vertices at distance $n$ from $X_{\gamma}$ are black and the other are white, depending on which neighbor of $v$ they are closest to. Therefore we have 
    \begin{align*}
        O_n(\gamma) &= \frac{1}{2}\left(q^{n-1}\left((q-1)|X^\gamma|+2\right)\right)\\
        &= q^{n+d_\gamma-1/2}.
    \end{align*}
\end{proof}

Note that once we extend the building to a ramified extension, all vertices in the base building become the same color (since midpoints will be the new color of vertices). We get that the stable orbital integral makes up for the missed vertices. We get 
\[O^{\mathrm{st}}_n(\gamma) = 2 O_n(\gamma).\]

\subsubsection{Shalika germs}

We saw that there are exactly $5$ conjugacy classes of unipotent elements in $\mathrm{SL}_2(F)$. Not all of them give us independent functions. Conjugacy classes that merge over the unique unramified quadratic extension yield the same orbital integrals. We will therefore group them and only consider $3$ unipotent classes $U_0, U_1, U_{\varpi}$, represented by 
\[\begin{pmatrix}1&0\\0&1\end{pmatrix},\ \begin{pmatrix}1&1\\0&1\end{pmatrix},\ \begin{pmatrix}1&\varpi\\0&1\end{pmatrix}\]
respectively. \\

For $x\in \{0,1,\varpi\}$ we associate the function 
\[L_x: \mathcal{H}_K(\mathrm{SL}_2(F))\rightarrow \mathbb{C},\]
given by $L_x (f) = 2O(\alpha(x), f)$ if $x\neq 0$ and $O(\alpha(0), f)$ if $x=0$ (this takes into account the number of orbits that are represented by $\alpha(x)$ over an unramified extension. Write $f_n = \mathds{1}_{Kt_nK}$. Every function in $\mathcal{H}_K(\mathrm{SL}_2(F))$ is a linear combination of such $f_i$'s. We claim that given a certain torus $\mathcal{T}$ (thought as a centralizer of some element $\gamma$), we can write all orbital integrals for regular elements in $\mathcal{T}(F)\cap K$ as a linear combination of unipotent orbital integrals. This motivates the following theorems.

\begin{theorem}[Elliptic Shalika germs]\label{th:shalikasl2ellip}
    For all elliptic regular semisimple $\gamma\in K$, there are explicit functions $A(\gamma), B(\gamma), C(\gamma)\in \mathbb{Z}(q)$ so that the map $\mathcal{L} : \mathrm{Orb}(\gamma, \underline{\ }): 
\mathcal{H}_K(F)\rightarrow \mathbb{C}$ can be written as a linear combination of unipotent orbital integrals, i.e. 
\begin{equation}\label{th:shalikaellipt}\mathcal{L}  = A(\gamma) L_0+B(\gamma) L_1+ C(\gamma) L_\varpi.\end{equation}
Explicitly, we have 
\[A(\gamma) = \frac{1}{1-q},\ B(\gamma) = q^{-1}\left(q|X^\gamma_{\scalebox{0.5}{\CIRCLE}
}| - |X^\gamma_{\scalebox{0.5}{\Circle}}|+1\right),\ C(\gamma) = q^{-1}\left(q|X^\gamma_{\scalebox{0.5}{\Circle}}| - |X^\gamma_{\scalebox{0.5}{\CIRCLE}}|+1\right).\]
\end{theorem}
\begin{proof}
    We can check it directly, but for simplicity, let us sketch the derivation process. Let us start from the conclusion and look for $A= A(\gamma), B = B(\gamma), C= C(\gamma)$ such that Equation (\ref{th:shalikaellipt}) holds. \\

    \[\mathcal{L}(f_n) = A\left\lbrace\begin{array}{ll}
       1  & n=0  \\
        0 & n>0
    \end{array}\right.\!+B\left\lbrace\begin{array}{ll}
       \frac{1}{1-q^{-2}}  & n=0  \\
        q^n & n>0\text{ even}\\
        0 & n>0\text{ odd}\\
    \end{array}\right.\!+C\left\lbrace\begin{array}{ll}
       \frac{q^{-1}}{1-q^{-2}}  & n=0  \\
        0 & n>0\text{ even} .\\
        q^n & n>0\text{ odd}\\
    \end{array}\right.  \]
    We use the fact that $\mathcal{L}(f_0) = |X^\gamma_{\scalebox{0.5}{\CIRCLE}}|$, and for $n>0$ we have $\mathcal{L}(f_n)$ equals the number of black vertices at distance $n$ from $X^\gamma$, which we can compute using Proposition  \ref{prop:coloredawayfromconvex}. 
    The system becomes 
    \[\left\lbrace\begin{array}{l}
      A + \frac{1}{1-q^{-2}}B+\frac{q^{-1}}{1-q^{-2}}C = |X^\gamma_{\scalebox{0.5}{\CIRCLE}}|    \\ \\ 
        q^n B = q^{n-1}\left(q |X^\gamma_{\scalebox{0.5}{\CIRCLE}}|-|X^\gamma_{\scalebox{0.5}{\Circle}}|+1|\right) \\ \\ 
        q^n C = q^{n-1}\left(q |X^\gamma_{\scalebox{0.5}{\Circle}}|-|X^\gamma_{\scalebox{0.5}{\CIRCLE}}|+1|\right) \\
    \end{array}\right., \]
    which yields the desired identity.
\end{proof}
\begin{rem}
Note that the three unipotent orbit in the Shalika germ expansions can be partitioned into two $\mathrm{GL}_2(F)$-conjugacy classes, namely writing $\mathcal{C}(g)$ for the $\mathrm{SL}_2$-orbit of $g$ 
\[\{\mathcal{C}(\alpha(0))\},\ \{\mathcal{C}(\alpha(1)),\mathcal{C}(\alpha(\varpi))\}.\]
We also know that elliptic elements  have exactly two orbits in the corresponding rational orbit. The orbital integral on the other orbit will have the same expansion, but where the coefficient for $\mathcal{C}(g)$ is switched to the orbit in the previous set. Indeed, notice that summing the coefficients for each orbit we get 
\[A(\gamma)+A(\gamma) = \frac{2}{1-q},\]
\[B(\gamma)+C(\gamma) = q^{-1}((q+1)|X^\gamma|+2),\]
which are exactly the Shalika germ expansion for $\mathrm{GL}_2$ as computed in \cite[p.416]{kott_bible}.
\end{rem}

We have a similar statement for split tori (hyperbolic orbits). 
\begin{theorem}[Hyperbolic Shalika germs]
    For all hyperbolic regular semisimple $\gamma\in K$, there are explicit functions $A(\gamma), B(\gamma), C(\gamma)\in \mathbb{Z}[q^{-1}]$ so that the map $\mathcal{L} : \mathrm{Orb}(\gamma, \underline{\ }): 
\mathcal{H}_K(F)\rightarrow \mathbb{C}$ can be written as a linear combination of unipotent orbital integrals, i.e. 
\begin{equation}\label{th:shalikahyper}\mathcal{L}  = A(\gamma) L_0+B(\gamma) L_1+ C(\gamma) L_\varpi.\end{equation}
Explicitly, we have 
\[A(\gamma) =0,\ B(\gamma) = C(\gamma) 
 = q^{d_\gamma}(1-q^{-1}).\]
\end{theorem}
\begin{proof}
    This is a simpler version of the previous Theorem. 
\end{proof}

\begin{rem}
    This time, the orbit of hyperbolic elements are equal to their stable orbit, hence the Shalika germ expansion is the same as the one for $\mathrm{GL}_2(F)$ on the nose, as described in \cite[p.415]{kott_bible}.
\end{rem}

\section{Explicit orbital integrals on the distinguished subgroup}\label{sec:exporbit}
In this section, we let $\mathbf{G} = \mathrm{GL}_2\times_{\det}\mathrm{GL}_2$ and $G = \mathbf{G}(F)$. Note that we have 
\[\mathrm{SL}_2(F)\times\mathrm{SL}_2(F) \subset G(F)\subset \mathrm{GL}_2(F)\times \mathrm{GL}_2(F),\]
and orbital integrals of $\mathrm{SL}_2^2$ are just products of the ones computed in \S \ref{sec:sl2} hence we expect orbital integrals to look similar to products of those. We will regularly do the explicit comparison. 

\subsubsection{The building and spherical functions}

We will still use a building-theoretic approach, where the building of $G$ is $\mathcal{B}(\mathrm{GL}_2\times\mathrm{GL}_2,F)\cong\mathcal{B}(\mathrm{GL}_2,F)\times \mathcal{B}(\mathrm{GL}_2,F)$, the product of two colored trees for $\mathrm{GL}_2$. Figure \ref{fig:aptgl22} represents an apartment. 

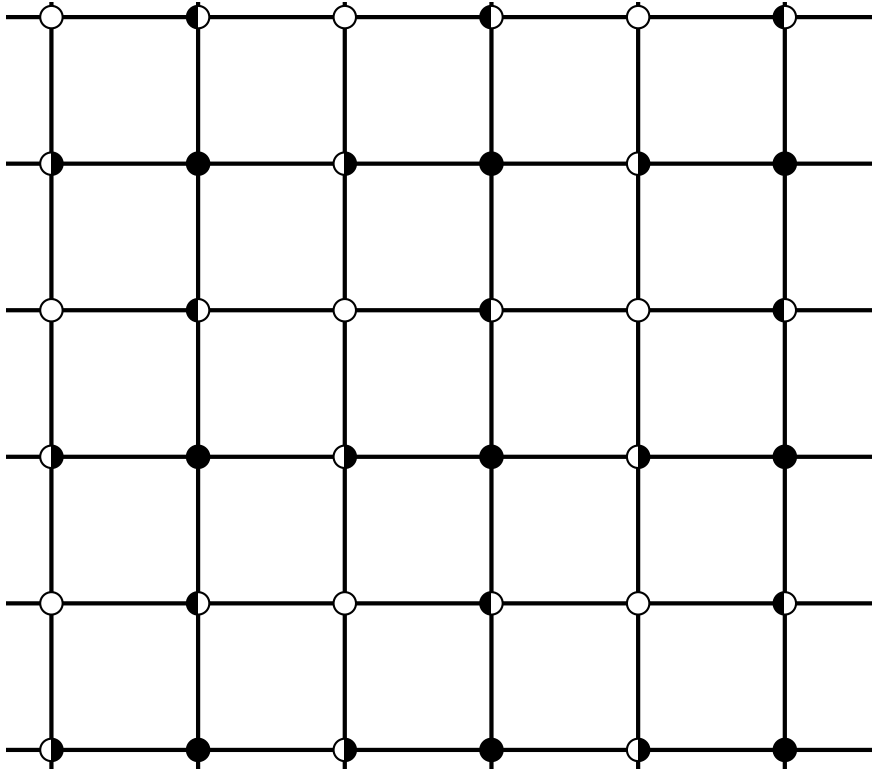
\begin{figure}
\centering
\begin{tikzpicture}[cross/.style={path picture={ 
  \draw[black]
(path picture bounding box.south east) -- (path picture bounding box.north west) (path picture bounding box.south west) -- (path picture bounding box.north east);
}}]
    \coordinate (Origin)   at (0,0);
    \coordinate (XAxisMin) at (-5,0);
    \coordinate (XAxisMax) at (5,0);
    \coordinate (YAxisMin) at (0,-3);
    \coordinate (YAxisMax) at (0,3);

    \clip (-4.5,-2.2) rectangle (7cm,8cm); 
     \pgftransformscale{1.3}

    \coordinate (Bone) at (0,2);
 \pgftransformscale{.75}    \coordinate (Btwo) at (2,-2);
    \draw[style=help lines, color=black ,ultra thick] (-14,-14) grid[step=2cm] (14,14);

    \foreach \x in {-7,-6,...,7}{
      \foreach \y in {-7,-6,...,7}{ 
        \node[draw,circle,inner sep=3pt,split fill=black and white, thick] at (2+4*\x,4*\y) {};
        \node[draw,circle,inner sep=3pt,fill=black, thick] at (2+4*\x,2+4*\y) {};
        \node[draw,circle,inner sep=3pt,,split fill=white and black, thick] at (4*\x,2+4*\y) {};
        \node[draw,circle,inner sep=3pt,fill=white, thick] at (4*\x,4*\y) {};
      }
    }
    
  \end{tikzpicture}
  \caption{Apartment of the building of $G$}
  \label{fig:aptgl22}
\end{figure}

Call the building $Y = X\times X$ where $X$ is the colored tee of $\mathrm{SL}_2(F)$. Note that we have $4$ colors coming from colorings of each building.  

Elements of the split diagonal torus acts on the fundamental apartment, via elements of the form $\begin{pmatrix}
    1&0\\ 0&\varpi^k
\end{pmatrix},\begin{pmatrix}
    \varpi^m&0\\ 0&\varpi^{k-m}
\end{pmatrix}$. If we move a vertex by a distance $k$ horizontally, it needs to be moved by some $k-2m$ vertically, hence the parity must be preserved.

We conclude that
\begin{itemize}
\item $\mathrm{GL}_2(F)\times\mathrm{GL}_2(F)$ acts transitively on vertices of $Y$,
\item $\mathrm{SL}_2(F)\times\mathrm{SL}_2(F)$ preserves the colored patterns and acts transitively on set of vertices of same pattern,
\item  $G$ preserves the two type of vertices: monocolor or bicolor, and acts transitively on each set.
\end{itemize}

\subsubsection{Unipotent orbits}

Since there are $2$ (resp. $5$) unipotent orbits in $\mathrm{GL}_2(F)$ (resp. $\mathrm{SL}_2(F)$), we have 
$4$ (resp. $25$) unipotent orbits in $\mathrm{GL}_2(F)\times\mathrm{GL}_2(F)$ (resp. $\mathrm{SL}_2(F)\times\mathrm{SL}_2(F)$).

Let $\alpha: F\times F\rightarrow G$ defined by 
\[\alpha(x,y) = \begin{pmatrix}
    1&x\\ 0&1
\end{pmatrix},\begin{pmatrix}
    1&y\\ 0&1
\end{pmatrix}.\]

Let $u\in \mathcal{O}_F$ be a nonsquare unit. The unipotent orbits in  $\mathrm{SL}_2(F)\times\mathrm{SL}_2(F)$  are represented by 
\[\alpha(x,y)\ x,y\in \{0,1,u,\varpi, u\varpi\}.\]
We just need to establish which of those orbits are $G$-conjugate.

\begin{prop}
    There are $7$ unipotent orbits in $G$, represented by 
    \[\alpha(0,0),\ \alpha(1,0),\ \alpha(0,1),\ \alpha(1,1),\]
    \[\alpha(1,u), \alpha(1,\varpi),  \alpha(u,\varpi).\]
\end{prop}
\begin{proof}
For $x\neq 0$ we have
\[\left(\begin{pmatrix}
        1&0\\0&x
    \end{pmatrix},\begin{pmatrix}
        1&0\\0&x
    \end{pmatrix}\right)^{-1}\alpha(x,0)\left(\begin{pmatrix}
        1&0\\0&x
    \end{pmatrix},\begin{pmatrix}
        1&0\\0&x
    \end{pmatrix}\right) = \alpha(1,0).\]
Therefore, elements $\alpha(0,0),\alpha(1,0), \alpha(0,1), \alpha(1,1)$ represent distinct orbits in $\mathrm{GL}_2(F)\times \mathrm{GL}_2(F)$ hence they are still distinct under $G$. We reduce the search to elements of the form $\alpha(x,y)$ and $x,y\neq 0$. \\

    Since $\mathrm{GL}_2(F)$ embeds diagonally in $G$, all elements of the form $\alpha(x,x)$ (where $x\neq 0$) are in the orbit of $\alpha(1,1)$.

Lastly, we observe  that $\alpha(x,y)$ is conjugate to $\alpha(1,xy)$ if $x\neq 0$ since 
    \[\left(\begin{pmatrix}
        x&0\\0&1
    \end{pmatrix},\begin{pmatrix}
        1&0\\0&x
    \end{pmatrix}\right)^{-1}\alpha(x,y)\left(\begin{pmatrix}
        x&0\\0&1
    \end{pmatrix},\begin{pmatrix}
        1&0\\0&x
    \end{pmatrix}\right) = \alpha(1,xy).\]
In particular, if both $x,y\neq 0$ then  $\alpha(x,y)$ is conjugate to $\alpha(1,xy)$, which itself is conjugate to $\alpha(y,x)$.
\end{proof}

Note that as seen in $\mathrm{SL}_2$,  unipotent orbits that are in the same orbit over the unique quadratic unramified extension have the same orbital integrals.

\begin{prop} After base change to the unique unramified quadratic extension of $F$ (which is $F[\sqrt{u}]$), we only have  $5$ unipotent orbits left, represented by 
    \[\alpha(0,0),\ \alpha(1,0),\ \alpha(0,1),\ \alpha(1,1),\ \alpha(1,\varpi).\]
    The last two represent the same stable regular unipotent orbit.
\end{prop}

Informed by our computations on $\mathrm{SL}_2$, we will only consider these orbits since orbits merging after unramified extensions give the same contributions. We will multiply the Haar measure  on $F$ chosen in the case of $\mathrm{SL}_1$ by the number of rational orbits.

We now let $K = \mathrm{GL}_2\times\mathrm{GL}_2(\mathcal{O})$ and  $f_{n,m} = \mathds{1}_{K(t_n, t_m)K}$ where $t_i = \begin{pmatrix}
    \varpi^i&0\\0&\varpi^{-i}
\end{pmatrix}$. Write $O_{n,m}(\gamma)$ for $\mathrm{Orb}(\gamma, f_{n,m})$.

\[O_{n,m}(\alpha(0,0)) = \left\lbrace \begin{array}{ll}
     1& \text{if }n=m=0 \\
     0 & \text{else}
\end{array}\right. .\]
\[O_{n,m}(\alpha(1,0)) = \left\lbrace \begin{array}{ll}
     0& \text{if }m\neq 0 \\
      \frac{1}{1-q^{-1}}&\text{if }n=m=0\\
     q^n& \text{if }m=0\ n>0
\end{array}\right. .\]
\[O_{n,m}(\alpha(0,1)) = \left\lbrace \begin{array}{ll}
     0& \text{if }n\neq 0 \\
      \frac{1}{1-q^{-1}}&\text{if }n=m=0\\
     q^m& \text{if }n=0\ m>0
\end{array}\right. .\]

While it might seem at a first glance that the orbital integrals at $\alpha(1,0)$ and $\alpha(0,1)$ are identical, the roles of $n$ and $m$ are swapped, so the functions are still linearly independent.

Now we look at $\alpha(1,1)$ and $\alpha(1,\varpi)$.
Putting coordinates $(i,j)$ on the fundamental apartment. The set of fixed points of $\alpha(1,1)$ is the quadrant $i,j\leq 0$ whereas the set of fixed points of $\alpha(1,\varpi)$ is the set $i\leq 0$ and $j\leq 1$.

For $n=m=0$ we compute the sum 
\[\sum_{\underset{i+j\equiv0\pmod2}{i,j\leq 0}}\frac{1}{\mathrm{Stab}_{N\times N}(x_{i},x_{j})} =\sum_{\underset{i+j\equiv0\pmod2}{i,j\leq 0}}\frac{1}{q^{i+j}}= \sum_{k\geq 0}\frac{2k+1}{q^{2k}}=f'(1/q),\]
where $\displaystyle f(x) = \sum_{k\geq 0}x^{2k+1} = \frac{x}{1-x^2}$. We get 
\[O_{0,0}(\alpha(1,1)) = \frac{1+q^{-2}}{(1-q^{-2})^2}.\]

\[O_{0,0}(\alpha(1,\varpi))=\sum_{\underset{i+j\equiv0\pmod2}{i\leq 0,j\leq 1}}\frac{1}{\mathrm{Stab}_{N\times N}(x_{i},x_{j})} =\sum_{\underset{i+j\equiv0\pmod2}{i\leq 0,j\leq 1}}\frac{1}{q^{i}q^{j+1}}= \sum_{k\geq 0}\frac{2k+2}{q^{2k+1}}=g'(1/q),\]
where $\displaystyle g(x) = \sum_{k\geq 0}x^{2k+2} = \frac{x^2}{1-x^2}$. We get 
\[O_{0,0}(\alpha(1,\varpi)) = \frac{2q^{-1}}{(1-q^{-2})^2}.\]

\begin{rem}
    We can compare this result with the corresponding integral over the orbit of $\alpha(1,1)$ for $\mathrm{GL}_2^2$ which yields
    \[\frac{1}{1-q^{-1}}\frac{1}{1-q^{-1}}=\frac{1}{(1-q^{-1})^2}.\]
     Note that we grouped the unipotent integrals for $\mathrm{SL}_2$ giving the same contribution.

The $\mathrm{GL}_2^2$-orbit of $\alpha(1,1)$ is decomposed in the two orbits in $G$ represented by   $\alpha(1,1)$ and $\alpha(1,\varpi)$. Indeed, observe that with our choice of measure, we get 
\[O_{0,0}(\alpha(1,1))+O_{0,0}(\alpha(1,\varpi))= \frac{1+q^{-2}}{(1-q^{-2})^2}+\frac{2q^{-1}}{(1-q^{-2})^2} = \frac{(1+q^{-1})^2}{(1-q^{-2})^2} = \frac{1}{(1-q^{-1})^2}.\]
\end{rem}

\begin{rem}
    Similarly, the $G$-orbit of $\alpha(1,1)$ corresponds to two $\mathrm{SL}_2^2$-orbits (up to unramified extension), represented by $\alpha(1,1)$ and $\alpha(\varpi, \varpi)$. Adding both those integrals we get 
    \[\frac{1}{1-q^{-2}}\frac{1}{1-q^{-2}}+\frac{q^{-1}}{1-q^{-2}}\frac{q^{-1}}{1-q^{-2}} = \frac{1+q^{-2}}{(1-q^{-2})^2} = O_{0,0}(\alpha(1,1)).\]
 The orbit of $\alpha(1,\varpi)$ is also decomposed in two $\mathrm{SL}_2^2$-orbits represented by $\alpha(1,\varpi)$ and $\alpha(\varpi,1)$, and summing them we get 
 \[\frac{q^{-1}}{1-q^{-2}}\frac{1}{1-q^{-2}}+\frac{q^{-1}}{1-q^{-2}}\frac{1}{1-q^{-2}} = \frac{2q^{-1}}{(1-q^{-2})^2} = O_{0,0}(\alpha(1,\varpi)).\]
\end{rem}

For the case $n,m>0$. 
In this case, we are looking for a point at a horizontal distance $n$ and vertical distance $m$ from the quadrant of fixed points. There is only one such point. For $\alpha(1,1)$ it is the point of coordinates $(n,m)$, however, this vertex is only monochrome if $n+m\equiv 0\pmod 2$. We therefore get 
\[O_{n,m}(\alpha(1,1)) = \left\lbrace\begin{array}{ll}
   q^{n+m}  & \text{if }n+m\equiv0\pmod2 \\
    0 & \text{else} 
\end{array}\right..\]

For $\alpha(1,\varpi)$, the only vertex to consider is the one of coordinates $(n,m+1)$. This vertex  is monochrome only if $n+m+1$ is even. We therefore get (remember the modified measure on the second copy of $F$)
\[O_{n,m}(\alpha(1,\varpi)) = \left\lbrace\begin{array}{ll}
   q^{n+m} & \text{if }n+m\equiv1\pmod2 \\
    0 & \text{else} 
\end{array}\right..\]

Now we go to $n=0,m>0$. We look for vertices $(i,j)$ such that $i\leq 0$ and $j = m$. If $m$ is even then we get vertices of the form $(-2i,m)$, but if $m$ is odd, then we get vertices of the form $(-(2i+1),m)$. We get 
\begin{align*}
O_{0,m}(\alpha(1,1)) &=\left\lbrace \begin{array}{ll}
 \sum_{i\leq 0}\frac{1}{q^{2i-m}}    &  m\equiv0\pmod2\\
   \sum_{i\leq 0}\frac{1}{q^{2i+1-m}}   & m\equiv1\pmod2
\end{array}\right. \\&= \frac{q^m}{1-q^{-2}} \left\lbrace\begin{array}{ll}
   1  & m\equiv0\pmod2 \\
    q^{-1} & m\equiv1\pmod2 
\end{array}\right. .\end{align*}

Similarly, for $\alpha(1,\varpi)$ we look at vertices $(i,j)$ where $i\leq 0$ and $j = m+1$.
\begin{align*}
O_{0,m}(\alpha(1,1)) &=\left\lbrace \begin{array}{ll}
 \sum_{i\geq 0}\frac{1}{q^{2i-m}}    &  m\equiv0\pmod2\\
   \sum_{i\geq 0}\frac{1}{q^{2i+1-m}}   & m\equiv1\pmod2
\end{array}\right. \\&= \frac{q^m}{1-q^{-2}} \left\lbrace\begin{array}{ll}
   1  & m\equiv0\pmod2 \\
    q^{-1} & m\equiv1\pmod2 
\end{array}\right. .\end{align*}

Lastly, we treat the case $m=0$ and $n>0$.

The situation for $\alpha(1,1)$ is identical to the previous case, swapping $n$ and $m$ so 

\[O_{n,0}(\alpha(1,1)) = \frac{q^n}{1-q^{-2}} \left\lbrace\begin{array}{ll}
   1  & n\equiv0\pmod2 \\
    q^{-1} & n\equiv1\pmod2 
\end{array}\right. .\]

For $\alpha(1,\varpi)$, however, we will need to count vertices $(i,j)$ with  $j\leq 1$ and $i = n$.

\begin{align*}
O_{n,0}(\alpha(1,\varpi)) &=\left\lbrace \begin{array}{ll}
 \sum_{i\geq 0}\frac{1}{q^{2i-n+1}}    &  n\equiv0\pmod2\\
   \sum_{i\geq 0}\frac{1}{q^{2i-1-n+1}}   & n\equiv1\pmod2
\end{array}\right. \\&= \frac{q^n}{1-q^{-2}} \left\lbrace\begin{array}{ll}
   q^{-1}  & n\equiv0\pmod2 \\
    1 & n\equiv1\pmod2 
\end{array}\right. .\end{align*}

To sum up, we proved. 

\begin{theorem}
    The unipotent orbital integrals are equal to 
    \[O_{n,m}(\alpha(0,0)) = \left\lbrace \begin{array}{ll}
     1& \text{if }n=m=0 \\
     0 & \text{else}
\end{array}\right. ,\]
\[O_{n,m}(\alpha(1,0)) = \left\lbrace \begin{array}{ll}
     0& \text{if }m\neq 0 \\
      \frac{1}{1-q^{-1}}&\text{if }n=m=0\\
     q^n& \text{if }m=0\ n>0
\end{array}\right. =O_{m,n}(\alpha(0,1))  ,\]
\[O_{n,m}(\alpha(1,1))= \left\lbrace\begin{array}{ll}
   \frac{1+q^{-2}}{(1-q^{-2})^2}  &  m=n=0\\&\\
 \left\lbrace\begin{array}{ll}
   q^{n+m}  & \text{if }n+m\equiv0\pmod2 \\
    0 & \text{if }n+m\equiv1\pmod2 
\end{array}\right.    &  m,n>0\\
&\\
 \frac{q^{n+m}}{1-q^{-2}} \left\lbrace\begin{array}{ll}
   1  & n+m\equiv0\pmod2 \\
    q^{-1} & n+m\equiv1\pmod2 
\end{array}\right. & mn=0, m+n>0
\end{array}\right.,\]
\[O_{n,m}(\alpha(1,\varpi))= \left\lbrace\begin{array}{ll}
   \frac{2q^{-1}}{(1-q^{-2})^2}  &  m=n=0\\&\\
 \left\lbrace\begin{array}{ll}
   0 & \text{if }n+m\equiv0\pmod2 \\
    q^{n+m}  & \text{if }n+m\equiv1\pmod2 
\end{array}\right.    &  m,n>0\\
&\\
 \frac{q^{n+m}}{1-q^{-2}} \left\lbrace\begin{array}{ll}
   q^{-1}  & n+m\equiv0\pmod2 \\
    1 & n+m\equiv1\pmod2 
\end{array}\right. & mn=0, m+n>0
\end{array}\right. .\]
\end{theorem}
\begin{rem}
    As previously, we can add orbital integrals in the same $\mathrm{GL}_2^2$-orbit to recover the corresponding orbit. For example, for $\alpha(1,1)$ we compute $O_{n,m}(\alpha(1,1))+O_{n,m}(\alpha(1,\varpi))$ and get 
    \[\left\lbrace\begin{array}{ll}
  \frac{1}{(1-q^{-1})^2}  &  m=n=0\\&\\
      q^{n+m} &  m,n>0\\
&\\
 \frac{q^{n+m}}{1-q^{-1}} & mn=0, m+n>0
\end{array}\right..\]
\end{rem}

\subsection{Regular semisimple elements}

Assume $\gamma\in G$ is regular semisimple. We know that $G_\gamma$ is a maximal torus in $\mathrm{GSp}_4$.

\subsection{Hyperbolic orbits}

Assume $\gamma  = (\gamma_1,\gamma_2)$ is hyperbolic, i.e. $G_\gamma \cong (F^\times)^3$. More explicitly, the set of rational points of the torus is identified with 

\[\{(x,y,z,t)\in (F^{\times})^4 \cong F[\gamma]^\times \ :\ xy=zt\}.\]

We proceed as in the $\mathrm{SL}_2$ case. For $O_0(\gamma)$ we look at monochrome vertices $(v_1,v_2)\in Y$ so that $v_1$ is at distance $d_{\gamma_1}$ from the fundamental apartment of $X$ and $v_2$ is at distance $d_{\gamma_2}$ from the fundamental apartment, up to action of $T$, the split torus corresponding to the apartment.

Up to action of $T$, acting transitively on sets of monochrome and bichrome vertices, we can assume that the closest vector on the apartment is either $(x_0,x_0)$ or $(x_0,x_1)$. For each vertex we will split the sum in counting the two combinations independently: $\CIRCLE\CIRCLE$ and $\Circle\Circle$.

\label{chi:notation}Given a vertex $x$ in the fundamental apartment write $\chi_{d}(x)$ the number of vertices in $X$ at distance $d$ from $x$ and whose closest vector in the apartment is $x$. Write $\chi_d^{{\scalebox{0.5}{\CIRCLE}}}(x), \chi_d^{{\scalebox{0.5}{\Circle}}}(x)$ for the counts restricted so vertices of specific colors. 

Note that $\chi_d^{{\scalebox{0.5}{\CIRCLE}}}(x_0)= \chi_d^{{\scalebox{0.5}{\Circle}}}(x_1)$, and similarly  $\chi_d^{{\scalebox{0.5}{\Circle}}}(x_0)=\chi_d^{{\scalebox{0.5}{\CIRCLE}}}(x_1) $.

\begin{align*}
   &\chi_{d_{1}}^{{\scalebox{0.5}{\CIRCLE}}}(x_0)\chi_{d_{2}}^{{\scalebox{0.5}{\CIRCLE}}}(x_0)+\chi_{d_{1}}^{{\scalebox{0.5}{\Circle}}}(x_0)\chi_{d_{2}}^{{\scalebox{0.5}{\Circle}}}(x_0)
    + \chi_{d_{1}}^{{\scalebox{0.5}{\CIRCLE}}}(x_0)\chi_{d_{2}}^{{\scalebox{0.5}{\CIRCLE}}}(x_1)+\chi_{d_{1}}^{{\scalebox{0.5}{\Circle}}}(x_0)\chi_{d_{2}}^{{\scalebox{0.5}{\Circle}}}(x_1)\\
    &= \chi_{d_{1}}^{{\scalebox{0.5}{\CIRCLE}}}(x_0)(\chi_{d_{2}}^{{\scalebox{0.5}{\CIRCLE}}}(x_0)+\chi_{d_{2}}^{{\scalebox{0.5}{\CIRCLE}}}(x_1))
    +\chi_{d_{1}}^{{\scalebox{0.5}{\Circle}}}(x_0)(\chi_{d_{2}}^{{\scalebox{0.5}{\Circle}}}(x_0)+\chi_{d_{2}}^{{\scalebox{0.5}{\Circle}}}(x_1))\\
    &=\chi_{d_{1}}^{{\scalebox{0.5}{\CIRCLE}}}(x_0)(\chi_{d_{2}}^{{\scalebox{0.5}{\CIRCLE}}}(x_0)+\chi_{d_{2}}^{{\scalebox{0.5}{\Circle}}}(x_0))
    +\chi_{d_{1}}^{{\scalebox{0.5}{\Circle}}}(x_0)(\chi_{d_{2}}^{{\scalebox{0.5}{\Circle}}}(x_0)+\chi_{d_{2}}^{{\scalebox{0.5}{\CIRCLE}}}(x_0))\\
    &=\chi_{d_{1}}^{{\scalebox{0.5}{\CIRCLE}}}(x_0)\chi_{d_{2}}(x_0)+\chi_{d_{1}}^{{\scalebox{0.5}{\Circle}}}(x_0)\chi_{d_{2}}(x_0)\\
    &=\chi_{d_{1}}(x_0)\chi_{d_{2}}(x_0) = q^{d_{2}}(1-q^{-1})q^{d_{1}}(1-q^{-1})\\
    &= q^{d_1+d_2}(1-q^{-1})^2.
\end{align*}

We get 
\[O_{0,0}(\gamma) = \sum_{i\leq d_{\gamma_1}}\sum_{j\leq d_{\gamma_2}} q^{i+j}(1-q^{-1})^2 = q^{d_{\gamma_1}+d_{\gamma_2}},\]
if $m>0$ then 
\[O_{0,m}(\gamma) = \sum_{j\leq d_{\gamma_2}} q^{d_{\gamma_1}+m+j}(1-q^{-1})^2 = q^{d_{\gamma_1}+m+d_{\gamma_2}}(1-q^{-1}),\]
similarly  $O_{n,0} =q^{d_{\gamma_1}+d_{\gamma_2}+n}(1-q^{-1})$ when $n>0$ and lastly, when $n,m>0$ we get 
\[O_{n,m}(\gamma) = q^{n+m+d_{\gamma_1}+d_{\gamma_2}}(1-q^{-1})^2 .\]

To sum up, we have 
\[O_{n,m}(\gamma) = \left\lbrace \begin{array}{ll}
   q^{d_{\gamma_1}+d_{\gamma_2}}  & n=m=0 \\
   q^{n+m+d_{\gamma_1}+d_{\gamma_2}}(1-q^{-1})  &nm=0,\ m+n>0\\
   q^{n+m+d_{\gamma_1}+d_{\gamma_2}}(1-q^{-1})^2 & n,m>0\\
\end{array}\right..\]

We get the following.
\begin{theorem}[Hyperbolic Shalika germs]\label{th:shalikhyper}
    Assume $\gamma$ is hyperbolic. Let $\mathcal{L}_{\gamma} =\mathrm{Orb}(\gamma,\underline{\ })$. We can write
    \[\mathrm{\mathcal{L}} = A(\gamma)\mathcal{L}_{\alpha(0,0)}+B(\gamma)\mathcal{L}_{\alpha(1,0)}+C(\gamma)\mathcal{L}_{\alpha(0,1)}\]\[+D(\gamma)\mathcal{L}_{\alpha(1,1)}+E(\gamma)\mathcal{L}_{\alpha(1,\varpi)}, \]
    where 
    \begin{align*}
        A(\gamma) &= B(\gamma)= C(\gamma) = 0,\\
        D(\gamma) &= E(\gamma)=q^{d_{\gamma_1}+d_{\gamma_2}}(1-q^{-1})^2 .
    \end{align*}
\end{theorem}

\begin{rem}
    We note that in this case, the conjugacy class is equal to its stable conjugacy class, and we can verify that the Shalika germ expansion matches the one for $\mathrm{GL}_{2}(F)\times  \mathrm{GL}_{2}(F)$.
\end{rem}

\subsection{Elliptic elements}
We can compute them directly. Note that if $\gamma = (\gamma_1,\gamma_2)$  is elliptic, then its stable orbit has $2\times 2 = 4$ $\mathrm{SL}_2^2(F)$-orbits and of course always exactly one $\mathrm{GL}_2^2(F)$-orbit.

Using Proposition \ref{prop:orbitsgl22} we see that there is 
a unique orbit within the stable orbit when $F[\gamma_1]\cong F[\gamma_2]$ and $2$ otherwise.

This suffices to compute the orbital integral, since
 may decompose each orbit into $\mathrm{SL}_2^2(F)$ orbits and add the corresponding 
 contribution using \ref{sec:sl2}.

\begin{ex}
\label{ex:compunram}
If both $F(\gamma_1)$ and $F(\gamma_2)$ are distinct unramified field extensions, then we get $2$ $G$-orbits, each made of two $\mathrm{SL}_2^2(F)$-orbits.

We can group such orbits and sum their corresponding $\mathrm{SL}_2^2(F)$-integrals to compute the integral over $G$ with compatible measures.

This depends on whether the vertex corresponding to $\gamma$ is monochromatic or not. We get: 

\begin{itemize}
    \item Monochrome pair with $n=m=0$:
    \begin{align*}O_{0,0}(\gamma) &=\CIRCLE\times \CIRCLE+ \Circle\times \Circle \\ 
    &=\left\lbrace\begin{array}{ll}
\frac{(q^{d_{\gamma_1}}-1)(q^{d_{\gamma_2}}-1)+(q^{d_{\gamma_1}+1}-1)(q^{d_{\gamma_2}+1}-1)}{(q-1)^2}   & \text{if } d_{\gamma_1} \equiv d_{\gamma_2} \pmod2 \\ & \\
 \frac{(q^{d_{\gamma_1}+1}-1)(q^{d_{\gamma_2}}-1)+(q^{d_{\gamma_1}}-1)(q^{d_{\gamma_2}+1}-1)}{(q-1)^2}   &  \text{else} \\ 
\end{array}\right.\\\end{align*}
    \item Bichrome pair with $n=m=0$:
    \begin{align*}O_{0,0}(\gamma) &=\CIRCLE\times \CIRCLE+ \Circle\times \Circle \\ 
    &=\left\lbrace\begin{array}{ll}
\frac{(q^{d_{\gamma_1}}-1)(q^{d_{\gamma_2}}-1)+(q^{d_{\gamma_1}+1}-1)(q^{d_{\gamma_2}+1}-1)}{(q-1)^2}   & \text{if } d_{\gamma_1} \not\equiv d_{\gamma_2} \pmod2 \\ & \\
 \frac{(q^{d_{\gamma_1}+1}-1)(q^{d_{\gamma_2}}-1)+(q^{d_{\gamma_1}}-1)(q^{d_{\gamma_2}+1}-1)}{(q-1)^2}   &  \text{else} \\ 
\end{array}\right.\\\end{align*}
    \item Monochrome pair with $n,m>0$: 
    \begin{align*}O_{n,m}(\gamma) &=\CIRCLE\times \CIRCLE+ \Circle\times \Circle \\ 
    &=\left\lbrace\begin{array}{ll}
q^{n+m+d_{\gamma_1}+d_{\gamma_2}}(1+q^{-1})^2    & \text{if } d_{\gamma_1} +n\equiv d_{\gamma_2} +m\pmod2 \\ & \\
  0   &  \text{else} \\ 
\end{array}\right.\\\end{align*}
 \item Bichrome pair with $n,m>0$: 
    \begin{align*}O_{n,m}(\gamma) &=\CIRCLE\times \Circle+ \CIRCLE\times \Circle \\ 
    &=\left\lbrace\begin{array}{ll}
q^{n+m+d_{\gamma_1}+d_{\gamma_2}}(1+q^{-1})^2    & \text{if } d_{\gamma_1} +n\not\equiv d_{\gamma_2} +m\pmod2 \\ & \\
  0   &  \text{else} \\ 
\end{array}\right.\\\end{align*}
\end{itemize}

\end{ex}

\begin{ex}
\label{ex:compram}
We now proceed with the same computation when  $F(\gamma_1)$ and $F(\gamma_2)$ are distinct ramified field extensions. We assume odd residue characteristic for convenience.

This time the computations are simpler and we have
\begin{align*}O_{0,0}(\gamma)   &=
\frac{(q^{d_{\gamma_1+1/2}}-1)(q^{d_{\gamma_2+1/2}}-1)}{(q-1)^2},
\end{align*}
and if $m,n>0$ we get
\begin{align*}O_{0,0}(\gamma)   &=
 q^{d_{n+m+\gamma_1+\gamma_2+1}}
.\end{align*}
\end{ex}

\subsubsection{Elliptic Shalika germs}
We can obtain a general formula by describing the set $Y^\gamma$ when $\gamma$ is elliptic. Using $Y= X\times X$ we get
\[Y^\gamma = X^{\gamma_1}\times X^{\gamma_2}= X^{\gamma_1}_{\scalebox{0.5}{\CIRCLE}}\times X^{\gamma_2}_{\scalebox{0.5}{\CIRCLE}}\sqcup X^{\gamma_1}_{\scalebox{0.5}{\Circle}}\times X^{\gamma_2}_{\scalebox{0.5}{\Circle}}\sqcup  X^{\gamma_1}_{\scalebox{0.5}{\CIRCLE}}\times X^{\gamma_2}_{\scalebox{0.5}{\Circle}}\sqcup X^{\gamma_1}_{\scalebox{0.5}{\Circle}}\times X^{\gamma_2}_{\scalebox{0.5}{\CIRCLE}},\]
so the monochrome vertices of $Y^\gamma$ can be described as 
\[Y^\gamma_{\text{mono}} = X^{\gamma_1}_{\scalebox{0.5}{\CIRCLE}}\times X^{\gamma_2}_{\scalebox{0.5}{\CIRCLE}}\sqcup X^{\gamma_1}_{\scalebox{0.5}{\Circle}}\times X^{\gamma_2}_{\scalebox{0.5}{\Circle}}.\]

Let $\mathfrak{Y}^{\gamma}_{{\scalebox{0.5}{\CIRCLE\CIRCLE}}}, \mathfrak{Y}^{\gamma}_{{\scalebox{0.5}{\Circle\Circle}}}, \mathfrak{Y}^{\gamma}_{{\scalebox{0.5}{\Circle\CIRCLE}}}, \mathfrak{Y}^{\gamma}_{{\scalebox{0.5}{\CIRCLE\Circle}}}$ denote $|X^{\gamma_1}_{\scalebox{0.5}{\CIRCLE}}|| X^{\gamma_2}_{\scalebox{0.5}{\CIRCLE}}|$, $|X^{\gamma_1}_{\scalebox{0.5}{\Circle}}|| X^{\gamma_2}_{\scalebox{0.5}{\Circle}}|$,  $|X^{\gamma_1}_{\scalebox{0.5}{\Circle}}|| X^{\gamma_2}_{\scalebox{0.5}{\CIRCLE}}|$, and $|X^{\gamma_1}_{\scalebox{0.5}{\CIRCLE}}|| X^{\gamma_2}_{\scalebox{0.5}{\Circle}}|$ respectively. Moreover, write \[\mathfrak{Y}^{\gamma}_{\mathrm{mono}} =  \mathfrak{Y}^{\gamma}_{{\scalebox{0.5}{\CIRCLE\CIRCLE}}}+\mathfrak{Y}^{\gamma}_{{\scalebox{0.5}{\Circle\Circle}}}, \ \mathfrak{Y}^{\gamma}_{\mathrm{bi}} =  \mathfrak{Y}^{\gamma}_{{\scalebox{0.5}{\CIRCLE\Circle}}}+\mathfrak{Y}^{\gamma}_{{\scalebox{0.5}{\Circle\CIRCLE}}},\ \sigx= |X^{\gamma_1}|+|X^{\gamma_2}|.\] Multiplying counts from  Proposition \ref{prop:coloredawayfromconvex} we get 
\begin{prop}\label{prop:elliporb} We have
\[O_{n,m}(\gamma_1,\gamma_2) = \left\lbrace\begin{array}{ll}
 \ymono &  n=m=0 \\ &\\
    q^{n+m-2}\left(q^{2}\ymono+q(\sigx - 2\ybi)\right. & n,m>0\\ \left.+\mathfrak{Y}^{\gamma}_{\mathrm{mono}}-\sigx+2\right) &n\equiv m\pmod 2\\&\\    q^{n+m-2}\left(q^{2}\ybi+q(\sigx - 2\ymono)\right.  & n,m>0\\\left.+\ybi-\sigx+2\right) &n\not\equiv m\pmod 2\\&\\
  q^{m-1}(q\ymono-\ybi+|X^{\gamma_1}|)    & n=0,\ m>0,\ m\equiv0\pmod2\\
  &\\
   q^{m-1}(q\ybi-\ymono+|X^{\gamma_1}|)    & n=0,\ m>0,\ m\equiv1\pmod2\\&\\
    O_{0,n}(\gamma_2,\gamma_1) & n>0,\ m=0\\ 
\end{array}\right. .\]
\end{prop}
We get the corresponding Shalika germ expansion. 

\begin{theorem}[Elliptic Shalika germs]\label{th:shalik-elliptic}
    Assume $\gamma$ is elliptic. Let $\mathcal{L}_{\gamma} =\mathrm{Orb}(\gamma,\underline{\ })$. We can write
    \[\mathrm{\mathcal{L}} = A(\gamma)\mathcal{L}_{\alpha(0,0)}+B(\gamma)\mathcal{L}_{\alpha(1,0)}+C(\gamma)\mathcal{L}_{\alpha(0,1)}\]\[+D(\gamma)\mathcal{L}_{\alpha(1,1)}+E(\gamma)\mathcal{L}_{\alpha(1,\varpi)}, \]
    where 
    \begin{align*}
        A(\gamma) &= \frac{2}{(1-q)^2},\\
        B(\gamma) &= \frac{2q^{-1}}{1-q}-q^{-1}|X^{\gamma_1}|,\\
        C(\gamma) &= \frac{2q^{-1}}{1-q}-q^{-1}|X^{\gamma_2}|,\\
        D(\gamma)&=q^{-2}\left[(1+q^2)\ymono + (q-1)\sigx-2q\ybi+2\right],\\
        E(\gamma)&=q^{-2}\left[(1+q^2)\ybi + (q-1)\sigx-2q\ymono+2\right].
    \end{align*}
\end{theorem}
\begin{proof}
    We will not carry out the full verification but outline the process. The system of equations of Proposition \ref{prop:elliporb} become
    \begin{align*}
        A + \frac{1}{1-q^{-1}}(B+C) + \frac{1+q^{-2}}{(1-q^{-2})^2}D+\frac{2q^{-1}}{(1-q^{-2})^2}E = \ymono,\\
         q^nB+\frac{q^n}{1-q^{-2}}(D+q^{-1}E)= q^{n-1}(q\ymono-\ybi+|X^{\gamma_2}|),\\
         q^nB+\frac{q^{n}}{1-q^{-2}}(q^{-1}D+E) = q^{n-1}(q\ybi-\ymono+|X^{\gamma_2}|),\\
         q^mC+\frac{q^m}{1-q^{-2}}(D+q^{-1}E) = q^{m-1}(q\ymono-\ybi+|X^{\gamma_1}|),\\
         q^mC+\frac{q^{m}}{1-q^{-2}}(q^{-1}D+E) = q^{m-1}(q\ybi-\ymono+|X^{\gamma_1}|),\\
         q^{n+m}D =  q^{n+m-2}\left(q^{2}\ymono+q(\sigx - 2\ybi)+\mathfrak{Y}^{\gamma}_{\mathrm{mono}}-\sigx+2\right)\\
         q^{n+m}E =  q^{n+m-2}\left(q^{2}\ymono+q(\sigx - 2\ybi)+\mathfrak{Y}^{\gamma}_{\mathrm{mono}}-\sigx+2\right)\\
    \end{align*}

    Starting with the bottom, we solve the system by substitution and get the desired results.
\end{proof}


\subsection{Mixed setting}

In this last part, we assume that $\gamma = (\gamma_1,\gamma_2)$ where $\gamma_1$ is hyperbolic and $\gamma_2$ is elliptic. In other words, $F[\gamma]\cong F^2\oplus F[\gamma_2]$ and $F[\gamma_2]$ is a field.

\begin{ex}
    An example of such an element over $F = \mathbb{Q}_3$ is 
    \[\gamma = \left(\begin{pmatrix}2&0\\ 0&1/2\end{pmatrix},\begin{pmatrix}0&-1\\ 1&0\end{pmatrix}\right).\]
\end{ex}

The shape of the centralizer here is not as straightforward. The set of  rational points is 
\[\{(x,y,z)\in (F\times F\times F[\gamma_2])^\times \ :\ N_{F[\gamma_2]/F}(z) = xy\}.\]

Note that this torus is isomorphic to $F[\gamma_2]\times F^\times$ via 
\[(a, b)\mapsto (a, N_{F[\gamma_2]/F}(a)b^{-1},b),\]
which inverse morphism is 
\[(x,y,z)\mapsto (x, z).\]

Therefore, as in the elliptic case, we can factor out the volume of $F[\gamma_2]^\times$ modulo center and will determine $O_{0,0}(\gamma)$ counting monochrome vertices of the form $(v,w)\in X\times X$ where $w\in X^{\gamma_2}$ and $v$ spans over a set of representatives of $F^\times$-orbits pf $X^{\gamma_1}$ (identifying $F^\times$ with the diagonal torus of $\mathrm{SL}_2(F)$).

As in the case of $\mathrm{SL}_2(F)$, we see that $X^{\gamma_1}$ is the set of vertices at distance at most $d_{\gamma_1}$ from the fundamental apartment and we can pick a set of representatives for the action of the diagonal torus as the vertices whose closest vertex on the apartment is either $x_0$ or $x_1$. We use the notation $\chi$ as in \S \ref{chi:notation}.

\begin{align*}
    O_{0,0}(\gamma) &= \sum_{i\leq d_{\gamma_1}}\left(\chi_{i}^{\scalebox{0.5}{\CIRCLE}}(x_0)|X^{\gamma_2}_{\scalebox{0.5}{\CIRCLE}}| +\chi_{i}^{\scalebox{0.5}{\Circle}}(x_0)|X^{\gamma_2}_{\scalebox{0.5}{\Circle}}| \right)+ \sum_{i\leq d_{\gamma_1}}\left(\chi_{i}^{\scalebox{0.5}{\CIRCLE}}(x_1)|X^{\gamma_2}_{\scalebox{0.5}{\CIRCLE}}| + \chi_{i}^{\scalebox{0.5}{\Circle}}(x_1)|X^{\gamma_2}_{\scalebox{0.5}{\Circle}}| \right)\\
    &=(|X^{\gamma_2}_{\scalebox{0.5}{\CIRCLE}}|+|X^{\gamma_2}_{\scalebox{0.5}{\Circle}}|) \sum_{i\leq d_{\gamma_1}}\left(\chi_{i}^{\scalebox{0.5}{\CIRCLE}}(x_0)+\chi_{i}^{\scalebox{0.5}{\Circle}}(x_0)\right)\\
    &=|X^{\gamma_2}|\sum_{i\leq d_{\gamma_1}}\chi_i(x_0)\\
    &= |X^{\gamma_2}|q^{d_{\gamma_1}}.
    \end{align*}

The same computation shows (using equation (\ref{prop:pointfromconvex})) 
\begin{align*}O_{n,m}(\gamma) &= \left\lbrace\begin{array}{ll}
 q^{d_{\gamma_1}} |X^{\gamma_2}| & n=m=0 \\
 q^{d_{\gamma_1}} q^{m-1}\left((q-1)|X^{\gamma_2}|+2\right)   &  n=0,\ m>0\\
 q^{n+d_{\gamma_1}}(1+q^{-1})|X^{\gamma_2}| &n>0,\ m=0\\
 q^{n+d_{\gamma_1}}(1+q^{-1})q^{m-1}\left((q-1)|X^{\gamma_2}|+2\right)&n,m>0
\end{array}\right.\\
&=q^{d_{\gamma_1}+n+m}\left\lbrace\begin{array}{ll}
   |X^{\gamma_2}| & n=m=0 \\
   \left((1-q^{-1})|X^{\gamma_2}|+2q^{-1}\right)   &  n=0,\ m>0\\
  (1-q^{-1})|X^{\gamma_2}| &n>0,\ m=0\\
  (1-q^{-1})\left((1-q^{-1})|X^{\gamma_2}|+2q^{-1}\right)&n,m>0
\end{array}\right. .\end{align*}

Note that $|X^{\gamma_2}|$ has been computed explicitly in \S \ref{sl2unram} and \S \ref{sl2ram}. We will however not need it for the computation of Shalika germs.
\begin{theorem}[Mixed Shalika germs]\label{th:mixedshalika}
    Assume $\gamma_1$ is hyperbolic and $\gamma_2$ is elliptic. Let $\mathcal{L}_{\gamma} =\mathrm{Orb}(\gamma,\underline{\ })$. We can write
    \[\mathrm{\mathcal{L}}_\gamma = A(\gamma)\mathcal{L}_{\alpha(0,0)}+B(\gamma)\mathcal{L}_{\alpha(1,0)}+C(\gamma)\mathcal{L}_{\alpha(0,1)}\]\[+D(\gamma)\mathcal{L}_{\alpha(1,1)}+E(\gamma)\mathcal{L}_{\alpha(1,\varpi)}, \]
    where 
    \begin{align*}
        A(\gamma) &=C(\gamma) = 0,\\
        B(\gamma) &= -2q^{d_{\gamma_1}-1},\\
        D(\gamma)&=E(\gamma) = q^{d_{\gamma_1}}(1-q^{-1})\left((1-q^{-1})|X^{\gamma_2}|+2q^{-1}\right).
    \end{align*}
\end{theorem}

\begin{rem}
    Again, this matches the product of two germ expansions for $\mathrm{GL}_2(F)$. Indeed, the hyperbolic germ expansion of $\mathrm{GL}_2(F)$ is 
    \[(1-q^{-1})q^{d_{\gamma_1}}\mathcal{L}_{\alpha(1)}\]
    whereas the elliptic one is 
    \[-\frac{2}{q-1}\mathcal{L}_{\alpha(0)} + (2q^{-1}+(1-q^{-1})|X^{\gamma_2}|)\mathcal{L}_{\alpha(1)}.\]
Taking the product of those two expansions, we get that the coefficient of $\mathcal{L}_{\alpha(0,1)}$ is 
\[-\frac{2}{q-1}(1-q^{-1})q^{d_{\gamma_1}} = -2q^{d_{\gamma_1}-1} = B(\gamma),\]
and the coefficient of the regular orbit $\mathcal{L}_{\alpha(1,1)}+\mathcal{L}_{\alpha(1,\varpi)}$ is 
\[(2q^{-1}+(1-q^{-1})|X^{\gamma_2}|)(1-q^{-1})q^{d_{\gamma_1}} = D(\gamma).\]

 Using Proposition \ref{prop:orbitsgl22}, the stable conjugacy class of $\gamma$ only contains the one conjugacy class, and the orbital integral matches the one of $\mathrm{GL}_2(F)\times \mathrm{GL}_2(F)$ as expected.\end{rem}

\section{Comparison for non-elliptic elements}\label{sec:parabdescent}
In this section, we show that the comparison is very explicit as long as the centralizer of $\gamma$ is 
not compact modulo center. Moreover, the comparison holds in the distinguished subgroup of $\mathrm{GSp}_{2n}$ 
and not only $\mathrm{GSp}_4$.

Let $\gamma = (\gamma_1,\cdots,\gamma_n)\in G$ be a regular semisimple element. 
We have $F[\gamma] = \bigoplus_{i=1}^n F[\gamma_i]$. 
In this section we will assume that at least one $F[\gamma_i]$ is split, i.e. 
$F[\gamma_i]\cong F\oplus F$. For convenience we will assume that $F[\gamma_1]$ is split. 
 
In that case, let $P = \mathrm{Stab}_{\mathrm{GSp}(V)}(F\cdot e_1)$ and $P_G = P\cap G$.

\begin{prop}
    We have that $P$ (resp. $P_G$) is a parabolic subgroup of $\mathrm{GSp}(V)$ (resp. $G$) and the two Levi subgroups are equal. 
\end{prop}
\begin{proof}
    As a stabilizer of a line, which is automatically isotropic, $P$  is a parabolic group. Same for $P_G$, which is the parabolic consisting of upper-triangular matrices in the first component, and the full group in the others.

    We can always make a choice of positive roots such that a wall of the positive Weyl chamber is along a long root (assume the one whose rootspace spans the first $\mathrm{GL}_2$). All the other long roots are perpendicular to this wall, but none of the short roots are perpendicular to a long root, so the Levi is spanned by a subset of the long roots and hence is contained in $G$.
\end{proof}

We can therefore write the Levi decompositions of these parabolics as $P = MN$ and $P_G = M N_G$ where $N_G = N\cap G$.

Note that to each parabolic subgroup we associate a modular character $\delta^{G}_P: x\mapsto \det(\mathrm{Ad}(x): \mathrm{Lie}(P)/\mathrm{Lie}(M))$ measuring the lack of unimodularity of $P$. Diagonalizing the action of $x$ on root spaces of $\mathrm{Lie}(G)$, we get contribution for all roots $\alpha$ such that the corresponding root space is in $\mathrm{Lie}(P)$ and the root space corresponding to $-\alpha$ does not. We therefore have 
\[\delta_P^G = \prod_{\alpha\in \Phi^+_G\setminus\Phi^+_P}\alpha,\]
where $\Phi^+$ is the set of positive roots corresponding to $N$ and $\Phi^+_P$ the positive roots whose rootspace is contained in $\mathrm{Lie}(P)$. 

\begin{lemma}\label{lem:modulusone}
    If $x\in K$  then $|\delta_P^G(x)| = 1$. The same holds for $\delta_{P_G}(x)$ if $x\in K\cap G$.
\end{lemma}
\begin{proof}
    If $x\in K$ then $\alpha(x)$ is a unit of $\mathcal{O}_F$ for all $\alpha$, hence $\delta_P^G(x)$ is a unit. 
\end{proof}

\begin{rem}
    Alternatively, one has 
    \[|\delta_{P}(x)| = \frac{[x^{-1}Kx : x^{-1}Kx\cap K]}{[K: x^{-1}Kx\cap K]},\]
    which also gives the desired result.
\end{rem}
\begin{prop}[{\cite[Lemma 16.3, p.451]{kott_bible}} ]\label{prop:parabdescent}
   Assume $\gamma\in G$ is regular semisimple with  $\gamma_1$  hyperbolic, picking measures on $\mathrm{GSp}(V)$ (resp. $G$) giving $K$ (resp. $K\cap G$) measure $1$, we have 
    \[\int_{G_\gamma\backslash\mathrm{GSp}(V)}\mathds{1}_K(g^{-1}\gamma g)\ \mathrm{d}g = \left|D(\gamma)\right|^{-1/2}\int_{G_\gamma\backslash G}\mathds{1}_{K}(h^{-1}\gamma h)\ \mathrm{d}h,\]
    where $D(\gamma) = \det\left(1-\mathrm{Ad}_\gamma:\mathfrak{gsp}(V)/\mathrm{Lie}(G)\right)$.
\end{prop}
\begin{proof}
    The notes \cite{kott_bible} states and proves this result in the Lie algebra setting. 
    For the sake of completeness, we write it for integrals over orbits of the elements.

    We can assume $\gamma_1$ is diagonal, and therefore $G_\gamma\subset M$. The idea of the result, is to use Parabolic descent, both from $\mathrm{GSp}(V)$ and $G$ to $M$ and conclude with 
    \[ |D_{\mathrm{GSp}(V)}(\gamma)|^{1/2}\int_{G_\gamma\backslash \mathrm{GSp}(V)} \!\!\!\!\!\!\!\!\!\!\!\!\!\!\!\!\!\!\cdots\qquad  =|D_M(\gamma)|^{1/2}\int_{G_\gamma\backslash M}  \!\!\!\!\!\!\!\!\!\!\!\cdots\quad  = |D_G(\gamma)|^{1/2}\int_{G_\gamma\backslash G}   \!\!\!\!\!\!\!\!\!\!\cdots\ ,\]
    where we use the notation $D_{\mathbf{G}(F)}(\gamma) = \det(1-\mathrm{Ad}_\gamma : \mathrm{Lie}(\mathbf{G}(F)))$. We outline now the general process of parabolic descent for $G$, but $G$ can be replaced with $\mathrm{GSp}(V)$ throughout.

    By Iwasawa decomposition, we can write $G = MN_GK_G$. We give $K_G$ measure $1$ and pick a measure on $N_G$ giving $N_G\cap K_G$ measure $1$. We get 
    \begin{align*}\int_{G_\gamma\backslash G}\mathds{1}_K(h^{-1}\gamma h)\ \mathrm{d}h &= \int_{G_\gamma \backslash M}\int_{N_G}\int_{K_G}\mathds{1}_{K}(k^{-1}n^{-1}m^{-1}\gamma m n k)\ \mathrm{d}k\mathrm{d}n\mathrm{d}m\\
    &= \int_{G_\gamma \backslash M}\int_{N_G} \mathds{1}_{K}(n^{-1}m^{-1}\gamma m n )\  \mathrm{d}n\mathrm{d}m .\end{align*}
We now ``undo'' the conjugation of $n$. Write $g = m^{-1}\gamma m$. We want to write  $n^{-1} g n = g\varphi(n)$, where 
\[\varphi(n) =\underset{\in N_G}{\underbrace{g^{-1}n^{-1} g}} n \in N_G.\]

Note that $\varphi(N_G) = N_G$. Picking $n = 1+\varepsilon$ with $\varepsilon\in \mathrm{Lie}(N_G)$, we have 
\[\varphi(1+\varepsilon) = g^{-1}(1-\varepsilon+O(\varepsilon^2))g(1+\varepsilon) = 1+\varepsilon-\mathrm{Ad}(g)(\varepsilon) + O(\varepsilon^2),\]
hence $\mathrm{d}_1(\varphi)(\varepsilon) = (1-\mathrm{Ad}(g))\varepsilon$. This lets us compute the Jacobian. To further simplify it, first note that the determinant of $1-\mathrm{Ad}(g)$ is the same as the determinant of $1-\mathrm{Ad}(\gamma)$. Secondly, the determinant of $1-\mathrm{Ad}(\gamma)$ on the opposite nilpotent to $\mathrm{Lie}(N_G)$ differs by a factor of $(-1)^{\dim(N_G)}$, hence by the decomposition $\mathrm{Lie}(G) = \mathrm{Lie}(M)\oplus \mathrm{Lie}(N_G)\oplus \mathrm{Lie}(\overline{N_G})$, we have that 
\[|\det(1-\mathrm{Ad}(\gamma): \mathrm{Lie}(N_G))|^2 =|\det(1-\mathrm{Ad}(\gamma): \mathrm{Lie}(G)/\mathrm{Lie}(M))| = \left|\frac{D_G(\gamma)}{D_M(\gamma)}\right|.\]
We can now compute 

\begin{align}\int_{N_G= \varphi(N_G)}\mathds{1}_K(n^{-1}gn)\ \mathrm{d}n
&=  \left|\frac{D_G(\gamma)}{D_M(\gamma)}\right|^{-1/2}\int_{N_G}\mathds{1}_K(gn)\ \mathrm{d}(g^{-1}ng)\\ \label{eq:onekmn}
&= \left|\frac{D_G(\gamma)}{D_M(\gamma)}\right|^{-1/2}\int_{N_G}|\delta_{P_G}(g)|^{1/2}\mathds{1}_K(g)\mathds{1}_K(n)\ \mathrm{d}(n)\\
&= \left|\frac{D_G(\gamma)}{D_M(\gamma)}\right|^{-1/2} |\delta_{P_G}(g)|^{1/2}\mathds{1}_K(g),\end{align}
where $\delta_{P_G}(g) = \det(\mathrm{Ad}(g): \mathrm{Lie}(G)/\mathrm{Lie}(M))$. We also used that an element of $x = P_G = MN_G$ can be decomposed as $x_Mx_{N_G}$ with $x_M\in N$ and $x_{N_G}\in N_G$ in a unique way, and moreover belongs to $K$ if and only if both $x_M$ and $x_{N_G}$ are in $K$. \\

By Lemma \ref{lem:modulusone}, if $g\in K$ then $|\delta_{P_G}(g)| = 1$, therefore we finally have 
\[\int_{G_\gamma\backslash G}\mathds{1}_K(h^{-1}\gamma h)\ \mathrm{d}h =\left|\frac{D_G(\gamma)}{D_M(\gamma)}\right|^{-1/2} \int_{G_\gamma \backslash M} \mathds{1}_K(m^{-1}\gamma m)\ \mathrm{d}(m),\]
as desired.
\end{proof}

\begin{prop}\label{prop:parabdescent2}
    Similarly, letting \[a = \left(\pmat{\varpi^{a_1}&0\\0&1},\pmat{\varpi^{a_{2}}&0\\0&\varpi^{a_1-a_2}},\cdots,\pmat{\varpi^{a_{n}}&0\\0&\varpi^{a_1-a_n}} \right)\in G,\]
    we get
        \[\int_{G_\gamma\backslash\mathrm{GSp}(V)}\mathds{1}_{KaK}(g^{-1}\gamma g)\ \mathrm{d}g = \left|D(\gamma)\right|^{-1/2}\left|\frac{\delta_P^{\mathrm{GSp}(V)}(a)}{\delta_P^G(a)}\right|^{1/2}\int_{G_\gamma\backslash G}\mathds{1}_{KaK}(h^{-1}\gamma h)\ \mathrm{d}h.\]
\end{prop}
\begin{proof}
    The proof works the same verbatim, where we replace equation (\ref{eq:onekmn}) by 
    \[\mathds{1}_{KaK}(gn) =\mathds{1}_{KaK}(g)\mathds{1}_K(n).\]
    This is because $\det(gn) = \det(g)$ and so if $gn$ is in  $KaK$ then $|\det(g)| = |\det(a)| = |\varpi|^{na_1}$.
\end{proof} 
Note that $\delta_P^{\mathrm{GSp}(V)}$ and ${\delta_P^G}$ only differ by the short roots of $\mathrm{GSp}_{2n}$, therefore,  is the product of the short roots, so using the notations above we get $\lambda  = a_1$ and therefore 
\[\frac{\delta_P^{\mathrm{GSp}(V)}(a)}{\delta_P^G(a)} = \frac{\prod_{1\leq i\leq j\leq n}\varpi^{2a_i-\lambda}\prod_{1\leq i<j\leq n}\varpi^{a_i-a_j}}{\prod_{1\leq i\leq n}\varpi^{2a_i-\lambda}} = \prod_{1\leq i<j\leq  n} \varpi^{3a_i-a_j-\lambda}.\]
Consequently, we get the following.

\begin{cor}\label{cor:specialdescent} Using the same notation as in Proposition \ref{prop:parabdescent2}, if we have $a = \pmat{\varpi&0\\0&1}^{\oplus n}$ then 
\[\int_{G_\gamma\backslash\mathrm{GSp}(V)}\mathds{1}_{KaK}(g^{-1}\gamma g)\ \mathrm{d}g = \left|D(\gamma)\right|^{-1/2}q^{-n(n-1)/4}\int_{G_\gamma\backslash G}\mathds{1}_{KaK}(h^{-1}\gamma h)\ \mathrm{d}h.\]
\end{cor}
\begin{proof}
In this case we have $a_i=a_j = 1$ hence $\displaystyle\frac{\delta_P^{\mathrm{GSp}(V)}(a)}{\delta_P^G(a)}  = \prod_{1\leq i<j\leq n} \varpi^{3a_i-a_j-a_1} =\prod_{1\leq i<j\leq n} \varpi $ as desired. 
\end{proof}

We can make everything explicit by using the root systems which are described in the next section. The value of $D(\gamma)$ is $\prod_{\alpha}(\gamma)$ as $\alpha$ ranges over short roots of $\mathrm{GSp}_{2n}$.

\begin{ex}
When $n=2$, let $\gamma = (\gamma_1,\gamma_2)\in G$ where $t_i, \lambda t_i^{-1}$ are eigenvalues of $\gamma_i$, we get 
\begin{equation}\label{eq:weyldisc} D(\gamma) = \left(1-\frac{t_2}{t_1}\right)\left(1-\frac{t_1}{t_2}\right)\left(1-\frac{\lambda}{t_1t_2}\right)\left(1-\frac{t_1t_2}{\lambda}\right)= \frac{(t_1-t_2)^2(\lambda-t_1t_2)^2}{\lambda t_1^2t_2^2}.\end{equation}

For example, if $\gamma\in \mathrm{Sp}_4$  (equivalently, $\lambda = 1$), such that $t_1 = 1 + a\varpi^k$ and $t_2 = 1+b\varpi^k$. We get 
\begin{align}\left|D(\gamma)\right| &= \left|\frac{(a-b)^2\varpi^{2k}+ (a+b)^2\varpi^{2k}+2ab(a+b)\varpi^{3k}+a^2b^2\varpi^{4k}}{t_1^2t_2^2}\right|\\
& = \left|2(a^2+b^2)\varpi^{2k}+2ab(a+b)\varpi^{3k}+a^2b^2\varpi^{4k}\right|.\end{align}
In particular, when $a=1$, $b$ is a nonsquare residue (in odd residue characteristic), then $|D(\gamma)| = q^{2k}$. 
\end{ex}

\begin{rem}\label{rem:idealworld}
    Let $E_1,E_2$ be two nonisogenous ordinary elliptic curves defined over a finite field $\mathbb{F}_p$, where $p$ is prime. 
    
    In this case, the splitting field of the Frobenius element splits over a CM-biquadratic extension of $\mathbb{Q}$ of Galois group isomorphic to $\mathbb{Z}/2\mathbb{Z}\times \mathbb{Z}/2\mathbb{Z}$.
    
    By Chebotarev Density Theorem we know that exactly $25\%$ of primes will remain inert over their corresponding completion, hence the descent of Proposition \ref{prop:parabdescent} will be correct 
     for only $75\%$ of orbital integrals involved in the Langlands--Kottwitz formula.\\

     However, in Gekeler's work \cite{gekeler}, the heuristic used to compute the size of the isogeny class of an elliptic curve fails
     at infinitely many places, but their product converges to the correct number.

     Let $a,q a^{-1}$ (resp.  $b, p b^{-1}$) be the eigenvalues of the Frobenius element of $E_1$ (resp. $E_2$). Let $I$ be the moduli space corresponding to the isogeny class of $E_1\times E_2$ as product of elliptic curves and $\tilde{I}$ be the moduli space of the isogeny class of $E_1\times E_2$ as principally polarized abelian varieties.

     If we were to use a similar heuristic and assume that we can use Propositions \ref{prop:parabdescent}  and \ref{prop:parabdescent2}
     at all places, we would get  
    
    \begin{equation}\frac{|\tilde{I}|}{|I|}  = \left|\frac{(b-a)(p-ab)}{pab}\right| = \sqrt{\frac{|\mathrm{Ext}^1(E_1,E_2)|}{p}}.\end{equation}
 
The last equality comes from Theorem \ref{th:milne}. 

For the first equality, let $v = \mathrm{vol}(G_\gamma(\mathbb{Q})\backslash G_\gamma(\mathbb{A}))$. Since $E_1,E_2$ are ordinary, we know that their Frobenius is conjugate to the matrix $\pmat{p&&&\\&p&&\\&&1&\\&&&1}$ in $\mathbb{Q}_p$. 
Langlands--Kottwitz formula yields  
\begin{align*}\tilde{I} &=v \mathrm{Orb}(\gamma,\mathds{1}_{\mathrm{GSp}_4(\mathbb{Z}_p)\mathrm{diag}(p,p,1,1)\mathrm{GSp}_4(\mathbb{Z}_p)}) \prod_{\ell\neq p}\mathrm{Orb}(\gamma,\mathds{1}_{\mathrm{GSp}_4(\mathbb{Z}_\ell)})\\
&= v\left|D(\gamma)\right|_p^{-1/2}\left|\frac{\delta_P^{\mathrm{GSp}(V)}(a)}{\delta_P^G(a)}\right|_p^{1/2}\mathrm{Orb}(\gamma,\mathds{1}_{G(\mathbb{Z}_p)\mathrm{diag}(p,p,1,1)G(\mathbb{Z}_p)}) \prod_{\ell\neq p}\left|D(\gamma)\right|_\ell^{-1/2}\mathrm{Orb}(\gamma,\mathds{1}_{G(\mathbb{Z}_\ell)}),\\
&= |D(\gamma)|^{1/2}_{\infty} p^{-1/2}  I.
 \end{align*}
Use  equation (\ref{eq:weyldisc}) to conclude.
\end{rem}

\begin{lemma} \label{lem:manyellcurvesweyldisc}In general, write $\gamma = (\gamma_1,\dots, \gamma_n)$ where $\gamma_i$ has eigenvalues $t_i, \lambda t_i^{-1}$ 
\begin{align*}
D(\gamma) &= \prod_{i<j}\left(1-\frac{t_i}{t_j}\right)\left(1-\frac{t_j}{t_i}\right)\left(1-\frac{t_it_j}{\lambda}\right)\left(1-\frac{\lambda}{t_it_j}\right),\\
&=  \frac{\prod_{1\leq i<j\leq n}(t_i-t_j)^2(\lambda-t_it_j)^2}{\lambda^{n(n-1)/2}\prod_{1\leq i\leq  n}t_i^{2(n-1)}}.
\end{align*}
In particular, if $\gamma\in \mathrm{GSp}_{2n}(\mathcal{O}_F)$ then 
\[|D(\gamma)|^{1/2} = \prod_{1\leq i<j\leq n} |t_i-t_j|\ |\lambda-t_it_j|.\]
\end{lemma}

\begin{rem}\label{rem:idealworld2}
Continuing Remark \ref{rem:idealworld}, assume we have $n$ pairwise nonisogenous ordinary elliptic curves $E_1,\dots, E_n$ over $\mathbb{F}_p$.  The Galois group of the splitting field of the corresponding 
Frobenius over $\mathbb{Q}$ is of type $(\mathbb{Z}/2\mathbb{Z})^{n}$ and therefore the proportion of places Proposition \ref{prop:parabdescent} can be applied is $\frac{2^{n}-1}{2^n}$. Therefore, as $n$ gets large, this 
descent case is by far the most common one.

Let $a_i,p a_i^{-1}$ be eigenvalues of the elliptic curve $E_i$. In that case, assuming that the descent formula works on average, Lemma \ref{lem:manyellcurvesweyldisc} and Corollary \ref{cor:specialdescent} would imply 
\begin{equation}\left|\frac{\tilde{I}}{I}\right| = |D(\gamma)|_{\infty}^{1/2} p^{-n(n-1)/4} =  \frac{\prod_{1\leq i<j\leq n}|t_i-t_j|\ |p-t_it_j|}{p^{n(n-1)/2}\prod_{1\leq i\leq  n}|t_i|^{n-1}}.\end{equation}
\end{rem}

\begin{rem}
So far we have only studied orbits of elements of $G$ but in the next section we will study orbits of elements  $X\in \mathrm{Lie}(G)$. Everything so far can be however carried out the same way, replacing $d_\gamma$ by $d_X$, the valuation of the difference of eigenvalues of $X$. In that case, the corresponding choice for $D(\gamma)$ is 
$D(X) = \det\left(\mathrm{ad}(X): \mathfrak{gsp}_{2n}(F)/\mathrm{Lie}(G)\right).$ Once again, letting $X = (X_1,\dots, X_n)$ where $X_i\in \mathfrak{gl}_2(F)$ has eigenvalues  $t_i, \lambda-t_i$, we get 
\[D(X) = \prod_{i<j}(t_i-t_j)(t_j-t_i)(\lambda-t_i-t_j)(t_i-t_j-\lambda) = (t_i-t_j)^2(\lambda-(t_i+t_j))^2.\]
\end{rem}
\section{Symplectic orbital integrals for equivalued elliptic elements}\label{sec:equivalued}

This section aims to explicitly compute the orbital integral of the characteristic function of 
$\mathrm{GSp}_4(\mathcal{O}_F)$ under specific conditions. The computations are equivalent if we compute the integrals
over an orbit of the Lie algebra, so we will carry this case for ease of notations. 

We also change notations, namely we will write $G = \mathrm{GSp}_4$, 
 $\mathfrak{g} = \mathrm{Lie}(\mathrm{GSp}_4)$. The distinguished subgroup will be denoted by $H$, with associated 
 Lie algebra $\mathfrak{h}$.

\subsection{Roots and equivalued elements}
Recall that we write matrices in $G$ in the  Witt basis $\{e_1,e_2, f_1,f_2\}$.

Looking at the action of the diagonal torus $T$ of $G$ on $\mathfrak{g}$ by conjugation, 
 we can determine the corresponding root system. 

Indeed, the element  $t = \mathrm{diag}(t_1,t_2,\lambda t_1^{-1},\lambda t_2^{-1})\in T$ acts on $\mathfrak{g}$

by multiplying each entry by the image of $t$ under a root. We can therefore represent the root system as follows.

\begin{figure}[H]
    \centering
  \begin{tikzpicture}
  \fill[teal!10]  (0,0) -- (90:3cm) -- (45:3cm);
    \foreach\ang in {45,135,...,315}{
     \draw[->,blue!80!black,thick] (0,0) -- (\ang:2cm);
    }
    \foreach\ang in {90,180,...,360}{
     \draw[->,magenta!80!black,thick] (0,0) -- (\ang:3cm);
    }
    \draw[dashed,green!80!black,thick] (0,0) -- (-15:3cm);
    \draw[dashed,green!80!black,thick] (0,0) -- (165:3cm);
    \node[anchor=south west] at (2,0) {$\alpha=t_1^2/\lambda$};
    \node[anchor=south west] at (1.3,1.3) {$\alpha+\beta=t_1t_2/\lambda$};
    \node[anchor=south west] at (0,3) {$\alpha+2\beta=t_2^2/\lambda$};
    \node[anchor=south west] at (-3,1.3) {$\beta=t_1^{-1}t_2$};
  \end{tikzpicture}
 
  \[\begin{pmatrix}
      \star&{\color{blue}-\beta}&{\color{magenta}\alpha}&{\color{blue}\alpha+\beta}\\
      {\color{blue}\beta}&\star&{\color{blue}\alpha+\beta}&{\color{magenta}\alpha+2\beta}\\
      {\color{magenta}-\alpha}&{\color{blue}-(\alpha+\beta)}&\star&{\color{blue}\beta}\\
      {\color{blue}-(\alpha+\beta)}&{\color{magenta}-(\alpha+2\beta)}&{\color{blue}-\beta}&\star
  \end{pmatrix}\]
 
  \caption{Root system of $\mathrm{GSp}_4$ with choice of positive roots, 
  the long roots is the root system of $\mathrm{GL}_2\times_{\det}\mathrm{GL}_2$. Shaded is the positive Weyl Chamber of $ \mathrm{GSp}_4$. Below are root spaces} 
  \end{figure}
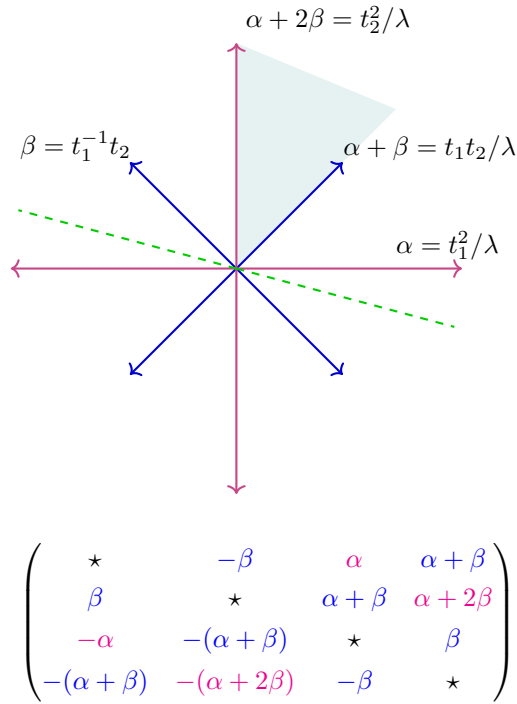

  Similarly, we get the same root system by looking at the action of $\mathrm{Lie}(T)(F)$ on $\mathfrak{g}(F)$, where $\alpha (t_1,t_2,\lambda-t_1,\lambda-t_2) = 2t_1-\lambda$, and 
  $\beta(t_1,t_2,\lambda-t_1,\lambda-t_2)= t_2-t_2$.

 We say that an element  $\gamma \in G(F)$ (resp. $X\in \mathfrak{g}(F)$) is \emph{equivalued} if 
 its images under all roots of $G(\bar{F})$ (resp. of $\mathfrak{g}(\bar{F})$) corresponding 
 to the maximal torus centralizing $\gamma$ (resp. $X$) have equal valuations.

More concretely, $X\in \mathfrak{G}(F)$ is equivalued if there is $r\in \mathbb{Q}$ such that for all $\alpha\in \Phi(G_{\bar{F}}, T_{\bar{F}})$ (where $T = G_X$), we have 
\[\mathrm{val}_{\bar{F}}(\mathrm{d}\alpha(X)) = r.\]
This $r$ is the depth of $X$.

\begin{rem} This informal remark aims to explain the relevance of studying equivalued elements.

Rather paradoxically, comparing orbital integrals for equivalued elements 
could be seen as either the easiest or hardest case, depending on one's optics. 

Indeed, all \emph{interesting} properties of equivalued elements have the same depth and therefore 
all the information needed is present in one single reductive quotient of some compact open subgroup. This 
makes the strategy to compute such volumes explicitly straightforward.
 
However, this means that equivalued elements will have minimal cancelation, and so 
we expect symplectic orbital integrals to contain a lot more information than integrals 
over the much smaller subgroup. So this case is where we expect orbital integrals to look 
the most different. 
\end{rem}

\begin{ex}
Let $X_1 = \varpi^n\pmat{0&\varepsilon\\ 1&0}\in \mathfrak{sl}_2(F)$ where $\varepsilon\in \mathcal{O}_F$ is not a square. Let $\alpha$ be the unique nontrivial root. We have $\mathrm{d}\alpha\pmat{x&0\\0&-x} = 2x$ so 
\[\mathrm{val}_{\bar{F}}(\mathrm{d}\alpha X_1) = \mathrm{val}_{\bar{F}}(2\sqrt{\varepsilon}\varpi^n) = n.\]
So $X_1$ has depth $n$.
\end{ex}
\begin{ex}
Let $X_2 = \varpi^n\pmat{0&\varepsilon\varpi\\ 1&0}\in \mathfrak{sl}_2(F)$ where $\varepsilon\in \mathcal{O}_F$ is not a square.  This time we have
\[\mathrm{val}_{\bar{F}}(\mathrm{d}\alpha X_2) = \mathrm{val}_{\bar{F}}(2\sqrt{\varepsilon}\varpi^{n+1/2}) = n+1/2.\]
The element $X_2$ has depth $n+1/2$.
\end{ex}
 
\subsection{Description of the Moy--Prasad filtration} \label{sec:roots}

We now assume $X = (X,X_2)\in \mathfrak{g}(\mathcal{O}_F)$ is elliptic,
 i.e. its centralizer is compact modulo center. 
Equivalently, $F[X] =F[X_1]\oplus F[X_2]$ where $F[X_1]/F,F[X_2]/F$ 
are quadratic field extensions.\\

Recall that  $\mathcal{B}(G,F)$  denotes the building of $G$ over $F$ and let $G_X$ the stabilizer of $X$ in $G$.

The Moy--Prasad filtration associates subgroups (resp. subalgebras) $G(F)_{v,r}$ (resp. $\mathfrak{g}(F)_{v,r}$) for each building vertex $v$ and $r\in \mathbb{R}_{\geq 0}$. 
Under this filtration, the parahoric subgroup of $G(F)$ stabilizing a vertex $v$ is $G(F)_{v,0}$ and $G(F)_{v,r_1}\subset G(F)_{v,r_2}$ whenever $r_1\geq r_2$.

The depth of $\gamma\in G(F)$ (resp. $X\in \mathfrak{g}(F)$) relative to a vertex $v$, denoted by $d(v,X)$ is the unique value of $r$ so that $\gamma\in G(F)_{v,r}\backslash G(F)_{v,s}$ (resp. $\gamma\in \mathfrak{g}(F)_{v,r}\backslash \mathfrak{g}(F)_{v,s}$) for all $s>r$.

The \emph{depth} of $\gamma$ (resp $X$) is the supremum of its relative depths.

\bigskip
\begin{theorem}[{\cite[Theorem 2.3.1]{kmvertex}}]
Assume $X\in \mathfrak{g}(F)$ is equidistributed with depth $d$, and the eigenvalues of $X$ belong to a tamely ramified extension of $F$. The following are equivalent:
\begin{enumerate}[(i)]
\item $v\in \mathcal{B}(G_X,F)$,
\item $X\in \mathfrak{g}(F)_{v,d}$.
\end{enumerate}
Equivalently, $\displaystyle\mathcal{B}(G_X, F) = \{v\in \mathcal{B}(G,F)\ :\ d(v,X)\leq d \}.$
\end{theorem}

In our case, the residue characteristic is odd hence the eigenvalues of $X$ generate a tamely ramified extension, and $X$ is elliptic hence the apartment corresponding to its
 centralizer is not contained in the building over $F$. Therefore, there is a unique Galois-fixed
point (not necessarily a vertex) from that apartment. This point is exactly $\mathcal{B}(G_X,F)$.

Let $x$ be that vertex. Up to conjugation, 
$x$ may only be a vertex or midpoint of an edge. We can see this by considering the building of $\mathrm{GL}_2\times\mathrm{GL}_2(F)$ inside the building of $G$, 
and consider the unique vertex fixed by $X$.  \\

Letting $r$ be the depth of $X$, the Theorem above tells us that $x$ is the unique point of $\mathcal{B}(G,F)$ so that $X\in G(F)_{x,r}$.

Up to conjugation, any elliptic element corresponds to one of  4 types of points in $\mathcal{B}(G,F)$, 
and only 3 of equivalued elements. Below is a table listing all possibilities, with information that will be 
relevant to our computations.

The following pages illustrate the Moy--Prasad filtration for each point of interest. They use the notation $\mathcal{O}$ for $\mathcal{O}_F$ and $\mathfrak{m} = \varpi\mathcal{O}_F$. For aesthetic purposes we decided to write the matrices in these pages in the basis $\{e_1,e_2,f_2,f_1\}$. 
We use a block matrices with integers to describe the set of matrices in $\mathfrak{g}(F)$ where each entry has valuation at least the corresponding integer. For example,
\[

\end{center}

\restoregeometry
\subsection{Cartan decompositions and reductions of Springer fibers}

Let $X\in \mathfrak{H}(F)$
is elliptic equivalued  with depth $d$, so that $\mathcal{B}(G_X,F) = \{x_i\}$, the point of type $i$ as shown in the table above.

A direct computation shows the following useful fact.
\begin{lemma}
Define  $\varpi^{k,\ell}$ as the diagonal matrix $\mathrm{diag}(\varpi^{a}, \varpi^{b}, \varpi^{c-a}, \varpi^{c-b})\in G(F)$ where $k = b-a$ and $\ell = 2a-c$. This correspondence is well-defined modulo center. The element $\varpi^{\ell,k}$ acts on $\mathfrak{g}(F)$ by $Y\mapsto \varpi^{\ell,k}Y(\varpi^{\ell,k})^{-1}$ increasing valuations of each entry as follows: 
\[\begin{pmatrix}+0&+a-b &+2a-c&+a+b-c\\ +b-a & +0&+a+b-c& +2b-c\\ +c-2a&+c-a-b&+0&+b-a\\c-a-b&+c-2b &+a-b&+0\end{pmatrix}=\begin{pmatrix}+0& -k&+\ell&+\ell+k\\ +k & +0&+\ell+k& +\ell+2k\\ -\ell&-\ell-k&+0&+k\\- \ell-k&-\ell-2k &-k&+0\end{pmatrix}.\]
\end{lemma}

Let $Z\cong F\rtimes \mathbb{Z}/2\mathbb{Z}$ be the center of $G(F)$ Pick a Haar measure giving $Z\backslash G_\gamma(F)$ measure $1$, then we want to compute 
\begin{equation}\mathrm{Orb}_{\mathfrak{g}(F)}(X,\mathds{1}_{\mathfrak{g}(\mathcal{O}_F)})=\mathrm{Orb}_{\mathfrak{g}(F)}(X,\mathds{1}_{\mathfrak{g}(F)_{x_1,0}})= \int_{Z\backslash G(F)} \mathds{1}_{\mathfrak{g}(F)_{x_1,0}}(g^{-1}Xg)\ \mathrm{d}g= \mathrm{vol}(X_\gamma),
\end{equation}
where $X_\gamma = \left\lbrace g\in Z\backslash G(F)/G(F)_{x}\ :\ g^{-1}Xg\in G(F)_{x_1,0} \right\rbrace$ is known as an \emph{affine Springer fiber}.

 We have a Cartan decomposition  
 \[Z\backslash G(F) = Z\backslash\bigsqcup_{k,\ell} G(F)_{x_i,0}\varpi^{k,\ell}G(F)_{x_1,0},\]  for the corresponding  $k,\ell$ (see the table).

Our goal is to decompose the Springer fiber 
 \[X_\gamma = \sqcup_{k,\ell}X_\gamma\cap (G(F)_{x_i,0}\varpi^{k,\ell}G(F)_{x_1,0}),\ \mathrm{vol}(X_\gamma) = \sum_{k,\ell}\mathrm{vol}(X_\gamma\cap (G(F)_{x_i,0}\varpi^{k,\ell}G(F)_{x_1,0})).\]
 The sum above is finite. Indeed our strategy will be to isolate the finitely-many double cosets intersecting with the Springer fiber nontrivially, and then compute their volume as follows.

 Let $\lambda = (k,\ell)$ and let $P_\lambda =\varpi^{\lambda}G(F)_{x_1,0}\varpi^{-\lambda}\cap G(F)_{x_i,0} $. We have 
 \begin{align*}
 X_\gamma\cap (G(F)_{x_i,0}\varpi^{\lambda}G(F)_{x_1,0}) &= \left\lbrace g_1\varpi{\lambda}g_2\in G(F)_{x_i,0}\varpi^\lambda G(F)_{x_1,0}\ :\ \mathrm{Ad}(g_1\varpi^{\lambda}g_2)X\in \mathfrak{g}(F)_{x_1,0}\right\rbrace\\
 &=  \left\lbrace g_1\varpi{\lambda}g_2\in G(F)_{x_i,0}\varpi^\lambda G(F)_{x_1,0}\ :\ \varpi^{-\lambda}\mathrm{Ad}(g_1)X\varpi^{\lambda}\in \mathfrak{g}(F)_{x_1,0}\right\rbrace\\
 &= \left\lbrace g\in G(F)_{x_i,0}/P_\lambda:\varpi^{-\lambda}\mathrm{Ad}(g_1)X\varpi^{\lambda}\in \mathfrak{g}(F)_{x_1,0}\right\rbrace P_\lambda\end{align*}

Let $Y\mapsto \bar{Y}$ be the surjection of $g(F)_{x_i,0}$ onto its
reductive quotient $\mathfrak{g}(F)_{x_i,0:0+}$.  Define 
\[ S_\lambda(X) = \left\lbrace g\in Z\backslash  G(F)_{x_i,0}: \varpi^{-\lambda}\mathrm{Ad}(g)(X)\varpi^\lambda\in \mathfrak{g}(F)_{x_1,0} \right\rbrace,\]
 \[\bar{S}_{\lambda}(X)=\left\lbrace \bar{g}\in  \bar{Z}\backslash G(F)_{x_i,0:0+}: \overline{\varpi^{-\lambda}\mathrm{Ad}(g)X \varpi^\lambda}\in\mathfrak{g}(F)_{x_1:0:0+}\right\rbrace.\]

Consider the following commutative diagram 
\[\begin{tikzcd}
S_\lambda(X)\ar[r]\ar[d,"\pi"]&S_\lambda(X)/P_\lambda\ar[d]\\
\bar{S}_\lambda(X)\ar[r]&\bar{S}_\lambda(X)/\bar{P}_\lambda\\
\end{tikzcd}\]

The group $\bar{P}_\lambda$ is a parabolic subgroup of $G(F)_{x_i,0:0+}$ and therefore we can write $\bar{S}_\lambda(X)/\bar{P}_\lambda$ as a flag variety over a finite field.  We will conclude with the volume of each $S_\lambda$ as follows

\begin{theorem}[{\cite{gkmpure}}] \label{th:gkm} Recall that $X$ is equivalued of depth $d$ and $\mathcal{B}(G_X,F)= \{x_i\}$. The reduction map $\pi:S_\lambda(X)\rightarrow \bar{S}_\lambda(X)$ is smooth, 
with affine fibers of dimension $d_{\lambda}$ described below, hence  
\[\mathrm{vol}(S_\lambda(X)) = q^{d_{\lambda}}|\bar{S}_\lambda(X)| = q^{d_{\lambda}}|\bar{S}_\lambda(X)/\bar{P}_\lambda|\ |\bar{P}_\lambda|.\]
The dimension  $d_{\lambda}$ is equal to the number of affine roots $\tilde{\alpha}$ so 
that $0\leq \tilde{\alpha}( x_i)< d$  and $\tilde{\alpha}(\varpi^{-\lambda}x_1)<0$.
\end{theorem}

\begin{cor}
Using the same notations as Theorem \ref{th:gkm}, we have 
\[\mathrm{Orb}_{\mathfrak{g}(F)}(X,\mathds{1}_{\mathfrak{g}(\mathcal{O}_F)}) = \sum_{\lambda = (k,\ell)} q^{d_{\lambda}}|\bar{S}_\lambda(X)/\bar{P}_\lambda|\ |\bar{P}_\lambda|.\]
\end{cor}

The computation of $d_\lambda$ boils down to counting the number of root hyperplanes separating $x_1$ and $\varpi^\lambda x_i$, not containing $x_1$, and at distance at most $d$ from the latter. 
Letting $\lambda = (k,\ell)$, and identifying the fundamental apartment with $\mathbb{Z}^2$ as in Figure \ref{fig:aptgl22}, with $x_1$ corresponding to the origin, we get that $\varpi^{k,\ell}(x,y) = (x+\ell,y+\ell+2k)$.

Before we tackle every case separately, let us observe that
\begin{itemize}
\item if $x_i$ is of type $1$, the jumps in the Moy--Prasad filtration only happen at integers so $X$ must have integral depth $d = n$ for some $n\in \mathbb{Z}_{\geq 0}$, 
\item if $x_i$ is of type $2$, then $X$ cannot be equivalued, 
\item if $x_i$ is of type $3$, then the Moy--Prasad quotients at half integers do not have (the groups are unipotent), 
\end{itemize} 
$X$ needs to have integral depth, since jumps in the Moy--Prasad filtration only happen at integer.

\subsection{Vertex of type 1}

Let $X\in \mathfrak{g}(F)_{x_1,n}=\varpi^n \mathfrak{g}(\mathcal{O})$ be an elliptic regular semisimple element. We decompose 
\[G(F) = \prod_{k,\ell}G(F)_{x_1,0}\varpi^{k,\ell}G(F)_{x_1,0},\]
with $k,\ell \geq 0$.\\

\begin{lemma}\label{lem:t1}
    Let $g\in G(F)_{x_1,0}$ and $Y = \mathrm{Ad}(g)X\in \mathfrak{g}(F)_{x_1,n}$. If $\mathrm{Ad}(\varpi^{\ell,k})Y\in \mathfrak{g}(F)_{x_1,0}$ for some $k,\ell\in\mathbb{Z}_{\geq0}$ then we must have $0\leq k,\ell\leq n$ and one of the following must be true
    \begin{enumerate}[(i)]    
        \item  $2k+\ell\leq n$,
        \item $k+\ell\leq n$, $2k+\ell>n$,
        \item $k+\ell>n$.
    \end{enumerate}
\end{lemma}

\begin{figure}\centering 
\begin{tikzpicture}
    \draw[thick, fill=cyan!30] (5,0) -- (0,5) -- (5,5) -- (5,0);
    \draw[thick, fill=blue!30] (2.5,0) -- (0,5) -- (5,0) -- (2.5,0);
    \path[fill=violet!30] (0,0) -- (0,5) -- (2.5,0) -- (0,0);
    \node[above left] at (0,6) {$\ell$};
    \node[below right] at (6,0) {$k$};
    \node[below] at (5,0) {$n$};
    \node[left] at (0,5) {$n$};
    \node[below] at (2.5,0) {$n/2$};
    \node at (1,1) {$(i)$};
    \node at (2.3,2) {$(ii)$};
    \node at (3.7,3.5) {$(iii)$};
    \draw[ultra thick, ->] (0,0)--(6,0);
    \draw[ultra thick, ->] (0,0)--(0,6);
\end{tikzpicture}
\begin{tikzpicture}
    \draw[rounded corners, fill=pink!30] (5,1) -- (5,5) -- (1,5) -- (1,1) -- (5,1);
    \draw[rounded corners, thick, fill=teal!30] (0,0) -- (1,0) -- (1,1) -- (0,1) -- (0,0);
    \draw[rounded corners, fill=green!30] (1,1) -- (5,1) -- (5,0) -- (1,0) -- (1,1);
    \draw[rounded corners, fill=orange!30] (1,1) -- (1,5) -- (0,5) -- (0,1) -- (1,1);
    \node[above left] at (0,6) {$\ell$};
    \node[below right] at (6,0) {$k$};
    \node[below] at (5,0) {$n$};
    \node[left] at (0,5) {$n$};
    \node[below] at (2.5,0) {$n/2$};
    \node at (.5,.5) {$G$};
    \node at (3,3) {Borel};
    \node at (3,.5) {Klingen};
    \node at (.5,3) {Siegel};
    \draw[ultra thick, ->] (0,0)--(6,0);
    \draw[ultra thick, ->] (0,0)--(0,6);
\end{tikzpicture}

\caption{Case of vertex of type 1 with $\bar{P}_\lambda$ on the right}
\label{fig:t1}
\end{figure}

\begin{proof} Assume we have
    \[\mathrm{Ad}(\varpi^{\ell,k})Y\in \begin{bmatrix}n& n-k&n+\ell&n+\ell+k\\ n+k & n&n+\ell+k& n+\ell+2k\\n -\ell&n-k-\ell&n&n+k\\n-k-\ell&n-2k-\ell &n-k&n\end{bmatrix}\cap \mathfrak{g}(\mathcal{O}_F). \]
 
    \begin{itemize}
        \item \textbf{Case 1}. $n-k<0$.  Since $\ell>0$ we get that $n-k-\ell<0$ and $n-k-2\ell<0$. 
 This means that $Y_{1,2}, Y_{3,2}, Y_{4,1}, Y_{4,2}$ all must have valuation strictly larger than $n$, and will vanish in $\mathfrak{g}(F)_{x_1,n:n+} = \mathfrak{gsp}_4(\mathbb{F}_q)$. In other words, we get that 
        \[\bar{Y} = \begin{pmatrix}
            \star &0&\star&\star\\
            \star &\star&\star&\star\\
            \star &0&\star&\star\\
            0 &0&0&\star\\
        \end{pmatrix}\in \mathfrak{gsp}_4(\mathbb{F}_q),\]
        which is impossible since $Y$ must be elliptic. This case does therefore not happen. 
        \item \textbf{Case 2.} $n-\ell<0$ then again since $k\geq 0$ we get that $n-k-\ell<0$ and $n-2k-\ell <0$. Following  the same reasonning as before, we can write
        \[\bar{Y} = \begin{pmatrix}
            \star &\star&\star&\star\\
            \star &\star&\star&\star\\
            0 &0&\star&\star\\
            0 &0&\star&\star\\
        \end{pmatrix} = \left(\begin{array}{c|c}
            X_1 & \star \\
             \hline
             0& X_2
        \end{array}\right)\in \mathfrak{gsp}_4(\mathbb{F}_q).\]
        Since $\bar{Y}$ is regular semisimple, it is therefore contained in a Borel and centralized by the corresponding Levi, which is a split torus, hence cannot 
        be elliptic.
       \item \textbf{Case 3.} $n-k\geq 0$, $n-\ell\geq 0$ and $n-k-\ell <0$. We proceed in the same way, we must have   $n-\ell-2k<0$, hence
        \[\bar{Y} = \begin{pmatrix}
            \star &\star&\star&\star\\
            \star &\star&\star&\star\\
            \star &0&\star&\star\\
            0 &0&\star&\star\\
        \end{pmatrix} \in \mathfrak{gsp}_4(\mathbb{F}_q).\]
        This is case $(iii)$ in the statement of the Lemma.
        \item \textbf{Case 4}.  $n-k\geq 0$, $n-\ell\geq 0$, $n-k-\ell \geq 0$, and $n-2k-\ell<0$.
        \[\bar{Y} = \begin{pmatrix}
            \star &\star&\star&\star\\
            \star &\star&\star&\star\\
            \star &\star&\star&\star\\
            \star &0&\star&\star\\
        \end{pmatrix} \in \mathfrak{gsp}_4(\mathbb{F}_q).\]
        This corresponds to case $(ii)$.
        \item \textbf{Case 5}.  $n-k\geq 0$, $n-\ell\geq 0$, and $n-2k-\ell\geq 0$.
        \[\bar{Y} = \begin{pmatrix}
            \star &\star&\star&\star\\
            \star &\star&\star&\star\\
            \star &\star&\star&\star\\
            \star &\star&\star&\star\\
        \end{pmatrix} \in \mathfrak{gsp}_4(\mathbb{F}_q)\] so there are no conditions. This is case $(i)$ of the Lemma.
    \end{itemize}
\end{proof}

We list the conditions corresponding to each case below.
\begin{cor} Let $\lambda = (k,\ell)$ and $\bar{g}\in G(F)_{x_1, 0:0+}$. Define $\bar{Y} = \mathrm{Ad}(\bar{g})(\bar{X})$. Then 
\begin{itemize}
\item In case $(i)$, there are no condition on $\bar{g}$ so that $\bar{g}\in \bar{S}_\lambda(X)$.
\item In case $(ii)$, we have $\bar{g}\in \bar{S}_\lambda(X)$ if and only if $\bar{Y}\bar{e}_2\subset \overline{e_2^{\perp}}$.
\item In case $(iii)$, we have  $\bar{g}\in \bar{S}_\lambda(X)$ if and only if $\bar{Y}\bar{e}_2\subset \mathrm{Span}\{\bar{e_1},\bar{e}_2\}$, and $\bar{Y}\mathrm{Span}\{\bar{e_1},\bar{e}_2\}\subset\mathrm{Span}\{\bar{e}_1,\bar{e}_2,\bar{f}_1\} = \overline{e_2^\perp}$.
\end{itemize}
\end{cor}

\begin{lemma}
        Let $\lambda = (k,\ell)$ as described in Lemma \ref{lem:t1}. We have the following:
        \begin{itemize}
        \item If $\ell = k=0$ then $\bar{P}_\lambda = G(F)_{x_1,0:0+}$ and the quotient $G(F)_{x_1,0:0+}/\bar{P}_\lambda\cong \{\mathrm{pt}\}$.
        \item If $\ell=0$ and  $k>0$ then $\bar{P}_\lambda$ is a Klingen parabolic the quotient $G(F)_{x_1,0:0+}/\bar{P}_\lambda\cong \mathbb{P}^3(\mathbb{F}_q)$ which has $(q^4-1)/(q-1) = (q+1)(q^2+1)$ points.
        \item If $k=0$ and  $\ell>0$ then $\bar{P}_\lambda$ is a Siegel parabolic and the quotient $G(F)_{x_1,0:0+}/\bar{P}_\lambda$ is isomorphic to the space of Lagrangians in $\mathbb{F}_q^4$ which has $(q+1)(q^2+1)$ points. 
        \item If $\ell,k>0$ and then $\bar{P}_\lambda$ is a Borel subgroup and the quotient $G(F)_{x_1,0:0+}/\bar{P}_\lambda$ is isomorphic to the space of full isotropic flags in $\mathbb{F}_q^4$ which has $(q+1)^2(q^2+1)$ points.

        \end{itemize}
as in Figure \ref{fig:t1}.
\end{lemma}
\begin{proof}  
Let $\lambda = (k,\ell)$. We use the notation $\mathcal{O}= \mathcal{O}_F$ and $\mathfrak{m} = \varpi\mathcal{O}_F$. Look at the reduction map 
\begin{align*}
P_\lambda &= \varpi^{\lambda}G(F)_{x_1,0}\varpi^{-\lambda}\cap G(F)_{x_1,0}\rightarrow \bar{P}_\lambda \subset G(F)_{x_1,0:0+} \cong \mathrm{GSp}_4(\mathbb{F}_q)\\
\end{align*}
\[\pmat{ \mathcal{O}&\mathfrak{m}^{k}& \mathfrak{m}^{-\ell}&\mathfrak{m}^{-\ell-k}\\
 \mathfrak{m}^{-k}&\mathcal{O}&\mathfrak{m}^{-\ell-k}&\mathfrak{m}^{-\ell-2k}\\
 \mathfrak{m}^{\ell}&\mathfrak{m}^{k+\ell}&\mathcal{O}&\mathfrak{m}^{-k}\\
 \mathfrak{m}^{k+\ell}&\mathfrak{m}^{\ell+2k}&\mathfrak{m}^{k}&\mathcal{O}}
\cap \pmat{ \mathcal{O}&\mathcal{O}&\mathcal{O}&\mathcal{O}\\
 \mathcal{O}&\mathcal{O}&\mathcal{O}&\mathcal{O}\\
 \mathcal{O}&\mathcal{O}&\mathcal{O}&\mathcal{O}\\
 \mathcal{O}&\mathcal{O}&\mathcal{O}&\mathcal{O}}
\rightarrow
\pmat{\mathcal{O}/\mathfrak{m}&\mathcal{O}/\mathfrak{m}&\mathcal{O}/\mathfrak{m}&\mathcal{O}/\mathfrak{m}\\ \mathcal{O}/\mathfrak{m}&\mathcal{O}/\mathfrak{m}&\mathcal{O}/\mathfrak{m}&\mathcal{O}/\mathfrak{m}\\ \mathcal{O}/\mathfrak{m}&\mathcal{O}/\mathfrak{m}&\mathcal{O}/\mathfrak{m}&\mathcal{O}/\mathfrak{m}\\
\mathcal{O}/\mathfrak{m}&\mathcal{O}/\mathfrak{m}&\mathcal{O}/\mathfrak{m}&\mathcal{O}/\mathfrak{m}}.\]
\begin{itemize}
\item When $\ell=k=0$ then $\bar{P}_\lambda = G(F)_{x_1,0:0+}$.
\item When $\ell=0$ and $0< k\leq n$, we get  
$\bar{P}_\lambda = \left\lbrace
\pmat{\mathcal{O}/\mathfrak{m}&0&\mathcal{O}/\mathfrak{m}&\mathcal{O}/\mathfrak{m}\\ 
\mathcal{O}/\mathfrak{m}&\mathcal{O}/\mathfrak{m}&\mathcal{O}/\mathfrak{m}&\mathcal{O}/\mathfrak{m}\\ 
\mathcal{O}/\mathfrak{m}&0&\mathcal{O}/\mathfrak{m}&\mathcal{O}/\mathfrak{m}\\
0&0&0&\mathcal{O}/\mathfrak{m}}\right\rbrace = \mathrm{Stab}(\mathbb{F}_q e_2)$ is a Klingen parabolic subgroup. Therefore, $G(F)_{x_1,0:0+}/\bar{P}_\lambda$ is 
isomorphic to $\mathbb{P}^3(\mathbb{F}_q)$.
\item When $k=0$  and $\ell >0$ we have
 $\bar{P}_\lambda = \left\lbrace
    \pmat{\mathcal{O}/\mathfrak{m}&\mathcal{O}/\mathfrak{m}&\mathcal{O}/\mathfrak{m}&\mathcal{O}/\mathfrak{m}\\ 
    \mathcal{O}/\mathfrak{m}&\mathcal{O}/\mathfrak{m}&\mathcal{O}/\mathfrak{m}&\mathcal{O}/\mathfrak{m}\\
     0&0&\mathcal{O}/\mathfrak{m}&\mathcal{O}/\mathfrak{m}\\
0&0&\mathcal{O}/\mathfrak{m}&\mathcal{O}/\mathfrak{m}}\right\rbrace$ is the Siegel parabolic subgroup which is the stabilizer 
of the Lagrangian spanned by $\{e_1,e_2\}$. Therefore, $G(F)_{x_1,0:0+}/\bar{P}_\lambda$ is the partial flag variety corresponding to Lagrangians. 
\end{itemize}
\item When $\ell,k>0$ then $\bar{P}_\lambda = \left\lbrace
    \pmat{\mathcal{O}/\mathfrak{m}&0&\mathcal{O}/\mathfrak{m}&\mathcal{O}/\mathfrak{m}\\ 
    \mathcal{O}/\mathfrak{m}&\mathcal{O}/\mathfrak{m}&\mathcal{O}/\mathfrak{m}&\mathcal{O}/\mathfrak{m}\\
     0&0&\mathcal{O}/\mathfrak{m}&\mathcal{O}/\mathfrak{m}\\
0&0&0&\mathcal{O}/\mathfrak{m}}\right\rbrace
$ is a Borel subgroup and the stabilizer of the full flag 
\[0\subset\mathrm{Span}_{\mathbb{F}_q}\{e_2\}\subset \mathrm{Span}_{\mathbb{F}_q}\{e_1,e_2\}\subset \{e_2\}^{\perp} = \mathrm{Span}_{\mathbb{F}_q}\{e_1,e_2,f_1\}\subset \mathbb{F}_q^4.\] We get that 
$G(F)_{x_1,0:0+}/\bar{P}_\lambda$ corresponds to the flag variety of full isotropic flags, $0\subset W\subset L$, where $L$ is a Lagrangian. This in turn induces the full flag 
\[0\subset W\subset L\subset W^{\perp}\subset \mathbb{F}_q^4.\] 
      This concludes the proof.  \end{proof}




 For each $\lambda = (k,\ell)$ let $\mathcal{P}_{\lambda} = G(F)_{x_1,0:0+}/\bar{P}_\lambda$ we have

\begin{itemize}
\item Case $(i)$. There is no condition so $\bar{S}_\lambda(X) = G(F)_{x_1,0:0+}$.
\item Case ($(ii)$-Klingen). This is when $n/2< k\leq n$. 
We have $\mathcal{P}_\lambda\cong \mathbb{P}^3(\mathbb{F}_q)$ and an element is determined by the 
image of $e_2$. Let $g\in G(F)_{x_1,0:0+}$. The condition $g^{-1}\bar{X}g \in \pmat{\star&\star&\star&\star\\\star&\star&\star&\star\\\star&\star&\star&\star\\\star&0&\star&\star} = \mathrm{Stab}(e_2^{\perp})$ 
if and only if $\langle \bar{X}ge_2,ge_2\rangle = 0$. 
So 
\[\bar{S}_\lambda(X)/\bar{P}_\lambda \cong \{\text{line }\mathcal{L}\in\mathbb{P}^3(\mathbb{F}_q)\ :\ \bar{X}\mathcal{L}\subset \mathcal{L}^\perp\}.\]
\item Case ($(ii)$-Borel). As in the previous case, the condition on $g^{-1}\bar{X}g$ is that $\langle \bar{X}ge_2,g e_2\rangle$. We get
\[\bar{S}_\lambda(X)/\bar{P}_\lambda \cong \{(\mathcal{L},L)\in \mathcal{P}_\lambda\ :\ \bar{X}\mathcal{L}\subset \mathcal{L}^\perp\}.\]
\item Case ($(iii)$-Borel). Let $Y = g^{-1}\bar{X}g$. The condition for $(iii)$  is that 
\[Ye_2\subset \mathrm{Span}\{e_1,e_2\},\ Y\mathrm{Span}\{e_1,e_2\}\subset e_2^{\perp}.\]
We get 
\[\bar{S}_\lambda(X)/\bar{P}_\lambda \cong \{(\mathcal{L},L)\in \mathcal{P}_\lambda\ :\ \bar{X}\mathcal{L}\subset L,\ \bar{X}L\subset \mathcal{L}^\perp\}.\]
\end{itemize}

\begin{lemma} The number of lines $\mathcal{L}\in \mathbb{P}^3(\mathbb{F}_q)$ so that $\bar{X}\mathcal{L} \subset \mathcal{L}^\perp$ is equal to $(q+1)^2$. 
\end{lemma}
\begin{proof}
Write $X = X_1\oplus X_2\in \mathfrak{h}(F)$. Write $\mathcal{L} = \mathrm{Span}(u\oplus v)$. We cannot have $\langle X_1u, u\rangle=0$ because $u$ spans a Lagrangian of $\mathrm{Span}\{e_1,f_1\}$
and so  $\bar{X}_1u\in u^\perp = \mathrm{Span}\{u\}$ only if $u$ is an eigenvector of $\bar{X}_1$, which isn't possible since $\bar{X}$ is elliptic. Similarly, we cannot have $\langle\bar{X}_2v,v\rangle =0$.

Both $\bar{X}_1,\bar{X}_2$ split over the unique quadratic extension of $\mathbb{F}_q$ generated by a nonsquare residue, so the number of desired lines can be obtained by computing 
\[\frac{|\{(a,b,c,d)\in \mathbb{F}_q^4\backslash \{0\}\ :\ a^2+b^2+c^2\varepsilon_1+d^2\varepsilon_2=0\}|}{q-1},\]
where $\varepsilon_1,\varepsilon_2$ are distinct nonsquare residues of $\mathbb{F}_q$.

This is equivalent to counting the number of solutions of $(x,y)\in\mathbb{F}_{q^2}^\times$ so that $N_{\mathbb{F}_{q^2}/\mathbb{F}_{q}}(x/y) = -1$.
 Consider the map $R_{\mathbb{F}_{q^2}/\mathbb{F}_{q}}\mathbb{G}_m\times R_{\mathbb{F}_{q^2}/\mathbb{F}_{q}}\mathbb{G}_m\rightarrow \mathbb{G}_m$. We look for the size of the fiber over $-1$ which is equal to the 
 number of solutions to $N_{\mathbb{F}_{q^2}/\mathbb{F}_{q}}(x) =N_{\mathbb{F}_{q^2}/\mathbb{F}_{q}}(y)$. Fixing $x$, any two solutions for $y$ differ by a norm $1$ element, and using the isogeny of $\mathbb{G}_m\times R^{(1)}_{\mathbb{F}_{q^2}/\mathbb{F}_{q}}\mathbb{G}_m$
 with $R_{{\mathbb{F}_{q^2}/\mathbb{F}_{q}}}\mathbb{G}_m$, we get that there are $(q^2-1)/(q-1)=q+1$ possible $y$'s. And there are $q^2-1$ possible $x$'s, therefore we end up with the count 
 \[\frac{(q^2-1)(q+1)}{(q-1)} = (q+1)^2.\]
\end{proof}
\begin{rem} Alternatively, we may note that 
 $\overline{S}_{\lambda}(X)/\bar{P}_\lambda$ is a toric variety over the root datum $T = Z_{\mathrm{Sp}_4}(\bar{X})$ (the centralizer in $\mathrm{Sp}_4$ of $X$). 
        The action of $T$ on $\overline{S}_{\lambda}(X)/\mathbb{F}_q$ has an open orbit and all $\mathbb{F}_q$-points of $\overline{S}_{\lambda}(X)$ are contained in this open orbit, 
        therefore $|\overline{S}_{\lambda}(X)(\mathbb{F}_q)| = |T(\mathbb{F}_q)| = (q+1)^2$.
\end{rem}
\begin{cor}
The number of images of full isotropic flags $0\subset \mathcal{L}\subset L$ in $\mathbb{F}_q^4$ so that $\bar{X}\bar{\mathcal{L}}\subset \bar{\mathcal{L}}^\perp$ is equal to $(q+1)^3$.   
\end{cor}
\begin{proof}
There are $(q+1)^2$ choices for $\bar{\mathcal{L}}$ as in the previous Lemma. For each line $\mathcal{L}$, the space $\mathcal{L}^\perp$ has dimension $3$,
hence the  Lagrangians containing $\mathcal{L}$ are in correspondence with $\mathbb{P}(\mathcal{L}^\perp/L)\cong \mathbb{P}^1$.
Therefore, the total count is $(q+1)^2|\mathbb{P}^1(\mathbb{F}_q)| = (q+1)^3$. 
\end{proof}

\begin{lemma}The number of images of full isotropic flags $0\subset \mathcal{L}\subset L\subset \mathcal{L}^\perp\subsetneq \mathcal{O}_F^4$ in $\mathbb{F}_q^4$ so that $\bar{X}\bar{\mathcal{L}}\subset \bar{L}$ and $X\bar{L}\subset \bar{\mathcal{L}}^\perp$ is equal to  $(q+1)^2$.
\end{lemma}
\begin{proof}
Let us show that the condition is equivalent to $X\mathcal{L}\subset \mathcal{L}^\perp$. Pick a line $\mathcal{L}$ so that $X\mathcal{L}\subset \mathcal{L}^{\perp}$ Since $X$ is elliptic, we may not have $X\mathcal{L}\subset \mathcal{L}$, therefore the space spanned by $\mathcal{L}$ and $X\mathcal{L}$
is the unique Lagrangian containing $X\mathcal{L}$. Call it $L$.

The last condition is $XL \subset \mathcal{L}^{\perp}$, but $X\mathcal{L}\subset \mathcal{L}^{\perp}$ so the only condition left is $X^2\mathcal{L}\subset \mathcal{L}^{\perp}$ which is vacuously true, since $X$ is 
symplectic, it is equivalent to $X\mathcal{L}\subset (X\mathcal{L})^{\perp}$, but any line is isotropic.
\end{proof}

First count all $k,\ell$ of type $(i)$, which corresponds to no condition hence $|\bar{S}_{\lambda}(X)| = |G(F)_{x_1,0:0+}|$. Note that when $n=0$ we get $1$ vertex, and when increasing from $n-1$
to $n$ we add integer points on the line $2x+y = n$, which is equal to $\lfloor n/2\rfloor +1$. 

Let $r = \lfloor n/2\rfloor$ and $\varepsilon = n-2r$. The number of vertices of type $(i)$ is therefore 
\[n+1+r(r+1)-(1-\varepsilon) r =(n+1-r)(r+1)=\left\lfloor\frac{(n+2)^2}{4}\right\rfloor = \left\lbrace\begin{array}{ll}(n+2)^2/4& n\text{ even}\\
(n+1)(n+3)/4& n\text{ odd}\end{array}\right. .\]

It is also straightforward to count the vertices of type $(iii)$, there are exactly $n(n+1)/2$ and by subtraction we get 
\[(n+1)^2-\left\lfloor\frac{(n+2)^2}{4}\right\rfloor - n(n+1)/2 =\left\lfloor \frac{(n+1)^2}{4}\right\rfloor= \left\lbrace\begin{array}{ll}n(n+2)/4& n\text{ even}\\
(n+1)^2/4& n\text{ odd}\end{array}\right.\] vertices of type $(ii)$, of which $\lceil n/2\rceil$
correspond to Klingen parabolic subgroups $(\ell = 0)$. 

To sum up we get 
\begin{itemize}
\item $\left\lfloor\frac{(n+2)^2}{4}\right\rfloor$ $\lambda$'s with $|\bar{S}_\lambda(X)| = |Z\backslash G(F)_{x_1,0:0+}| = |\mathrm{Sp}_4(\mathbb{F}_q)| = q^4(q^2-1)(q^4-1)$ (type $(i)$).
\item $n(n+1)/2$ $\lambda$'s where $|\bar{S}_\lambda(X)/\bar{P}_\lambda| = (q+1)^2$ and $|\bar{Z}\backslash\bar{P}_\lambda| = q^4(q-1)^2$ (type $(iii)$-Borel).
\item $\lceil n/2\rceil$ $\lambda$'s where $|\bar{S}_\lambda(X)/\bar{P}_\lambda| = (q+1)^2$ and $|\bar{Z}\backslash\bar{P}_\lambda| =\frac{|\mathrm{Sp}_4(\mathbb{F}_q)|}{|\mathbb{P}^3(\mathbb{F}_q)|}= q^4(q^2-1)(q-1)$ (type $(ii)$-Klingen).
\item $\lfloor (n+1)^2/4 \rfloor-\lceil n/2\rceil = \lfloor n^2/4\rfloor$ $\lambda$'s where $|\bar{S}_\lambda(X)/\bar{P}_\lambda| = (q+1)^3$ and $|\bar{Z}\backslash \bar{P}_\lambda| = q^4(q-1)^2$ (type $(ii)$-Borel).
\end{itemize} 
Note that in case $(ii)$, regardless of the type of parabolic, we get that 
\[|\bar{S}_\lambda(X)| = |\bar{S}_\lambda(X)/\bar{P}_\lambda|\times |\bar{Z}\backslash \bar{P}_\lambda| = (q+1)^2\times q^4(q^2-1)(q-1) = (q+1)^3\times q^4(q-1)^2.\]

Let $\lambda = (k,\ell)$. The dimension of the fiber in Theorem \ref{th:gkm} is given by  
\begin{align*}
d_{\lambda} &= \max(0,\min(n,k))+\max(0,\min(n,\ell))+\max(0,\min(n,\ell+k))+\max(0,\min(n,\ell+2k))\\
&= k+\ell + \min(n,\ell+k) + \min(n,\ell+2k).
\end{align*}
This implies 
\begin{equation}\label{eq:dimt1}
d_{\lambda} = \left\lbrace \begin{array}{ll}3\ell+4k&\text{ in case }(i)\\2\ell+2k+n&\text{ in case }(ii)\\ \ell+k+2n&\text{ in case }(iii)\end{array}\right. .
\end{equation}

We therefore have 
\begin{equation}
\mathrm{Orb}_{\mathfrak{g}(F)}(X,\mathds{1}_{\mathfrak{g}(\mathcal{O}_F)})= q^4(q^2-1)^2\left((q^2+1)\!\!\!\sum_{\underset{\ell+2k\leq n}{k,\ell\geq 0}}q^{3\ell+4k} +(q+1)\!\!\!\!\!\!\!\!\!\sum_{\underset{\ell+2k> n, \ell+k\leq n}{k,\ell\geq 0}}q^{2\ell+2k+n}+ \sum_{\underset{\ell+k> n}{0\leq k,\ell\leq n }}q^{\ell+k+2n}\right).
\end{equation}
\begin{theorem}
In the case of vertices of type  $1$, we have 
\[\mathrm{Orb}_{\mathfrak{g}(F)}(X,\mathds{1}_{\mathfrak{g}(\mathcal{O}_F)}) =\frac{q^4(q^2-1)}{q^3-1} \left(A(q) +B(q)q^{2n}+(C_1(q)+nC_2(q))q^{3n}+D(q)q^{4n}\right),\]
Where 
\begin{align*}A &= 1,\\
B &= \left\lbrace\begin{array}{ll}
q^2(2-q)(1+q+q^2)& \text{if }n\text{ is even}\\
q(1-q+2q^2-q^3)(1+q+q^2)& \text{if }n\text{ is odd}
\end{array}\right.,\\
C_1 &= q^2(q^5-4q^2-4q-3),\\
C_2 &= q(q^2-1)(q^3-1),\text{ and}\\
D &=q^2+2q^3+2q^4+q^5=q^2(1+q)(1+q+q^2).
\end{align*}

\end{theorem}
\begin{proof}
  The third sum is the easier one to compute since 
  \[\sum_{\underset{\ell+k> n}{k,\ell\geq 0}}q^{\ell+k+2n} = \sum_{i = 1}^ni q^{4n+1-i} = q^{4n+1}\sum_{i=1}^niq^{-i}  = q^{3n+1}\frac{q(q^n-1)-n(q-1)}{(q-1)^2}.\]

For the first sum, we sum over $i = \ell+2k$ and $q^{3\ell+4k} = q^{3i}q^{-2k}$. Such a line has integer points with $k$ values from $0$ to $\lfloor i/2\rfloor.$ Threfore, we get 
\[\sum_{\underset{\ell+2k\leq n}{k,\ell\geq 0}}q^{3\ell+4k} = \sum_{i=0}^nq^{3i}\sum_{k=0}^{\lfloor i/2\rfloor}q^{-2k} =\sum_{i=0}^nq^{3i} \frac{1-q^{-2\left\lfloor\frac{i}{2}\right\rfloor-2}}{1-q^{-2}} = \frac{1}{1-q^{-2}}\sum_{i=0}^n q^{3i} - \frac{q^{-2}}{1-q^{-2}}\sum_{i=0}^n q^{3i-2\lfloor i/2\rfloor}, \]
which simplifies to 
\begin{align*}&\frac{1-q^{3n+3}}{(1-q^3)(1-q^{-2})} - \frac{q^{-2}}{1-q^{-2}}\left(\sum_{\underset{i\text{ even}}{i=0}}^nq^{2i}+\sum_{\underset{i\text{ odd}}{i=0}}^nq^{2i+1}\right),\\
&=\frac{1-q^{3n+3}}{(1-q^3)(1-q^{-2})} -\frac{q^{-2}}{1-q^{-2}}\left(\frac{1-q^{4\lfloor n/2\rfloor +4}}{1-q^{4}}+q^3\frac{1-q^{4\lceil n/2\rceil}}{1-q^4}\right),\\
&= \left\lbrace\begin{array}{ll} \frac{1-q^{3n+3}}{(1-q^3)(1-q^{-2})} -\frac{q^{-2}}{1-q^{-2}}\left(\frac{1-q^{2n +4}}{1-q^{4}}+q^3\frac{1-q^{2n}}{1-q^4}\right)& n\text{ even}\\
\frac{1-q^{3n+3}}{(1-q^3)(1-q^{-2})} -\frac{q^{-2}}{1-q^{-2}}\left(\frac{1-q^{2n+2}}{1-q^{4}}+q^3\frac{1-q^{2n+2}}{1-q^4}\right)& n\text{ odd}\end{array}\right.,\\
&= \frac{1}{(q^3-1)(q^4-1)}\left\lbrace\begin{array}{ll} {(q^{n+2}-1)(q^{2n+5}+q^{2n+3}-q^{n+2}-1)}& n\text{ even}\\
 (q^{n+1}-1)(q^{n+3}-1)(q^{n+3}+q^{n+1}+1)& n\text{ odd}\end{array}\right..\\
\end{align*}
Write the middle sum in term of $n-i=\ell+k$. When $i = 0$ we get $n$ points on the line, $i=1$ we get $n-2$ points, then $n-4$ points for $i=3$ and so on, we end up with 

\begin{align*}
\sum_{\underset{\ell+2k> n, \ell+k\leq n}{k,\ell\geq 0}}q^{2\ell+2k+n}&=\sum_{i=0}^{\lceil n/2\rceil-1} (n-2i)q^{3n-2i},\\
& =\frac{q^{2n+1}}{(q^2-1)^2}\left(nq^{n+3}-(2+n)q^{n+1} + 
\left\lbrace\begin{array}{ll}
2q& n\text{ even}\\
1+q^2&n\text{ odd}
\end{array}\right.\right) .
\end{align*}
Summing everything together we get the desired result.
\end{proof}

\begin{rem}
Note that we could simplify the notation slightly, most terms have a factor $(q^2+q+1)$ which can be simplified 
with the denominator $(q-1)(q^2+q+1)$. 
\end{rem}

One can compute the Shalika germ expansion on the distinguished subgroup using Theorem \ref{th:shalik-elliptic}.



\subsection{Vertex of type 3}
Let $X\in \mathfrak{g}(F)_{x_3,n}=\varpi^n \mathfrak{g}(F)_{x_3,0}$ be an elliptic regular semisimple element. We decompose 
\[G(F) = \prod_{k,\ell}G(F)_{x_3,0}\varpi^{k,\ell}G(F)_{x_1,0},\]
with $\ell, \ell+2k \geq 0$.\\

\begin{lemma}\label{lem:t3}
     Let $g\in G(F)_{x_3,0}$ and $Y = \mathrm{Ad}(g)X\in \mathfrak{g}(F)_{x_3,n}$.
      If $\mathrm{Ad}(\varpi^{\ell,k})Y\in \mathfrak{g}(F)_{x_1,0}$ for
       some $k,\ell\in\mathbb{Z}$ so that $\ell, \ell+2k\geq 0$.  
       We must have
    \[0\leq \ell\leq n-1, \text{ and }  -\ell\leq 2k\leq n-\ell.\]
\end{lemma}
\begin{proof}  
We have 
    \[\mathrm{Ad}(\varpi^{\ell,k})Y\in 
    \begin{bmatrix}n& n+1-k&n+1+\ell&n+1+\ell+k\\ n+k & n&n+1+\ell+k& n+\ell+2k\\n-1 -\ell&n-k-\ell&n&n+k\\n-k-\ell&n-2k-\ell &n+1-k&n\end{bmatrix}\cap \mathfrak{g}(F)_{x_1,0}. \]
        and 
        \[\bar{Y} \in \mathfrak{g}(F)_{x_3,n:n+} = \mathfrak{h}(\mathbb{F}_q) = \begin{bmatrix}
            \star&0&\star&0\\
            0&\star&0&\star\\
            \star&0&\star&0\\
            0&\star&0&\star
        \end{bmatrix}.\]

      \begin{itemize}
      \item {\bf Case 1.} If $n-1-\ell <0$ then the valuation of $Y_{3,1}$ must be strictly greater than $n$, hence  
      \[\bar{Y} \in  \begin{bmatrix}
            \star&0&\star&0\\
            0&\star&0&\star\\
            0&0&\star&0\\
            0&\star&0&\star
        \end{bmatrix},\]
        which cannot happen since $\bar{Y}$ is elliptic. 
      \item {\bf Case 2.} If $n-2k-\ell<0$, the same reasonning for $Y_{4,2}$ yields
 \[\bar{Y} \in   \begin{bmatrix}
            \star&0&\star&0\\
            0&\star&0&\star\\
            \star&0&\star&0\\
            0&0&0&\star
        \end{bmatrix}, \]
        which is impossible for elliptic $\bar{Y}$.
      \end{itemize} 
      This conclude the proof.
\end{proof}

\begin{lemma}
        Let $\lambda=(k,\ell)$ as described in Lemma \ref{lem:t3} and let $B$ denote the upper-triangular Borel of $\mathrm{GL}_2(\mathbb{F}_q)$. We have the following:
        \begin{itemize}
        \item If $\ell+2k=0$ then $\bar{P}_\lambda\cong (B\times_{\det} \mathrm{GL}_2)(\mathbb{F}_q)$ and therefore  $G(F)_{x_3,0:0+}/\bar{P}_\lambda \cong \mathbb{P}^1(\mathbb{F}_q)$,
        \item Otherwise, $\bar{P}_\lambda\cong (B\times_{\det} B)(\mathbb{F}_q)$, and therefore  $G(F)_{x_3,0:0+}/\bar{P}_\lambda \cong  (\mathbb{P}^1\times \mathbb{P}^1)(\mathbb{F}_q)$,
        \end{itemize}
as in Figure \ref{fig:t3}.
\end{lemma}
\begin{proof}
Let $\lambda = (k,\ell)$. We use the notation $\mathcal{O}= \mathcal{O}_F$ and $\mathfrak{m} = \varpi\mathcal{O}_F$. Look at the reduction map 
\begin{align*}
P_\lambda &= \varpi^{\lambda}G(F)_{x_1,0}\varpi^{-\lambda}\cap G(F)_{x_3,0}\rightarrow \bar{P}_\lambda \subset G(F)_{x_3,0:0+} \cong ({\color{orange}\mathrm{GL}_2}\times_{\det}{\color{purple}\mathrm{GL}_2})(\mathbb{F}_q)\\
\end{align*}
\[\pmat{{\color{orange}\mathcal{O}}&\star &{\color{orange}\mathfrak{m}^{-\ell}}&\star \\ 
\star &{\color{purple}\mathcal{O}}&\star &{\color{purple}\mathfrak{m}^{-\ell-2k}}\\
{\color{orange}\mathfrak{m}^{\ell}}&\star& {\color{orange}\mathcal{O}}&\star\\
 \star &{\color{purple}\mathfrak{m}^{\ell+2k}}&\star &{\color{purple}\mathcal{O}}}
\cap \pmat{{\color{orange}\mathcal{O}}&\star &{\color{orange}\mathfrak{m}}&\star \\ 
\star &{\color{purple}\mathcal{O}}&\star &{\color{purple}\mathcal{O}}\\
{\color{orange}\mathfrak{m}^{-1}}&\star& {\color{orange}\mathcal{O}}&\star\\ 
\star &{\color{purple}\mathcal{O}}&\star &{\color{purple}\mathcal{O}}}
\rightarrow
\underset{\left(
\pmat{{\color{orange}\mathcal{O}/\mathfrak{m}}&{\color{orange}\mathfrak{m}/\mathfrak{m}^2}\\
 {\color{orange} \mathfrak{m}^{-1}/\mathcal{O}}&{\color{orange}\mathcal{O}/\mathfrak{m}}},
 \pmat{{\color{purple}\mathcal{O}/\mathfrak{m}}&{\color{purple}\mathcal{O}/\mathfrak{m}}\\
{\color{purple}\mathcal{O}/\mathfrak{m}}&{\color{purple}\mathcal{O}/\mathfrak{m}}}\right)} {\underbrace{\pmat{{\color{orange}\mathcal{O}/\mathfrak{m}}&0&{\color{orange}\mathfrak{m}/\mathfrak{m}^2}&0
 \\0&{\color{purple}\mathcal{O}/\mathfrak{m}}&0&{\color{purple}\mathcal{O}/\mathfrak{m}}\\
{\color{orange} \mathfrak{m}^{-1}/\mathcal{O}}&0&{\color{orange}\mathcal{O}/\mathfrak{m}}&0\\
 0&{\color{purple}\mathcal{O}/\mathfrak{m}}&0&{\color{purple}\mathcal{O}/\mathfrak{m}}}}}.\]
Since $\ell\geq 0$, the image of the left matrix is the Borel $B = \pmat{{\color{orange}\mathcal{O}/\mathfrak{m}}&{\color{orange}\mathfrak{m}/\mathfrak{m}^2}\\
 {\color{orange}0}&{\color{orange}\mathcal{O}/\mathfrak{m}}}$. The image on the second 
 matrix is either $\mathrm{GL}_2(\mathbb{F}_q)$ when $\ell+2k = 0$, or the upper-triangular Borel $B$ if $\ell+2k>0$. 
        \end{proof}

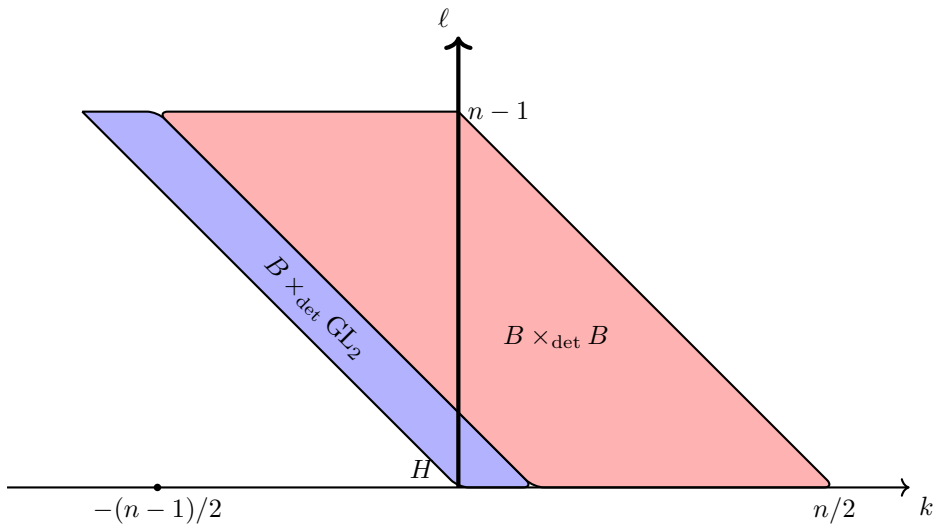
\begin{figure}[ht]
\centering
\begin{tikzpicture}
    \draw[rounded corners, thick, fill=blue!30] (-5,5) -- (0,0) -- (1,0) -- (-4,5) -- (-5,5);
    \draw[rounded corners, thick, fill=red!30] (0,5) -- (-4,5) -- (1,0) -- (5,0) -- (0,5);
    \node[above left] at (0,6) {$\ell$};
    \node[below right] at (6,0) {$k$};
    \node[below] at (5,0) {$n/2$};
    \node[right] at (0,5) {$n-1$};
        \node[below] at (-4,0) {$-(n-1)/2$};
    \node[left] at (-.2,.25) {$H$};
    \node at (1.3,2) {$B\times_{\det} B$};
    \node[rotate=-45] at (-1.9,2.4) {$B\times_{\det} \mathrm{GL}_2$};
    \draw[thick, ->] (-6,0)--(6,0);
    \draw[ultra thick, ->] (0,0)--(0,6);
    \node[circle,inner sep=1pt,fill=black] at (-4,0) {};
\end{tikzpicture}
\caption{Case of vertex of type 3 with the description of $\bar{P}_{\lambda}$}
\label{fig:t3}
\end{figure}


\begin{prop} For each $\lambda = (k,\ell)$ let $\mathcal{P}_{\lambda} = G(F)_{x_3,0:0+}/\bar{P}_\lambda$. We have
\[S_\lambda(X) = Z\backslash G(F)_{x_3,0:0+}\cong(\bar{Z}\backslash\mathrm{GL}_2\times_{\det}\mathrm{GL}_2)(\mathbb{F}_q),\ |S_\lambda| = q^2(q^2-1)^2.\]

\end{prop}
\begin{proof}
We have no condition on  $g\in \mathcal{P}_{\lambda}$ so that 
$g^{-1}Xg \in \mathfrak{g}(F)_{x_3,0:0+}$, so we always have 
$S_\lambda = G(F)_{x_3,0:0+}.$  We then have 
\[|\bar{Z}\backslash (\mathrm{GL}_2\times_{\det}\mathrm{GL}_2)(\mathbb{F}_q)| = \frac{|\mathrm{GL}_2(\mathbb{F}_q)|^2}{2(q-1)^2} = \frac{(q^2-1)^2(q^2-q)^2}{(q-1)^2} = q^2(q^2-1)^2.\]
\end{proof}

\begin{rem} Shouldn't we have a $2$ at the denominator? Since $|\bar{Z}| = 2(q-1)$ and not $q-1$. We have a slight abuse of notation, what we really want is to count the points on the image of 
 $Z\backslash P_\lambda$ over the residue field. We use the isogeny between $G$ and $\mathbb{G}_m\times G^{\mathrm{der}}$ to reduce this to the point count on a parabolic of the derived subgroup.

 One could make the choice to include the $2$, which would correspond to a slightly different choice of measure, differing from ours by a factor of $2$.
\end{rem}

Now we compute the dimension $d_\lambda$. And using our conditions, we can write it as
\begin{align*}
d_\lambda &= \min(\ell+1,n)+\min(|k-1|,n)+\min(|\ell+2k|,n)+\min(|\ell+k|,n)\end{align*}

This case is easier than the previous case, we can split the computations depending on whether $k\geq 0$ or $k\leq 0$, and we get 
\[\sum_\lambda q^{\lambda} = \frac{q^2(1-q^n)(1-q^{n+1})(1+q^n+q^{n+1})}{(1-q^2)(1-q^3)}.\]

We therefore have 
\[\mathrm{Orb}_{\mathfrak{g}(F)}(X,\mathds{1}_{\mathfrak{g}(\mathcal{O}_F)}) =q^2(q^2-1)^2\sum_\lambda q^{\lambda} = \frac{q^4(q+1)}{(q^2+q+1)}(q^n-1)(q^{n+1}-1)(1+q^n+q^{n+1}),\]

which we rewrite in its Shalika germ expansion below.
\begin{theorem}
In the case of vertices of type  $3$, we have 
\[\mathrm{Orb}_{\mathfrak{g}(F)}(X,\mathds{1}_{\mathfrak{g}(\mathcal{O}_F)}) =
\frac{q^4(q+1)}{(q^2+q+1)}\left(A(q) -B(q)q^{2n} +C(q)q^{3n}\right),\]
Where 
\begin{align*}A &= 1,\\
B &= 1+q+q^2,\\
C &=q(1+q).\\
\end{align*}
\end{theorem}
We can use Example \ref{ex:compunram} to see that in that case we get 
\[\mathrm{Orb}_{\mathfrak{h}(F)}(X,\mathds{1}_{\mathfrak{h}(\mathcal{O}_F)}) = 2 \frac{(q^n-1)(q^{n+1}-1)}{(q-1)^2}.\]
\begin{cor}
If $X$ has type $3$ then we have 
\[\mathrm{Orb}_{\mathfrak{g}(F)}(X,\mathds{1}_{\mathfrak{g}(\mathcal{O}_F)}) =(1+q^n+q^{n+1})\frac{(q^2-1)(q-1)q^4}{2(1+q+q^2)} \mathrm{Orb}_{\mathfrak{h}(F)}(X,\mathds{1}_{\mathfrak{h}(\mathcal{O}_F)}).\]
In particular, we can always pick measures so that 
\[\mathrm{Orb}_{\mathfrak{g}(F)}(X,\mathds{1}_{\mathfrak{g}(\mathcal{O}_F)}) =(1+q^n+q^{n+1}) \mathrm{Orb}_{\mathfrak{h}(F)}(X,\mathds{1}_{\mathfrak{h}(\mathcal{O}_F)}).\]
\end{cor}
\subsection{Vertex of type 4}

Let $X\in \mathfrak{g}(F)_{x_4,n+1/2}=\varpi^n \mathfrak{g}(F)_{x_4,1/2}$ be an elliptic regular semisimple element. We decompose 
\[G(F) = \prod_{k,\ell}G(F)_{x_4,0}\varpi^{k,\ell}G(F)_{x_1,0},\]
with $ k \geq 0$.\\ 

\begin{lemma}\label{lem:t4}
     Let $g\in G(F)_{x_4,0}$ and $Y = \mathrm{Ad}(g)X\in \mathfrak{g}(F)_{x_4,n+1/2}$, where $X\in \mathfrak{h}(F)\cap \mathfrak{g}(F)_{x_4,n+1/2}$.
      If $\mathrm{Ad}(\varpi^{\ell,k})Y\in \mathfrak{g}(F)_{x_1,0}$ for
       some $k\in\mathbb{Z}_{\geq 0}$.  
       One of the following must hold
       \begin{enumerate}[(i)]
        \item $n-\ell-2k\geq 0$ and $n+1+\ell\geq 0$,
        \item $n-\ell-2k<0$, $n-\ell-k\geq 0$, and $n+1+\ell\geq 0$,
        \item $n+1+\ell<0$, $n+1+\ell+k\geq 0$, and $n-\ell-2k\geq 0$.
    \end{enumerate}
\end{lemma}

We will need an intermediate result to prove the Lemma.  \\

\begin{lemma}\label{lem:quadsres}
    Let $V$ be a vector space over $\mathbb{F}_q$ of dimension $2$ and let $a: V\rightarrow V^*$, $b:V^*\rightarrow V$ are two self-adjoint isomorphisms such that $ba$ has distinct eigenvalues. Then there is no nonzero vector $v\in V$ so that $v$ is isotropic for $a$ and $av$ is isotropic for $b$, i.e.
    \[\langle v, av\rangle = 0 = \langle bav, av \rangle.\]
\end{lemma}
\begin{proof} 
One can show the lemma directly by writing quadratic forms corresponding to $a$ and $b$ in coordinates corresponding to a basis with first element $v$ We give another proof below.

    Define the pairing inner product $\langle \ ,\ \rangle_a : V\times V\rightarrow \mathbb{F}_q$ defined by $\langle x, y\rangle_a = \langle x, ay\rangle$. Assume there is $v\in V$ satisfying the conditions of the Lemma. 
    
    The condition becomes  $v\perp_a v$ and  $ba v \perp_a v$. Since $a$ is an isomorphism, we get that $ba\in v^{\perp_a} =  \mathrm{Span}_{\mathbb{F}_q}(v)$ and therefore $v$ is an eigenvector of $ba$.
    Note that $ba$ has distinct eigenvalues hence it is diagonalizable, moreover $\langle bax,y\rangle_a = \langle x, ba y \rangle_a$ hence $ba$'s eigenvectors form an orthogonal basis of $(V, \langle\ ,\ \rangle_a)$. If $w\neq 0$ an eigenvector of $ba$ so that $\{v,w\}$ is an orthogonal
     basis of $(V, \langle\ ,\ \rangle_a)$, then $ba w \in \mathrm{Span}_{\mathbb{F}_q}(w)\cap v^{\perp_a} = \{0\}$ so $ab$ has $0$ as an eigenvalue hence is not invertible, which is impossible.
\end{proof}

\begin{proof}[Proof of Lemma \ref{lem:t4}]
First observe that since $g\in G(F)_{x_4,0}$, we have $\bar{g}\in G(F)_{x_4,0:0+}\cong(\mathrm{GL}_3\times \mathrm{GL}_1)(\mathbb{F}_q)$ which can be written 
\[\bar{g} = \pmat{M&0\\0&\lambda({}^tM)^{-1} },\]
where $\lambda\in \mathrm{GL}_1(\mathbb{F}_q)$ and $M\in \mathrm{GL}_2(\mathbb{F}_q)$. Since $X\in \mathfrak{h}(F)$, we have 
\[X = \pmat{\varpi^{n+1} D_1&\varpi^{n+1}D_2\\ \varpi^n D_3&\varpi^{n+1} D_4},\]
for diagonal matrices $D_1,D_2,D_3,D_4$. Therefore 
\[\overline{{g}^{-1}Xg} =\left(\begin{array}{c|c}
0&\overline{\lambda M^{-1}D_2 ({}^tM)^{-1}}\\ \hline \overline{\lambda^{-1} ({}^tM)D_3M}&0
\end{array}\right) \in \left(\begin{array}{c|c}
0&\mathrm{Sym}_2(\mathbb{F}_q)\\ \hline \mathrm{Sym}_2(\mathbb{F}_q)&0
\end{array}\right) .\]
Letting $L$ be the Lagrangian generated by $e_1,e_2$, we can rewrite
\begin{equation}\label{eq:redquad}\overline{\mathrm{Ad}(g){X}} = \left(\begin{array}{c|c}
0&a\\ \hline b&0
\end{array}\right),\end{equation}
where $a: \bar{L}\rightarrow \bar{L}^*$ and $a: \bar{L}^*\rightarrow\bar{L}$ are two self-adjoint isomorphisms. Since $\bar{X}$ is regular, we need $ba$ to have distinct eigenvalues.

We know that $Y\in \mathfrak{g}(F)_{x_4:n+1/2}$ and so 
\[\mathrm{Ad}(\varpi^{\ell,k})Y\in \begin{bmatrix}n+1& n+1-k&n+1+\ell&n+1+\ell+k\\ n+1+k & n+1&n+1+\ell+k& n+1+\ell+2k\\n -\ell&n-k-\ell&n+1&n+1+k\\n-k-\ell&n-2k-\ell &n+1-k&n+1\end{bmatrix}\cap \mathfrak{g}(F)_{x_1,0}. \]
If $(i)$ holds, then there are no conditions and we have  
\[\bar{Y} \in \begin{bmatrix} 0&0&\star&\star\\0&0&\star&\star\\\star&\star&0&0\\\star&\star&0&0 \end{bmatrix}.\]

Assume we are not in case $(i)$. 
\begin{itemize}
    \item {\bf Case $1$.} $n-\ell-2k<0$. This means that $\nu(Y_{4,2})>n$ hence 
    \[\bar{Y}\in \begin{bmatrix} 0&0&\star&\star\\0&0&\star&\star\\\star&\star&0&0\\\star&0&0&0 \end{bmatrix}\]
    and $e_2$ is an isotropic vector for $b$ (using the notation of equation (\ref{eq:redquad}). Since $b$ is nondegenerate we get that $\nu(Y_{4,1})= n$ therefore $n-k-\ell\geq 0$. Lemma \ref{lem:quadsres} can be applied since $X$ is regular, and it tells us that $e_2^*=e_1$ cannot be an isotropic vector for $a$, therefore $\nu(Y_{1,3}) = n+1$ and we need $n+1+\ell\geq 0$.

    \item {\bf Case $2$.} $n+1+\ell<0$. This means that $\nu(Y_{2,4})>n+1$. We follow the same proof. Since $b$ is nondegenerate we have $\nu(Y_{1,4}) = n+1$ hence we need $n+1+\ell+k\geq 0$. Using Lemma \ref{lem:quadsres} with isotropic vector $e_1$ we find that $\nu(Y_{4,2})=n$ and so $n-2k-\ell\geq 0$.

    This corresponds to 
    \[\bar{Y} \in \begin{bmatrix} 0&0&0&\star\\0&0&\star&\star\\\star&\star&0&0\\\star&\star&0&0 \end{bmatrix}.\]
\end{itemize}
This concludes the proof.
\end{proof}

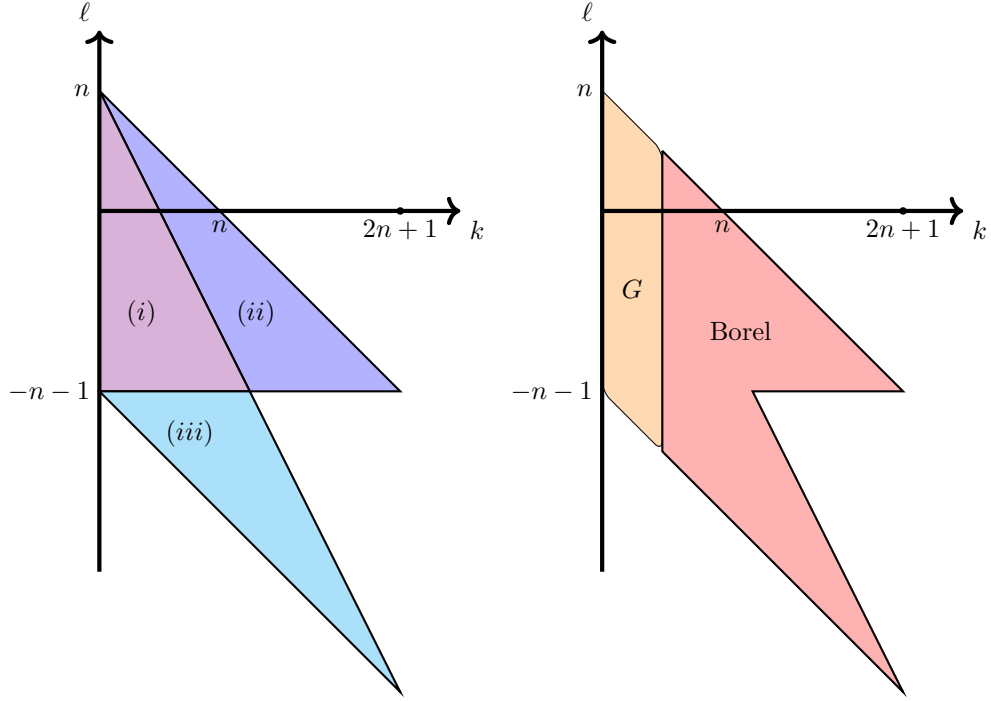
\begin{figure}[ht]
\centering 
\begin{tikzpicture}[scale=.8]
    \draw[thick, fill=blue!30] (0,2) -- (5,-3) -- (2.5,-3) -- (0,2);
    \draw[thick, fill=violet!30] (0,2) -- (2.5,-3) -- (0,-3) -- (0,2);
    \draw[thick, fill=cyan!30] (2.5,-3) -- (0,-3) -- (5,-8) -- (2.5,-3);
    \node[above left] at (0,3) {$\ell$};
    \node[below right] at (6,0) {$k$};
    \node[below] at (2,0) {$n$};
    \node[below] at (5,0) {$2n+1$};
    \node[left] at (0,-3) {$-n-1$};
    \node[circle,inner sep=1pt,fill=black] at (5,0) {};
    \node[left] at (0,2) {$n$};
    \node at (0.7,-1.7) {$(i)$};
    \node at (2.6,-1.7) {$(ii)$};
    \node at (1.5,-3.7) {$(iii)$};
    \draw[ultra thick, ->] (0,0)--(6,0);
    \draw[ultra thick, ->] (0,-6)--(0,3);
\end{tikzpicture}
\begin{tikzpicture}[scale=.8]
    \draw[rounded corners, fill=orange!30] (0,2) -- (1,1) -- (1,-4) -- (0,-3) -- (0,2);
    \draw[thick, fill=red!30] (1,1) -- (5,-3) -- (2.5,-3) -- (5,-8)--(1,-4)--(1,1);
    \node[above left] at (0,3) {$\ell$};
    \node[below right] at (6,0) {$k$};
    \node[below] at (2,0) {$n$};
    \node[below] at (5,0) {$2n+1$};
    \node[left] at (0,-3) {$-n-1$};
    \node[circle,inner sep=1pt,fill=black] at (5,0) {};
    \node[left] at (0,2) {$n$};
    \node at (0.5,-1.3) {$G$};
    \node at (2.3,-2) {Borel};
    \draw[ultra thick, ->] (0,0)--(6,0);
    \draw[ultra thick, ->] (0,-6)--(0,3);
\end{tikzpicture}
\caption{Case of vertex of type 4 with $\bar{P}_\lambda$ on the right}
\label{figt4}
\end{figure}

\begin{lemma}
        Let $\lambda = (k,\ell)$ as described in Lemma \ref{lem:t4} and let $B$ denote the upper-triangular Borel of $\mathrm{GL}_2(\mathbb{F}_q)$. We have the following:
        \begin{itemize}
        \item If $k=0$ then $\bar{P}_\lambda\cong (\mathrm{GL}_2\times\mathrm{GL}_1)(\mathbb{F}_q)$ and therefore  $G(F)_{x_4,0:0+}/\bar{P}_\lambda \cong \{\mathrm{pt}\})$,
        \item Otherwise, $\bar{P}_\lambda\cong (B\times \mathrm{GL}_1)(\mathbb{F}_q)$, and therefore  $G(F)_{x_4,0:0+}/\bar{P}_\lambda \cong  \mathbb{P}^1(\mathbb{F}_q)$,
        \end{itemize}
as in Figure \ref{figt4}.
\end{lemma}
\begin{proof}
Let $\lambda = (k,\ell)$. We again use the notation $\mathcal{O}= \mathcal{O}_F$ and $\mathfrak{m} = \varpi\mathcal{O}_F$. Look at the reduction map 
\begin{align*}
P_\lambda &= \varpi^{\lambda}G(F)_{x_1,0}\varpi^{-\lambda}\cap G(F)_{x_4,0}\rightarrow \bar{P}_\lambda \subset G(F)_{x_4,0:0+} \cong ({\mathrm{GL}_2}\times{\mathrm{GL}_1})(\mathbb{F}_q)\\
\end{align*}
\[\pmat{ \mathcal{O}&\mathfrak{m}^{k}&\star&\star\\
 \mathfrak{m}^{-k}&\mathcal{O}&\star&\star\\
 \star&\star&\mathcal{O}&\mathfrak{m}^{-k}\\
 \star&\star&\mathfrak{m}^{k}&\mathcal{O}}
\cap \pmat{ \mathcal{O}&\mathcal{O}&\star&\star\\
 \mathcal{O}&\mathcal{O}&\star&\star\\
 \star&\star&\mathcal{O}&\mathcal{O}\\
 \star&\star&\mathcal{O}&\mathcal{O}}
\rightarrow
\pmat{\mathcal{O}/\mathfrak{m}&\mathcal{O}/\mathfrak{m}&0&0\\ \mathcal{O}/\mathfrak{m}&\mathcal{O}/\mathfrak{m}&0&0\\ 0&0&\mathcal{O}/\mathfrak{m}&\mathcal{O}/\mathfrak{m}\\
0&0&\mathcal{O}/\mathfrak{m}&\mathcal{O}/\mathfrak{m}}.\]
We see that the image is determine by whether $k=0$ or $k>0$ as desired.
        \end{proof}
        
\begin{prop} For each $\lambda = (k,\ell)$ we have
\[|\bar{S}_\lambda(X)| =\left\lbrace \begin{array}{ll}
q(q+1)(q-1)^2& \text{ in case }(i)\\
2q(q-1)^2&\text{ else}\end{array}\right..\]
\end{prop}
\begin{proof}Let us first describe the conditions on $g\in G(F)_{x_4,0:0+}$ so that  $Y = g^{-1}Xg$ is in the desired subalgebra of $\mathfrak{g}(F)_{x_1,0:0+}$ as described in Lemma \ref{lem:t4}.
\begin{itemize}
\item Case $(i)$ has no conditions so $|\bar{S}_\lambda(X)| = |\bar{Z}\backslash G(F)_{x_4,0:0+}| = \frac{(q^2-1)(q^2-q)(q-1)}{(q-1)}$.
\item Case $(ii)$ $Y\in \pmat{0&0&\star&\star\\0&0&\star&\star\\ \star&\star&0&0\\\star&0&0&0}$. We get that $\bar{Y}$ correspond to a degenerate quadratic form, and $\bar{S}_\lambda(X)/\bar{P}_\lambda$ 
corresponds to the subset of $\mathbb{P}^1$ of isotropic lines. Since $\bar{Y}$ is nonzero, we have $|\bar{S}_\lambda(X)/\bar{P}_\lambda| = 2$ hence $|\bar{S}_\lambda(X)| = 2|Z\backslash \bar{P}_\lambda| = \frac{2q(q-1)^3}{(q-1)}$.
\item Case $(iii)$ is identical to case $(ii)$.
\end{itemize}
\end{proof}

We compute the dimension as 
\begin{align*}
d_\lambda &= \min(\max(\ell,-\ell-1),n)+\min(k,n+1)\\
&+\min(\max(\ell+2k,-\ell-2k-1),n)+\min(\max(\ell+k,-l-k-1),n)\end{align*}

The decomposition is slightly technical, we divide each region in three parts, noting that 
the contribution of case $(ii)$ and $(iii)$ are equal.  We get

\[\sum_{\lambda\ \text{type }(i)} q^{d_\lambda} = 
2\frac{1-q^{2+2n}-q^{3+2n}+q^{4+2n}+q^{4+3n}+q^{5+3n}}{(1-q^2)(1-q^3)},\]
and
\[\sum_{\lambda\ \text{type }(ii)} q^{d_\lambda}= \sum_{\lambda\ \text{type }(iii)} q^{d_\lambda} = 
\frac{q^{1+2n}(1-q^{1+n})^2}{(1-q)^2}.\]

Summing all together, weighing each term by the appropriate point count, we get

\[\mathrm{Orb}_{\mathfrak{g}(F)}(X,\mathds{1}_{\mathfrak{g}(\mathcal{O}_F)}) = 
2\frac{q(q^{n+1}-1)(-1-q^{1+n}-2q^{1+2n}-q^{3+2n}+q^{4+2n}+2q^{2+3n}+2q^{3+3n}+2q^{4+3n})}{q^2+q+1},\]

which we rewrite in its Shalika germ expansion below.
\begin{theorem}
In the case of vertices of type  $4$, we have 
\[\mathrm{Orb}_{\mathfrak{g}(F)}(X,\mathds{1}_{\mathfrak{g}(\mathcal{O}_F)}) =
\frac{2q}{q^3-1}\left(A(q)-B(q)q^{2n}+ C(q)q^{3n}+D(q)q^{4n}\right),\]
Where 
\begin{align*}A &= 1 ,\\
B &= q(q-2)(q^2+q+1),\\
C &= q^2(q^3-3q^2-4q-4)\\
D&=  2q^3(1+q+q^2).\\
\end{align*}

\end{theorem}

We can use Example \ref{ex:compram} to see that in that case we get 
\[\mathrm{Orb}_{\mathfrak{h}(F)}(X,\mathds{1}_{\mathfrak{h}(\mathcal{O}_F)}) =  \frac{(q^{n+1}-1)^2}{(q-1)^2}.\]

The comparison is not as straightforward as in Case 3, which makes sense since the reductive quotient for vertices of type 3 matches the 
desired distinguished subgroup.

\appendix

\section{Combinatorics on the regular tree}
\label{sec:combin-tre}

In this section, we consider $X$, and infinite $q+1$ regular tree, for some integer $q$, equipped with the combinatorial metric. The following 
results describe various subgraphs and count their vertices, with applications to integrals computed in \S \ref{sec:kottmethod}, \S \ref{sec:exporbit}, as well as Appendix \ref{sec:relative}.

\subsection{Number of vertices in a ball.}
For any vertex $v$ of $X$ and any $r\geq 0$ let  $B_r(v)$ be set of vertices in the 
ball centered at $v$ of radius $r$. 
\subsubsection{Ball centered at a vertex}
First we look at balls centered at a vertex.

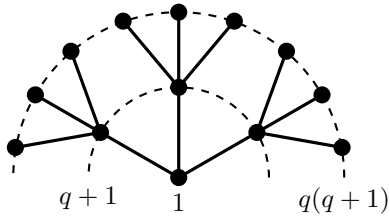
\begin{figure}[H]
\centering
\begin{tikzpicture}
\begin{scope}[rotate=90]
\foreach \i in{-60,0,60}{
\draw[very thick] (0,0) -- (\i:1.2);

\begin{scope}[shift=(\i:1.2), rotate=\i]

\draw[very thick] (0,0) -- (0:1);
\node[draw,circle,inner sep=2pt,fill=black, thick] at (0:1) {};

\draw[very thick] (0,0) -- (-40:1.15);
\node[draw,circle,inner sep=2pt,fill=black, thick] at (-40:1.15) {};

\draw[very thick] (0,0) -- (40:1.15);
\node[draw,circle,inner sep=2pt,fill=black, thick] at (40:1.15) {};

\node[draw,circle,inner sep=2pt,fill=black, thick] at (0,0) {};
\end{scope}}
\node[draw,circle,inner sep=2pt,fill=black, thick] at (0,0) {};
\end{scope}
\node[below] at (0,-0.1) {$1$};
\node[below] at (-1.2,0) {$q+1$};
\node[below] at (2.2,0) {$q(q+1)$};
 \clip (-3,0) rectangle (3,3);
    \draw[thick, dashed] (0,0) circle(1.2);
    \draw[thick, dashed] (0,0) circle(2.2); 
    \end{tikzpicture}
\caption{Vertices in a ball of radius 2, with $q=2$}\end{figure}

\begin{prop} Let $v$ be a vertex. The number of vertices $B_{i}(v)$ is equal to
    \[1+(q+1)\frac{q^i-1}{q-1}.\]
\end{prop}
\begin{proof}
The number of vertices at distance $j>0$ is equal to $(q+1)q^{j-1}$. The total number of vertices in the ball  
\[1+ (q+1)(1+\cdots q^{i-1}) =1+(q+1)\frac{q^i-1}{q-1}.\]
\end{proof}

We will also need to fix a geodesic on the tree and only count the vertices out of the apartment.  
\begin{prop}\label{prop:vertballoutofapt} Let $v$ be a vertex and $A$ be an apartment with $v\in A$. The number of vertices $w$ in $B_i(v)$ so that $\mathrm{d}(w,A) = \mathrm{d}(w,v)$ is equal to 
    \[1+(q-1)\frac{q^i-1}{q-1} = q^i.\]
\end{prop}
\begin{proof}
    As in the previous proposition the number we are computing is 
    \[1+(q-1)(1+\cdots+q^{i-1}) = 1+(q-1)\frac{q^{i}-1}{q-1} = q^i.\]
\end{proof}

\subsubsection{Ball centered at a midpoint}

Some cases have us computing the number of vertices in a ball centered at the midpoint of a vertex. To this extent, we prove the following Proposition. 

\begin{figure}[H]
\centering
\begin{tikzpicture}
\draw[very thick] (-2.5,0) -- (2.5,0);
\foreach \i in {-2.5,-1.5,...,2.5}{
\node[draw,circle,inner sep=2pt,fill=black, thick] at (\i,0) {};}
\node[draw,circle,inner sep=1pt,fill=blue, thick] at (0,0) {};
\node[below] at (0,0) {$x$};

\begin{scope}[shift=(0:1.5)]
\draw[very thick] (0,0) -- (-45:1.2);
\node[draw,circle,inner sep=2pt,fill=black, thick] at (-45:1.2) {};
\end{scope}
\begin{scope}[shift=(180:1.5)]
\draw[very thick] (0,0) -- (225:1.2);
\node[draw,circle,inner sep=2pt,fill=black, thick] at (225:1.2) {};
\end{scope}
\begin{scope}[shift=(0:.5)]
    \draw[very thick] (0,0) -- (-90:1.4);
    \node[draw,circle,inner sep=2pt,fill=black, thick] at (-90:1.4) {};
    \begin{scope}[shift=((-90:1.4))]
        \draw[very thick] (0,0) -- (-90:1.05);
        \node[draw,circle,inner sep=2pt,fill=black, thick] at (-90:1.05) {};
        \draw[very thick] (0,0) -- (-45:1.1);
        \node[draw,circle,inner sep=2pt,fill=black, thick] at (-45:1.1) {};
    \end{scope}
\end{scope}
\begin{scope}[shift=(180:.5)]
    \draw[very thick] (0,0) -- (-90:1.4);
    \node[draw,circle,inner sep=2pt,fill=black, thick] at (-90:1.4) {};
    \begin{scope}[shift=((-90:1.4))]
        \draw[very thick] (0,0) -- (-90:1.05);
        \node[draw,circle,inner sep=2pt,fill=black, thick] at (-90:1.05) {};
        \draw[very thick] (0,0) -- (225:1.1);
        \node[draw,circle,inner sep=2pt,fill=black, thick] at (225:1.1) {};
    \end{scope}
\end{scope}
\node[above] at (.5,0) {$2$};
\node[above] at (1.5,0) {$2q$};
\node[above] at (2.5,0) {$2q^2$};
\clip (-3,0) rectangle (3,-3);
    \draw[thick, dashed] (0,0) circle(.5);
    \draw[thick, dashed] (0,0) circle(1.5);
    \draw[thick, dashed] (0,0) circle(2.5);





    \end{tikzpicture}
\caption{Vertices in a ball of radius 2, with $q=2$}\end{figure}
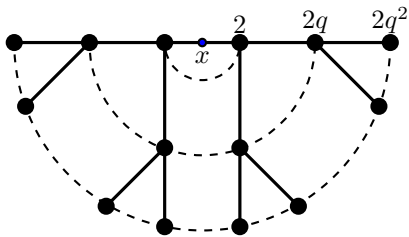

\begin{prop}
    Let $x$ be the midpoint of a vertex and let $i\in \mathbb{Z}_{>0}$. The number of vertices at distance at most $i$ from $x$ is equal to 
    \[2\frac{q^{i}-1}{q-1}.\]
\end{prop}
\begin{proof}
    We count the number of vertices in $B_i(x)$ closest to a fixed side of the edge containing $x$ as 
    \[1+q+\cdots+q^{i-1} = \frac{q^i-1}{q-1}.\]
    The edge has two endpoints so we multiply the result by $2$.
\end{proof}

\subsection{Intersection of balls}

\begin{lemma}\label{lem:interballs}
Let $x,y$ be two points (not necessarily vertices) in a $q+1$-regular tree at distance $\delta$. Let $\alpha, \beta\in \mathbb{Z}_{\geq 0}$ be two nonnegative integers.  Then 
\begin{enumerate}
\item if $\alpha+\beta< \delta$ then $B_\alpha(x)\cap B_\beta(y) = \emptyset$.
\item if $|\alpha-\beta|\leq \delta\leq \alpha+\beta$  then $B_\alpha(x)\cap B_\beta(y) = B_{\frac{1}{2}(\alpha+\beta-\delta)}(c)$ where $c$ a point on the geodesic $[x,y]$ at distance $\frac{1}{2}(\alpha-\beta+r)$ from $x$ (and therefore at distance $\frac{1}{2}(\beta-\alpha+r)$ from $y$. 
\item if $\delta \leq |\alpha- \beta|$ then $B_\alpha(x)\cap B_\beta(y) = B_\beta(y)$ if $\alpha\geq \beta$ or $B_\alpha(x)$ if $\alpha\leq \beta$.
\end{enumerate}
\end{lemma}
\begin{proof}
We get Case 1 by the triangle inequality. Similarly, Case 3 is immediate, since in that case, if $\alpha\geq \beta$ then $B_\beta(y)\subset B_\alpha(x)$, and if $\beta\geq\alpha$ then the reverse inclusion is true.

Assume that  $\alpha-\beta\leq \delta\leq \alpha+\beta$. Define $c$ as in the Lemma and let $r= \frac{1}{2}(\alpha+\beta-r)\geq 0$.
\begin{itemize}
\item Let us show that $B_{r}(c)\subset B_{\alpha}(x)\cap B_\beta(y)$. Let $v\in B_r(c)$. Then by the triangle inequality we have 
\[\mathrm{d}(v,x) \leq \mathrm{d}(v,c)+\rd(c,x)\leq \frac{1}{2}(\alpha+\beta-\delta)+\frac{1}{2}(\alpha-\beta+\delta) = \alpha.\]
Similarly $\rd(v,y)\leq \beta$ hence $v\in B_\alpha(x)\cap B_\beta(y)$.
\item We now prove the reverse inclusion. Let $v\in B_\alpha(x)\cap B_\beta(y)$. Let $w$ be the point on $[x,y]$ the closest to $v$. 
\begin{itemize}
\item If $w\in[x,c]$ then 
\[\rd(v,c) = \rd(v,w)+\setbrace{=\rd(x,c)-\rd(x,w)}{\rd(w,c)} = {\color{red}\rd(x,c)-\rd(x,w)}+\rd(x,c).\]
But also notice that ${\color{red}\rd(x,c)-\rd(x,w)}+\delta = \rd(v,y)\leq \beta$, hence 
\[\rd(v,c) \leq (\beta-\delta) + \rd(x,c)= (\beta-\delta) + \frac{1}{2}(\alpha-\beta+r) = r.\]
\item If $w\in [y,x]$ then the same reasoning swapping $x$ and $y$ yields $\rd(v,c)\leq r$ again. 
\end{itemize}
Therefore,  $v\in B_r(c)$ so we are done. 
\end{itemize}
\end{proof}

\begin{lemma}\label{lem:intergeoball} Let $A$ be an infinite geodesic in a $q+1$-regular tree. Let $b$ be a point and $\alpha,\beta\in \mathbb{R}_{\geq 0}$. Let $\delta$ denote the distance between $A$ and $v$ and let $a\in A$ such that $\rd(a,b)=\delta$. 
Let $X$ be the set of points at distance at most $\alpha$ from $A$ and at most $\beta$ from $b$. 
\begin{enumerate}
\item The inclusion $B_\alpha(a)\cap B_\beta(b)\subset X$ always holds.
\item If $\delta \geq \beta-\alpha$ then the reverse inclusion holds.
\item If $\delta < \beta-\alpha$ then $B_\alpha(a)\subset X\subset B_\beta(b)$ and each inclusion is strict.
\end{enumerate}
\end{lemma}
\begin{proof}
Vertices in $B_\alpha(a)\cap B_\beta(b)$ are at distance at most $\alpha$ from $a$ hence at distance at most $\alpha$ from $A$, and also immediately at distance at most $\beta$ from $b$. This proves the first point. \\

For the third point, we have $\beta>\delta+\alpha$ which means that $B_{\alpha}(a)\subset B_{\beta}(b)$ strictly and so $B_\alpha(a) = B_{\alpha}(a)\cap B_{\beta}(b) \subset X$. \\

Now assume $\delta \geq \beta-\alpha$ i.e. $\beta-\delta\leq \alpha$. Let $v\in X$. Since $X\subset B_\beta(b)$ we have $v\in B_\beta(b)$. We need to show $v\in B_\alpha(a)$. Let $w\in A$ the vertex closest to $v$.

If $w = a$ then we immediately have $\rd(v,a)= \rd(v,A)\leq \alpha$. \\

If $w\neq a$ then 
\[\rd(v,b) = \rd(v,b)+\rd(a,b) = \rd(v,a)+\delta\leq \beta,\]
so 
\[\rd(v,a)\leq \beta-\delta\leq \alpha\]
by assumption. We get the desired result. \end{proof}

\subsection{Elements at a fixed distance from a convex set}

In this section, for any subgraph $S\subset X$,  we let $V_S$ denote its set of vertices and $E_S$  its set of edges.

\begin{prop}\label{prop:coloredawayfromconvex}
    Assume that the $(q+1)$-regular tree is colored with alternating black and white colors. Let $S$ be a convex set, let $V_S^{\scalebox{0.5}{\CIRCLE}}$ the set of black vertices of $S$, and $V_S^{\scalebox{0.5}{\Circle}}$ the set of white vertices. If $n>0$, the number of black vertices at distance $n$ from $S$ is equal to 
    \[\left\lbrace\begin{array}{ll}
        q^{n-1}\left(q|V_S^{\scalebox{0.5}{\CIRCLE}}|-|V_S^{\scalebox{0.5}{\Circle}}|+1\right) &\text{if } n\equiv 0\pmod2\\ &\\
          q^{n-1}\left(q|V_S^{\scalebox{0.5}{\Circle}}|-|V_S^{\scalebox{0.5}{\CIRCLE}}|+1\right) &\text{if } n\equiv 1\pmod2
    \end{array}\right.\]
\end{prop}
\begin{proof} Call $V_m^{\scalebox{0.5}{\CIRCLE}}$ such a set.
  Let $v\in V_S$. Let $C_v$ be the number of neighbors of $v$ contained in $S$.
  
  We get that $(q+1)-C_v$ neighbors are outside of $S$. 
  Each of those vectors has a different color than $v$, and has $q$ neighbors that are still outside of $S$ and are the same color as 
  $v$,   and so on since the tree has no cycle. 
  We get that there are $q^{n-1}((q+1)-C_v)$ vertices at distance $n$ from $V_S$ 
  whose closest vertex in $S$ is $v$. Their color depends on the parity of $n$.

  Therefore, 
    \begin{align*}
    |V_m^{\scalebox{0.5}{\CIRCLE}}|&=\left\lbrace\begin{array}{ll}
        \sum_{v\in V_S^{\scalebox{0.5}{\CIRCLE}}}q^{n-1}((q+1)-C_v) & n\equiv 0\pmod2\\ &\\
         \sum_{v\in V_S^{\scalebox{0.5}{\Circle}}}q^{n-1}((q+1)-C_v)& n\equiv 1\pmod2
    \end{array}\right.\\
    &= \left\lbrace\begin{array}{ll}
        q^{n-1}\left( (q+1)|V_S^{\scalebox{0.5}{\CIRCLE}}| -\sum_{v\in V_S^{\scalebox{0.5}{\CIRCLE}}}C_v\right)& n\equiv 0\pmod2\\ &\\
         q^{n-1}\left( (q+1)|V_S^{\scalebox{0.5}{\Circle}}| -\sum_{v\in V_S^{\scalebox{0.5}{\Circle}}}C_v\right)& n\equiv 1\pmod2
    \end{array}\right.\\
    &=\left\lbrace\begin{array}{ll}
        q^{n-1}\left( (q+1)|V_S^{\scalebox{0.5}{\CIRCLE}}| -|E_S|\right)& n\equiv 0\pmod2\\ &\\
         q^{n-1}\left( (q+1)|V_S^{\scalebox{0.5}{\Circle}}| -|E_S|\right)& n\equiv 1\pmod2
    \end{array}\right..\\
    \end{align*}
    The last equality is due to the fact that each edge in $S$ has exactly two vertices, each of them having a different color. 
    
    Since $S$ is a tree, its Euler characteristic is $1$, therefore $|V_S| -|E_S|= 1$, hence 
    \[|V_m^{\scalebox{0.5}{\CIRCLE}}|=\left\lbrace\begin{array}{ll}
        q^{n-1}\left( (q+1)|V_S^{\scalebox{0.5}{\CIRCLE}}| -|V_S|+1\right)& n\equiv 0\pmod2\\ &\\
         q^{n-1}\left( (q+1)|V_S^{\scalebox{0.5}{\Circle}}| -|V_S|+1\right)& n\equiv 1\pmod2
    \end{array}\right.\]

    We can conclude by writing $|V_S| = |V_S^{\scalebox{0.5}{\CIRCLE}}|+|V_S^{\scalebox{0.5}{\Circle}}|$. 
\end{proof}

Note that the convex condition is equivalent to $S$ being connected since geodesics are unique in the tree. 

Moreover, summing both formulas, and using $V_S = V_S^{\scalebox{0.5}{\Circle}}\sqcup V_S^{\scalebox{0.5}{\CIRCLE}}$, we recover the result {\cite[p.416]{kott_bible}} stating that 
the total number of vertices at distance $n$ from $S$ is 
\begin{equation}\label{prop:pointfromconvex}q^{n-1}((q-1)|V_S|+2).\end{equation}

\subsection{Intersection of a tube and a ball.}\label{sec:tuball}

Consider a point $v$ at distance $\delta$ from an infinite geodesic $A$ (called an appartment). We want to compute the number of vertices at distance at most $\alpha$ from $A$ and $\beta$ from $v$. If $\delta>0$, let $a_0\in A$ be the closest vertex to $v$ and then label vertices of $A$ by $a_i$ where $i\in\mathbb{Z}$. This section will prove the following:

\begin{figure}[H]
\centering
\begin{tikzpicture}[scale=0.85]

\draw[ultra thick] (0,0) -- ++(10:2) -- ++(-10:2) -- ++(10:2);
\draw[ultra thick, dashed] (5.9,0.35) -- ++(0:2);
\draw[ultra thick, dashed] (0,0) -- ++(180:2);
\draw[very thick] (0,0) -- ++(130:1.5);
\draw[very thick] (0,0) -- ++(-130:1.5);
\node[draw,circle,inner sep=2pt,fill=black, thick] at (0,0) {};
\begin{scope}[shift=(130:1.5),rotate=130]
\draw[very thick] (0,0) -- (1.2,0);
\draw[very thick] (0,0) -- (40:1.2);
\draw[very thick] (0,0) -- (-40:1.2);
\foreach \i in {-40,0,...,40}{
\begin{scope}[shift=(\i:1.2),rotate=\i]
\draw[thick] (0,0) -- (1,0);
\draw[thick] (0,0) -- (30:1);
\draw[thick] (0,0) -- (-30:1);
\foreach \j in {-30,0,...,30}{
\begin{scope}[shift=(\j:1),rotate=\j]
\draw[dashed] (0,0) -- (0.4,0);
\draw[dashed] (0,0) -- (20:0.4);
\draw[dashed] (0,0) -- (-20:0.4);
\node[draw,circle,inner sep=2pt,fill=white, thick] at (0,0) {};
\end{scope}
\node[draw,circle,inner sep=2pt,fill=black, thick] at (0,0) {};
}
\end{scope}
}
\node[draw,circle,inner sep=2pt,fill=black, thick] at (0,0) {};
\end{scope}

\begin{scope}[shift=(-130:1.5),rotate=-130]
\draw[very thick] (0,0) -- (1.2,0);
\draw[very thick] (0,0) -- (40:1.2);
\draw[very thick] (0,0) -- (-40:1.2);
\foreach \i in {-40,0,...,40}{
\begin{scope}[shift=(\i:1.2),rotate=\i]
\draw[thick] (0,0) -- (1,0);
\draw[thick] (0,0) -- (30:1);
\draw[thick] (0,0) -- (-30:1);
\foreach \j in {-30,0,...,30}{
\begin{scope}[shift=(\j:1),rotate=\j]
\draw[dashed] (0,0) -- (0.4,0);
\draw[dashed] (0,0) -- (20:0.4);
\draw[dashed] (0,0) -- (-20:0.4);
\node[draw,circle,inner sep=2pt,fill=white, thick] at (0,0) {};
\end{scope}
\node[draw,circle,inner sep=2pt,fill=black, thick] at (0,0) {};
}
\end{scope}
}
\node[draw,circle,inner sep=2pt,fill=black, thick] at (0,0) {};
\end{scope}

\draw[very thick] (10:2) -- ++(110:1.2);
\draw[very thick] (10:2) -- ++(-110:1.2);
\node[draw,circle,inner sep=2pt,fill=black, thick] at (10:2) {};
\begin{scope}[shift=(10:2)]
\begin{scope}[shift=(110:1.2),rotate=110, scale=0.7]
\draw[very thick] (0,0) -- (1.2,0);
\draw[very thick] (0,0) -- (40:1.2);
\draw[very thick] (0,0) -- (-40:1.2);
\foreach \i in {-40,0,...,40}{
\begin{scope}[shift=(\i:1.2),rotate=\i]
\draw[thick] (0,0) -- (1,0);
\draw[thick] (0,0) -- (30:1);
\draw[thick] (0,0) -- (-30:1);
\foreach \j in {-30,0,...,30}{
\begin{scope}[shift=(\j:1),rotate=\j]
\draw[dashed] (0,0) -- (0.4,0);
\draw[dashed] (0,0) -- (20:0.4);
\draw[dashed] (0,0) -- (-20:0.4);
\node[draw,circle,inner sep=2pt,fill=white, thick] at (0,0) {};
\end{scope}
\node[draw,circle,inner sep=2pt,fill=black, thick] at (0,0) {};
}
\end{scope}
}
\node[draw,circle,inner sep=2pt,fill=black, thick] at (0,0) {};
\end{scope}
\begin{scope}[shift=(-110:1.2),rotate=-110, scale=0.7]
\draw[very thick] (0,0) -- (1.2,0);
\draw[very thick] (0,0) -- (40:1.2);
\draw[very thick] (0,0) -- (-40:1.2);
\foreach \i in {-40,0,...,40}{
\begin{scope}[shift=(\i:1.2),rotate=\i]
\draw[thick] (0,0) -- (1,0);
\draw[thick] (0,0) -- (30:1);
\draw[thick] (0,0) -- (-30:1);
\foreach \j in {-30,0,...,30}{
\begin{scope}[shift=(\j:1),rotate=\j]
\draw[dashed] (0,0) -- (0.4,0);
\draw[dashed] (0,0) -- (20:0.4);
\draw[dashed] (0,0) -- (-20:0.4);
\node[draw,circle,inner sep=2pt,fill=white, thick] at (0,0) {};
\end{scope}
\node[draw,circle,inner sep=2pt,fill=black, thick] at (0,0) {};
}
\end{scope}
}
\node[draw,circle,inner sep=2pt,fill=black, thick] at (0,0) {};
\end{scope}
\end{scope}

\draw[very thick] (0:3.94) -- ++(80:1.2);
\draw[very thick] (0:3.94) -- ++(-80:1.2);
\node[draw,circle,inner sep=2pt,fill=black, thick] at (0:3.94) {};

\begin{scope}[shift=(0:3.94)]
\begin{scope}[shift=(80:1.2),rotate=80, scale=0.7]
\draw[very thick] (0,0) -- (1.2,0);
\draw[very thick] (0,0) -- (40:1.2);
\draw[very thick] (0,0) -- (-40:1.2);
\foreach \i in {-40,0,...,40}{
\begin{scope}[shift=(\i:1.2),rotate=\i]
\draw[thick] (0,0) -- (1,0);
\draw[thick] (0,0) -- (30:1);
\draw[thick] (0,0) -- (-30:1);
\foreach \j in {-30,0,...,30}{
\begin{scope}[shift=(\j:1),rotate=\j]
\draw[dashed] (0,0) -- (0.4,0);
\draw[dashed] (0,0) -- (20:0.4);
\draw[dashed] (0,0) -- (-20:0.4);
\node[draw,circle,inner sep=2pt,fill=white, thick] at (0,0) {};
\end{scope}
\node[draw,circle,inner sep=2pt,fill=black, thick] at (0,0) {};
}
\end{scope}
}
\node[draw,circle,inner sep=2pt,fill=black, thick] at (0,0) {};
\end{scope}
\begin{scope}[shift=(-80:1.2),rotate=-80, scale=0.7]
\draw[very thick] (0,0) -- (1.2,0);
\draw[very thick] (0,0) -- (40:1.2);
\draw[very thick] (0,0) -- (-40:1.2);
\foreach \i in {-40,0,...,40}{
\begin{scope}[shift=(\i:1.2),rotate=\i]
\draw[thick] (0,0) -- (1,0);
\draw[thick] (0,0) -- (30:1);
\draw[thick] (0,0) -- (-30:1);
\foreach \j in {-30,0,...,30}{
\begin{scope}[shift=(\j:1),rotate=\j]
\draw[dashed] (0,0) -- (0.4,0);
\draw[dashed] (0,0) -- (20:0.4);
\draw[dashed] (0,0) -- (-20:0.4);
\node[draw,circle,inner sep=2pt,fill=white, thick] at (0,0) {};
\end{scope}
\node[draw,circle,inner sep=2pt,fill=black, thick] at (0,0) {};
}
\end{scope}
}
\node[draw,circle,inner sep=2pt,fill=black, thick] at (0,0) {};
\end{scope}
\end{scope}

\draw[very thick] (5.9,0.35) -- ++(50:1.5);
\draw[very thick] (5.9,0.35) -- ++(-50:1.5);

\node[draw,circle,inner sep=2pt,fill=black, thick] at (5.9,0.35) {};

\begin{scope}[shift=(3.4:5.91)]
\begin{scope}[shift=(50:1.2),rotate=50]
\draw[very thick] (0,0) -- (1.2,0);
\draw[very thick] (0,0) -- (40:1.2);
\draw[very thick] (0,0) -- (-40:1.2);
\foreach \i in {-40,0,...,40}{
\begin{scope}[shift=(\i:1.2),rotate=\i]
\draw[thick] (0,0) -- (1,0);
\draw[thick] (0,0) -- (30:1);
\draw[thick] (0,0) -- (-30:1);
\foreach \j in {-30,0,...,30}{
\begin{scope}[shift=(\j:1),rotate=\j]
\draw[dashed] (0,0) -- (0.4,0);
\draw[dashed] (0,0) -- (20:0.4);
\draw[dashed] (0,0) -- (-20:0.4);
\node[draw,circle,inner sep=2pt,fill=white, thick] at (0,0) {};
\end{scope}
\node[draw,circle,inner sep=2pt,fill=black, thick] at (0,0) {};
}
\end{scope}
}
\node[draw,circle,inner sep=2pt,fill=black, thick] at (0,0) {};
\end{scope}
\begin{scope}[shift=(-50:1.2),rotate=-50]
\draw[very thick] (0,0) -- (1.2,0);
\draw[very thick] (0,0) -- (40:1.2);
\draw[very thick] (0,0) -- (-40:1.2);
\foreach \i in {-40,0,...,40}{
\begin{scope}[shift=(\i:1.2),rotate=\i]
\draw[thick] (0,0) -- (1,0);
\draw[thick] (0,0) -- (30:1);
\draw[thick] (0,0) -- (-30:1);
\foreach \j in {-30,0,...,30}{
\begin{scope}[shift=(\j:1),rotate=\j]
\draw[dashed] (0,0) -- (0.4,0);
\draw[dashed] (0,0) -- (20:0.4);
\draw[dashed] (0,0) -- (-20:0.4);
\node[draw,circle,inner sep=2pt,fill=white, thick] at (0,0) {};
\end{scope}
\node[draw,circle,inner sep=2pt,fill=black, thick] at (0,0) {};
}
\end{scope}
}
\node[draw,circle,inner sep=2pt,fill=black, thick] at (0,0) {};
\end{scope}
\end{scope}

\end{tikzpicture}
\caption{Tube of radius 2 around an apartment}
\end{figure}
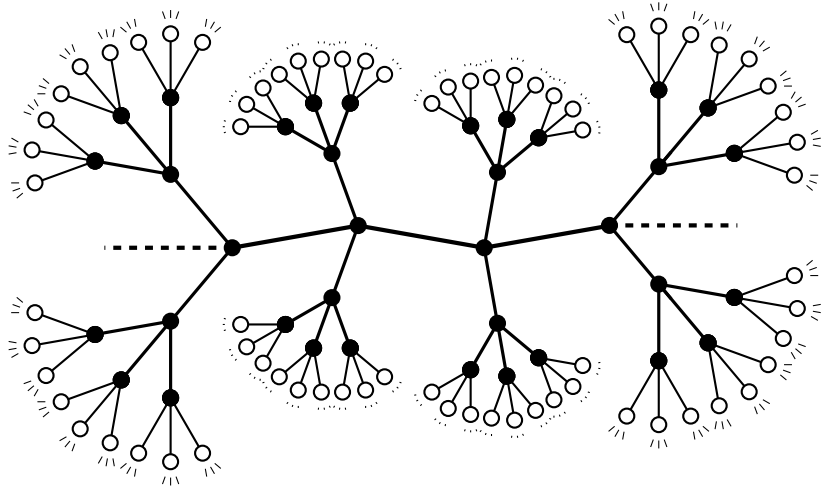

\begin{theorem}\label{thm:balltube}
    Assume that $\alpha\in \mathbb{Z}_{\geq 0}$ and $\beta$ is an integer if $v$ is a vertex, and a half-integer if $\beta$ is a midpoint of an edge. Then the number of vertices in $T_{\alpha}(A)\cap B_\beta(v)$ is 
    \[\left\lbrace \begin{array}{ll}
         0& \text{if }\alpha+\beta<\delta  \\
         | B_{\alpha}(a_0)\cap B_\beta(v)|& \text{if }\beta-\alpha\leq \delta\leq \alpha+\beta \\
         {[}1+2(\beta-\alpha-\delta)]q^\alpha + 2\frac{q^\alpha-1}{q-1}&\text{if }\delta\leq\beta-\alpha
    \end{array}\right. .\]
\end{theorem}

We will consider several cases.  As before, if $\delta>\alpha+\beta$ then $T_{\alpha}(A)\cap B_\beta(v) = \varnothing$.

\begin{lemma}\label{lem:balltubevert}
 If $\beta-\alpha\leq \delta\leq \alpha+\beta$ then 
 \[ T_{\alpha}(A)\cap B_\beta(v)=  B_{\alpha}(a_0)\cap B_\beta(v).\]
\end{lemma}
\begin{proof} 
Let $x\in T_{\alpha}(A)\cap B_\beta(v)$. 
    Let $c$ be vertex of $[v,a_0]$ closest to $x$.
    \begin{itemize}
    \item If $c \neq a_0$ then the vector the closest to $x$ on $A$ is $a_0$ and therefore $x\in B_\alpha(a_0)$. 
    \item If $c = a_0$ then 
    \[\mathrm{d}(x,v)=\mathrm{d}(x,a_0)+\mathrm{d}(a_0,v) = \mathrm{d}(x,a_0)+\delta\leq \beta,\]
    hence 
    \[\mathrm{d}(x,a_0)\leq \beta-\delta \leq \alpha\]
    hence $x\in B_\alpha(a_0)$.
\end{itemize}
\end{proof}

We can therefore compute this case based on the previous section.

\begin{lemma}[Ball centered at a vertex]
    If $v$ is a vertex, $\alpha,\beta\in \mathbb{Z}_{\geq 0}$ , and $\delta<\beta-\alpha$ then there are a total of 
    \[[1+2(\beta-\alpha-\delta)]q^\alpha + 2\frac{q^\alpha-1}{q-1}\]
    vertices in $ B_{\beta}(v)\cap T_\alpha(A)$.
\end{lemma}
\begin{proof}
    First, since $\beta > \delta+\alpha$, we get that $B_{\alpha}(a_0)\subset B_{\beta}(v)\cap T_\alpha(A)$. Then we look at $a_i$ and determine the maximal radius $r_i$ such that $B_{\alpha}(a_i)\subset B_{\beta}(v)\cap T_\alpha(A)$. We have 
    \[\mathrm{d}(a_i,v) = |i|+\delta,\]
    so we have $r_i = \min(\alpha,\beta-\mathrm{d}(a_i,v))= \min(\alpha,\beta-|i|-\delta)$ as shown in the following figure.

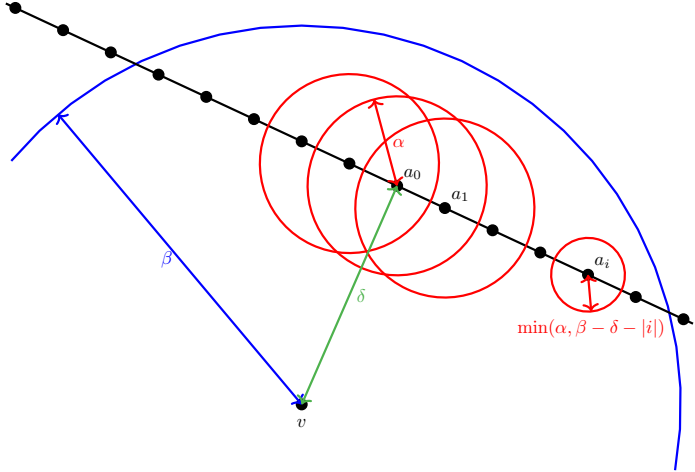
\begin{figure}[H]
\centering
\scalebox{.7}{
\begin{tikzpicture}
\node[draw,circle,inner sep=2pt,fill=black, thick] at (0,0) {};
\node[below=0.15] at (0,0) {$v$};
\draw[very thick, color=blue, <->] (0,0) -- ++(130:7.2) node[midway, left] {$\beta$};

\draw [blue,very thick,domain=-10:140] plot ({7.2*cos(\x)}, {7.2*sin(\x)});
\begin{scope}[shift=(90:5),rotate=-25]
\draw[very thick] (-6.2,0) -- (8.2,0);
\foreach \i in{-6,-5,...,8}{
\node[draw,circle,inner sep=2pt,fill=black, thick] at (\i,0) {};
}
\draw[very thick, color=red, <->] (2,0) -- ++(130:1.7) node[midway, right] {$\alpha$};
\node[above right] at (2,0) {$a_0$};
\node[above right] at (3,0) {$a_1$};  
\foreach \i in {1,2,...,3}{
 \draw[fill=none, very thick, color=red](\i,0) circle (1.7);
 }
  \draw[fill=none, very thick, color=red](6,0) circle (0.7);
  \draw[very thick, color=red, <->] (6,0) -- ++(-60:0.7) node[below] {$\min(\alpha, \beta-\delta-|i|)$};  
  \draw[very thick, color=green!40!gray, <->] (2,0) -- (-65:5) node[midway, right] {$\delta$};
  \node[above right] at (6,0) {$a_{i}$};
\end{scope}
\end{tikzpicture}}
\end{figure}
We have $\beta-|i|-\delta<\alpha$ when $|i|>\beta-\alpha-\delta$. 
Below we list the vertices of $A$ contained in $B_\beta(v)$ and represent the corresponding $r_i$ above. 

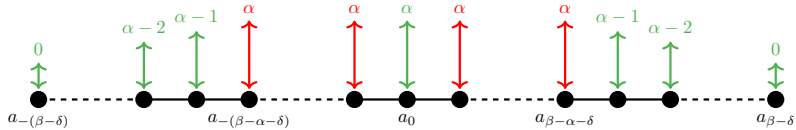
\begin{figure}[H]
\centering
\scalebox{.7}{
\begin{tikzpicture}
\draw[very thick] (-5,0) -- (-3,0);
\draw[very thick] (5,0) -- (3,0);
\draw[very thick] (-1,0) -- (1,0);
\draw[very thick, dashed] (-7,0) -- (-5,0);
\draw[very thick, dashed] (-3,0) -- (-1,0);
\draw[very thick, dashed] (7,0) -- (5,0);
\draw[very thick, dashed] (3,0) -- (1,0);
\foreach \i in {-7,-5,-4,-3,-1,0,1,3,4,5,7}{
\node[draw,circle,inner sep=3pt,fill=black, thick] at (\i,0) {};
}
\draw[very thick, color=green!40!gray, <->] (0,0.2) -- (0,1.5)  node [above] {$\alpha$};
\draw[very thick, color=red, <->] (1,0.2) -- (1,1.5)node [above] {$\alpha$};
\draw[very thick, color=red, <->] (-1,0.2) -- (-1,1.5)node [above] {$\alpha$};
\draw[very thick, color=red, <->] (-3,0.2) -- (-3,1.5)node [above] {$\alpha$};
\draw[very thick, color=red, <->] (3,0.2) -- (3,1.5)node [above] {$\alpha$};
\draw[very thick, color=green!40!gray, <->] (4,0.2) -- (4,1.3)node [above] {$\alpha-1$};
\draw[very thick, color=green!40!gray, <->] (5,0.2) -- (5,1.1)node [above] {$\alpha-2$};
\draw[very thick, color=green!40!gray, <->] (7,0.2) -- (7,0.7)node [above] {$0$};
\draw[very thick, color=green!40!gray, <->] (-4,0.2) -- (-4,1.3)node [above] {$\alpha-1$};
\draw[very thick, color=green!40!gray, <->] (-5,0.2) -- (-5,1.1)node [above] {$\alpha-2$};
\draw[very thick, color=green!40!gray, <->] (-7,0.2) -- (-7,0.7)node [above] {$0$};
\node[below=0.15] at (0,0) {$a_0$}; 
\node[below=0.15] at (3,0) {$a_{\beta-\alpha-\delta}$}; 
\node[below=0.15] at (7,0) {$a_{\beta-\delta}$}; 
\node[below=0.15] at (-3,0) {$a_{-(\beta-\alpha-\delta)}$};
\node[below=0.15] at (-7,0) {$a_{-(\beta-\delta)}$}; 
\end{tikzpicture}}
\end{figure}

Notice that we can group the green vertices to make a single ball of radius $\alpha$, and we are left with the red part. For each of the $2(\beta-\alpha-\delta)$ vertices with a red arrow, we include all vertices at distance at most $\alpha$ going out of the $q-1$ neighbors going out of $A$. We get a total of  
\[1+(q+1)\frac{q^{\alpha}-1}{q-1} + 2 (\beta-\alpha-\delta)q^\alpha,\]
where we used that for each red vertex we include 
\[1 + (q-1)[1+q+\cdots +q^{\alpha-1}] =  1+ (q-1)\frac{q^\alpha-1}{q-1} = q^\alpha \]
vertices.
We could also write $1+(q+1)\frac{q^{\alpha}-1}{q-1} = q^\alpha + 2 \frac{q^\alpha-1}{q-1}$ which gives us 
\[[1+2(\beta-\alpha-\delta)]q^\alpha + 2\frac{q^\alpha-1}{q-1}\]
total vertices.
\end{proof}

For the purpose computing orbital integrals, we will need the same result when $v$ is a midpoint of an edge, and $\beta$ is a half-integer.

\begin{lemma}[Ball centered at a midpoint]
Assume $v$ is a midpoint of an edge and $\beta$ is a half-integer. We still assume that  $\alpha$ is an integer. If $0<\delta<\beta-\alpha$ then there are a total of 
 \[[1+2(\beta-\alpha-\delta)]q^\alpha + 2\frac{q^\alpha-1}{q-1}\]
vertices in $B_{\beta}(v)\cap T_{\alpha}(A)$, as in Lemma \ref{lem:balltubevert}.
\end{lemma}
\begin{proof}
    Let $v',v''$ be the two neighbors of $v$ so that $v'$ is the closest one to $A$. 

    First note that 
    \[B_{\beta}(v) = B_{\beta-1/2}(v')\cap B_{\beta-1/2}(v'').\]
    If we can show that $B_{\beta-1/2}(v'')\cap T_\alpha(A)\subset B_{\beta-1/2}(v')$ then we can conclude that
    \[B_{\beta}(v) \cap T_\alpha(A) = B_{\beta-1/2}(v')\cap B_{\beta-1/2}(v'')\cap T_\alpha(A)=B_{\beta-1/2}(v')\cap T_\alpha(A),\] and we can apply the previous Lemma. Now we prove the claim. 

    Let $x\in B_{\beta-1/2}(v'')\cap T_\alpha(A)$. If $x$ is closer to $v'$ than $v''$ then by default we have 
    \[\mathrm{d}(x,v') = \mathrm{d}(x,v'')-1\leq \beta-3/2.\]
    On the other hand, if $x$ is closer to $v''$, then 
    \[d(x,A) = \mathrm{d}(x,v')+\underset{=\delta-1/2}{\underbrace{\mathrm{d}(v',A)}}\leq \alpha,\]
    hence 
    \[\mathrm{d}(x,v')\leq \alpha-\delta+1/2<\beta-2\delta+1/2.\]
    And since $\delta>0$ then $2\delta\geq 1$ and we can conclude $\mathrm{d}(x,v')\leq \beta-1/2$ as desired.\\

    We can use the formula of Lemma \ref{lem:balltubevert} replacing $\beta$ by $\beta-1/2$ and $\delta$ by $\delta-1/2$, the $1/2$s cancel out.
\end{proof}

We are only left to consider the case $\delta = 0$ and $v$ is a midpoint of an edge of $A$.

\begin{lemma}
    Let $v\in A$ be the midpoint of an edge and let $\beta\geq 0$ be a half-integer.Let $\alpha\geq 0$ be an integer. The number of vertices in $B_\beta(v)\cap T_\alpha(A)$ is equal to 
   \[ [1+2(\beta-\alpha))q^\alpha + 2\frac{q^\alpha-1}{q-1}.\]
\end{lemma}
\begin{proof}
    There is no longer a unique  vertex of $A$ closest to $v$. Instead we label the vertices of $A$ as $a_i, a_i'$ where $i\geq 0$. Let us use the same diagram and $r_i$'s from the proof of Lemma \ref{lem:balltubevert}.

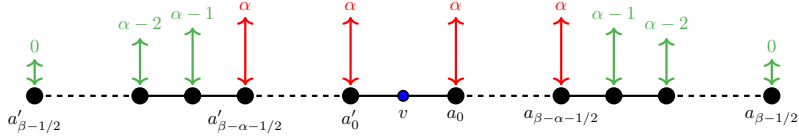
    \begin{figure}[H]
\centering
\scalebox{.7}{
\begin{tikzpicture}
\draw[very thick] (-5,0) -- (-3,0);
\draw[very thick] (5,0) -- (3,0);
\draw[very thick] (-1,0) -- (1,0);
\draw[very thick, dashed] (-7,0) -- (-5,0);
\draw[very thick, dashed] (-3,0) -- (-1,0);
\draw[very thick, dashed] (7,0) -- (5,0);
\draw[very thick, dashed] (3,0) -- (1,0);
\foreach \i in {-7,-5,-4,-3,-1,1,3,4,5,7}{
\node[draw,circle,inner sep=3pt,fill=black, thick] at (\i,0) {};
}
\node[draw,circle,inner sep=2pt,fill=blue, thick] at (0,0) {};
\draw[very thick, color=red, <->] (1,0.2) -- (1,1.5)node [above] {$\alpha$};
\draw[very thick, color=red, <->] (-1,0.2) -- (-1,1.5)node [above] {$\alpha$};
\draw[very thick, color=red, <->] (-3,0.2) -- (-3,1.5)node [above] {$\alpha$};
\draw[very thick, color=red, <->] (3,0.2) -- (3,1.5)node [above] {$\alpha$};
\draw[very thick, color=green!40!gray, <->] (4,0.2) -- (4,1.3)node [above] {$\alpha-1$};
\draw[very thick, color=green!40!gray, <->] (5,0.2) -- (5,1.1)node [above] {$\alpha-2$};
\draw[very thick, color=green!40!gray, <->] (7,0.2) -- (7,0.7)node [above] {$0$};
\draw[very thick, color=green!40!gray, <->] (-4,0.2) -- (-4,1.3)node [above] {$\alpha-1$};
\draw[very thick, color=green!40!gray, <->] (-5,0.2) -- (-5,1.1)node [above] {$\alpha-2$};
\draw[very thick, color=green!40!gray, <->] (-7,0.2) -- (-7,0.7)node [above] {$0$};
\node[below=0.15] at (0,0) {$v$}; 
\node[below=0.15] at (1,0) {$a_0$};
\node[below=0.15] at (-1,0) {$a_0'$}; 
\node[below=0.15] at (3,0) {$a_{\beta-\alpha-1/2}$}; 
\node[below=0.15] at (7,0) {$a_{\beta-1/2}$}; 
\node[below=0.15] at (-3,0) {$a_{\beta-\alpha-1/2}'$};
\node[below=0.15] at (-7,0) {$a_{\beta-1/2}'$}; 
\end{tikzpicture}}
\end{figure}
We have $r_i = \min(\alpha, \beta-(1/2+i))$. Note that the if we added $q^{\alpha}$ vertices connected to $v$ going out of $A$ to the green vertices, we would get exactly a ball of radius $\alpha$. Therefore the number of green vertices is exactly 
\[1+(1+q)\frac{q^{\alpha}-1}{q-1}-q^{\alpha} = 2\frac{q^{\alpha}-1}{q-1}.\] Adding the $2(\beta-\alpha-1/2+1)$ red parts, we get the total amount of vertices:
\[[1+2(\beta-\alpha)]q^\alpha + 2\frac{q^\alpha-1}{q-1},\]
which again matches the previous formulas with $\delta=0$.

\end{proof}

\subsection{Intersection of two tubes.}\label{sec:tubtub}

This last case is simpler than the previous section, mostly because we will not have to deal with a vertex being a midpoint of an edge. The distance between two apartments will always be minimized at vertices. However, we need to distinguish cases when the apartments intersect or not. 

This section will prove the following theorem. 

\begin{theorem}
    Let $A,B$ be two distinct infinite geodesics at distance $\delta \geq 0$ from each other in a $q+1$-regular tree. Let $\alpha,\beta\in \mathbb{Z}_{\geq 0}$ and let  $r\geq 0$ be the number of vertices in $A\cap B$. The number of vertices in $T_\alpha(A)\cap T_\beta(B)$ is equal to 
    \[\left\lbrace  \begin{array}{ll}
      0  &\text{if }\alpha+\beta<\delta  \\ &\\
     1+(q+1)\frac{q^{(\alpha+\beta-\delta)/2}-1}{q-1}    & \text{if }\begin{array}{l}|\alpha-\beta|\leq \delta\leq \alpha+\beta\\ \alpha+\beta+\delta\equiv0\pmod2\end{array}\\&\\
      2\frac{q^{(\alpha+\beta-\delta+1)/2}-1}{q-1}            & \text{if }\begin{array}{l}|\alpha-\beta|\leq \delta\leq \alpha+\beta\\ \alpha+\beta+\delta\equiv1\pmod2\end{array}\\&\\
       {[}1+2(|\alpha-\beta|-\delta)]q^{\min(\alpha,\beta)} + 2\frac{q^{\min(\alpha,\beta)}-1}{q-1}  &\text{if }0<\delta\leq |\alpha-\beta|\\ &\\
     {[}r+2|\alpha-\beta|]q^{\min(\alpha,\beta)}+2\frac{q^{\min(\alpha,\beta)}-1}{q-1}  &  \delta = 0    \end{array} \right.\]
\end{theorem}

\begin{lemma}
Let $A,B$ be two infinite geodesics and $\alpha,\beta \geq 0$. If $A\cap B$ contains infinitely many vertices then $T_\alpha(A)\cap T_\beta(B)$ contains infinitely many vertices.  If $A\cap B$ contains finitely-many vertices, then $A\cap B$ is always a closed segment $[a,b]$ where $a,b$ are vertices. 
\end{lemma}
\begin{proof}
    This just follows from the connectedness of the tree and absence of cycles. If $x,y\in A\cap B$ then $[x,y]\subset A\cap B$.
\end{proof}

\begin{lemma} Let $A,B$ be two infinite geodesics at distance $\delta >0$ from each other in a $q+1$-regular tree.  Let $\alpha,\beta\geq 0$. Let $a\in A$ and $b\in B$ so that $\rd(a,b) = \delta$.
set of vertices at distance at most $\alpha$ from $A$ and $\beta$ from $B$. If $|\alpha-\beta|\leq \delta$ then \[T_\alpha(a)\cap T_\beta(b) = B_{\alpha}(a)\cap B_\beta(b).\]
\end{lemma}
\begin{proof}
One inclusion is immediate. Let $X = T_\alpha(a)\cap T_\beta(b) $. We just need to show $X \subset B_{\alpha}(x)\cap B_\beta(y)$. 

Let $v\in X$. Let $w$ denote the closest vertex to $v$ on $[a,b]$.

If $w = a$ (resp. $w = b$) then the closest vector to $v$ on $B$ (resp. $A$) is $b$ (resp. $a$) hence $v\in B_{\beta}(b)$  (resp. $v\in B_\alpha(a)$) so we can use Lemma \ref{lem:intergeoball} to conclude that $v\in B_\alpha(a)\cap B_\beta(b)$.

If $w$ is in the interior of $[a,b]$ then $\rd(v,a) = \rd(v,A)$ and $\rd(v,b) = \rd(v,B)$ so we get $v\in B_{\alpha}(a)\cap B_\beta(b)$.
\end{proof}

\begin{lemma} Let $A,B$ be two infinite geodesics at distance $\delta >0$ from each other in a $q+1$-regular tree.  Let $\alpha,\beta\geq 0$. Let $a\in A$ and $b\in B$ so that $\rd(a,b) = \delta$.
set of vertices at distance at most $\alpha$ from $A$ and $\beta$ from $B$. If $|\alpha-\beta|\geq  \delta$ then \[T_\alpha(a)\cap T_\beta(b) =\left\lbrace \begin{array}{ll}T_{\alpha}(a)\cap B_\beta(b)&\text{if }\alpha\leq \beta\\B_{\alpha}(a)\cap T_\beta(B)&\text{if }\alpha\geq \beta \end{array}\right. ,\]
which we can deal with using Theorem \ref{thm:balltube}.
\end{lemma}

Putting together the previous two Lemmas, we get the following. 
\begin{prop}
    Let $A,B$ be two distinct infinite geodesics at distance $\delta >0$ from each other in a $q+1$-regular tree. Let $\alpha,\beta\in \mathbb{Z}_{\geq 0}$. The number of vertices in $T_\alpha(A)\cap T_\beta(B)$ is equal to 
    \[\left\lbrace  \begin{array}{ll}
      0  &\text{if }\alpha+\beta<\delta  \\ &\\
     1+(q+1)\frac{q^{(\alpha+\beta-\delta)/2}-1}{q-1}    & \text{if }\begin{array}{l}|\alpha-\beta|\leq \delta\leq \alpha+\beta\\ \alpha+\beta+\delta\equiv0\pmod2\end{array}\\&\\
      2\frac{q^{(\alpha+\beta-\delta+1)/2}-1}{q-1}            & \text{if }\begin{array}{l}|\alpha-\beta|\leq \delta\leq \alpha+\beta\\ \alpha+\beta+\delta\equiv1\pmod2\end{array}\\&\\
       {[}1+2(|\alpha-\beta|-\delta)]q^{\min(\alpha,\beta)} + 2\frac{q^{\min(\alpha,\beta)}-1}{q-1}  &\text{if }0<\delta\leq |\alpha-\beta|
    \end{array} \right.\]
\end{prop}
We are left with the case of geodesics intersecting at finitely many vertices. 

\begin{prop}
    Let $A,B$ be two distinct infinite geodesics in a $q+1$-regular tree such that $A\cap B$ contains $r>0$ many vertices. Let $\alpha,\beta\in \mathbb{Z}_{\geq 0}$. The number of vertices in $T_\alpha(A)\cap T_\beta(b)$ is equal to 
    \[[r+2|\alpha-\beta|]q^{\min(\alpha,\beta)}+2\frac{q^{\min(\alpha,\beta)}-1}{q-1}.\]
\end{prop}
\begin{proof}
Assume that $\alpha\geq \beta$ and let $X$ be the set of vertices contained in $T_\alpha(A)\cap T_\beta(B)$.\\

If $r=1$ then the exact same argument as the last case of the previous section holds, since the entire ball of radius $\beta$ around the unique vertex of $A\cap B$ is contained in the ball of radius $\alpha$, we can use the same formula with $\delta=0$ and get 
\[|X| = [1+2(\alpha-\beta)]q^\beta+2\frac{q^{\beta}-1}{q-1} .\]

When $r\geq 2$, we can split the tree into 4 parts. 
\begin{itemize}
\item In {\bf\color{red} red} and {\bf\color{blue} blue} the apartments $A$ and $B$ respectively, with their intersection in {\bf black}.
\item In {\bf\color{brown} brown}: vertices  $P_1\subset X$ whose closest vertex on $A\cap B$ belongs to a vertex inside (not an endpoint) of $A\cap B$. 
\item In {\bf\color{green!50!olive!50} green}: vertices in $P_2\subset X$ whose closest vertex on $A\cap B$ is an endpoint of $A\cap B$.
\item In {\bf\color{cyan} cyan}: vertices in $P_3\subset X$ closer to $B$ than $A$.
\item In {\bf\color{orange} orange}: vertices in $P_4\subset X$ closer to $A$ than $B$.
\end{itemize}

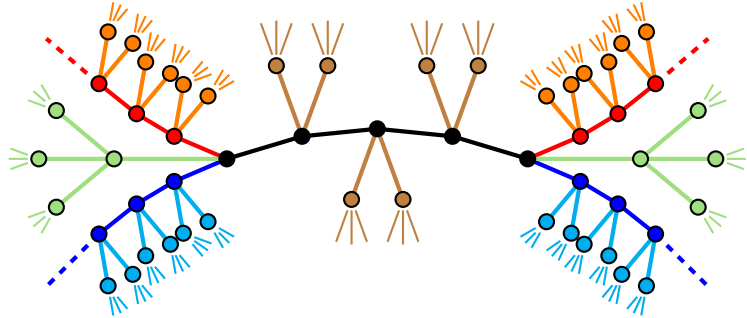
\begin{center}
\[\begin{tikzpicture}
\draw[ultra thick] (-2,-0.3) -- (-1,0) -- (0,0.1) -- (1,0) -- (2,-.3);
\draw[ultra thick, color=blue] (2,-.3) -- (2.7,-.6) -- (3.2,-.9) -- (3.7,-1.3);
\draw[ultra thick, dashed, color=blue] (3.7,-1.3) -- (4.4,-2);
\draw[ultra thick, color=blue] (-2,-.3) -- (-2.7,-.6) -- (-3.2,-.9) -- (-3.7,-1.3);
\draw[ultra thick, dashed, color=blue] (-3.7,-1.3) -- (-4.4,-2);
\draw[ultra thick, color=red] (-2,-.3) -- (-2.7,0) -- (-3.2,.3) -- (-3.7,.7);
\draw[ultra thick, dashed, color=red] (-3.7,.7) -- (-4.4,1.3);
\draw[ultra thick, color=red] (2,-.3) -- (2.7,0) -- (3.2,.3) -- (3.7,.7);
\draw[ultra thick, dashed, color=red] (3.7,.7) -- (4.4,1.3);
\draw[ultra thick, color=green!50!olive!50] (-2,-.3)--(-3.5,-.3);
\draw[ultra thick, color=green!50!olive!50] (2,-.3)--(3.5,-.3);
\draw[ultra thick, color=green!50!olive!50] (3.5,-.3) -- +(40:1) node (G1) {};
\draw[ultra thick, color=green!50!olive!50] (3.5,-.3) -- +(0:1) node (G2) {};
\draw[ultra thick, color=green!50!olive!50] (3.5,-.3) -- +(-40:1) node (G3) {};
\draw[ultra thick, color=green!50!olive!50] (-3.5,-.3) -- +(140:1) node (G4) {};
\draw[ultra thick, color=green!50!olive!50] (-3.5,-.3) -- +(180:1) node (G5) {};
\draw[ultra thick, color=green!50!olive!50] (-3.5,-.3) -- +(-140:1) node (G6) {};
\draw[thick, color=green!50!olive!50] (G1) -- +(60:.4);
\draw[thick, color=green!50!olive!50] (G1) -- +(40:.4);
\draw[thick, color=green!50!olive!50] (G1) -- +(20:.4); 
\draw[thick, color=green!50!olive!50] (G2) -- +(20:.4);
\draw[thick, color=green!50!olive!50] (G2) -- +(00:.4);
\draw[thick, color=green!50!olive!50] (G2) -- +(-20:.4); 
\draw[thick, color=green!50!olive!50] (G3) -- +(-60:.4);
\draw[thick, color=green!50!olive!50] (G3) -- +(-40:.4);
\draw[thick, color=green!50!olive!50] (G3) -- +(-20:.4); 
\draw[thick, color=green!50!olive!50] (G4) -- +(120:.4);
\draw[thick, color=green!50!olive!50] (G4) -- +(140:.4);
\draw[thick, color=green!50!olive!50] (G4) -- +(160:.4); 
\draw[thick, color=green!50!olive!50] (G5) -- +(160:.4);
\draw[thick, color=green!50!olive!50] (G5) -- +(180:.4);
\draw[thick, color=green!50!olive!50] (G5) -- +(200:.4); 
\draw[thick, color=green!50!olive!50] (G6) -- +(-120:.4);
\draw[thick, color=green!50!olive!50] (G6) -- +(-140:.4);
\draw[thick, color=green!50!olive!50] (G6) -- +(-160:.4);
\draw[ultra thick, color=cyan] (2.7,-.6) -- +(-100:.7) node (B12) {};
\draw[ultra thick, color=cyan] (2.7,-.6)-- +(-130:.7) node (B11) {};
\draw[ultra thick, color=cyan] (3.2,-.9) -- +(-100:.7) node (B10) {};
\draw[ultra thick, color=cyan] (3.2,-.9) -- +(-130:.7) node (B9) {};
\draw[ultra thick, color=cyan] (3.7,-1.3) -- +(-100:.7) node (B8) {};
\draw[ultra thick, color=cyan] (3.7,-1.3) -- +(-130:.7) node (B7) {};
\draw[ultra thick, color=cyan] (-2.7,-.6) -- +(-80:.7) node (B6) {};
\draw[ultra thick, color=cyan] (-2.7,-.6)-- +(-50:.7) node (B5) {};
\draw[ultra thick, color=cyan] (-3.2,-.9) -- +(-80:.7) node (B4) {};
\draw[ultra thick, color=cyan] (-3.2,-.9) -- +(-50:.7) node (B3) {};
\draw[ultra thick, color=cyan] (-3.7,-1.3) -- +(-80:.7) node (B2) {};
\draw[ultra thick, color=cyan] (-3.7,-1.3) -- +(-50:.7) node (B1) {};
\draw[ thick, color=cyan] (B1) -- +(-80:0.4);
\draw[ thick, color=cyan] (B1) -- +(-65:0.4);
\draw[thick, color=cyan] (B1) -- +(-50:0.4);
\draw[ thick, color=cyan] (B2) -- +(-80:0.4);
\draw[ thick, color=cyan] (B2) -- +(-65:0.4);
\draw[thick, color=cyan] (B2) -- +(-50:0.4);
\draw[ thick, color=cyan] (B3) -- +(-80:0.4);
\draw[ thick, color=cyan] (B3) -- +(-65:0.4);
\draw[thick, color=cyan] (B3) -- +(-50:0.4);
\draw[ thick, color=cyan] (B4) -- +(-80:0.4);
\draw[ thick, color=cyan] (B4) -- +(-65:0.4);
\draw[thick, color=cyan] (B4) -- +(-50:0.4);
\draw[ thick, color=cyan] (B5) -- +(-60:0.4);
\draw[ thick, color=cyan] (B5) -- +(-45:0.4);
\draw[thick, color=cyan] (B5) -- +(-30:0.4);
\draw[ thick, color=cyan] (B6) -- +(-50:0.4);
\draw[ thick, color=cyan] (B6) -- +(-35:0.4);
\draw[thick, color=cyan] (B6) -- +(-20:0.4);
\draw[ thick, color=cyan] (B7) -- +(-100:0.4);
\draw[ thick, color=cyan] (B7) -- +(-115:0.4);
\draw[thick, color=cyan] (B7) -- +(-130:0.4);
\draw[ thick, color=cyan] (B8) -- +(-100:0.4);
\draw[ thick, color=cyan] (B8) -- +(-115:0.4);
\draw[thick, color=cyan] (B8) -- +(-130:0.4);
\draw[ thick, color=cyan] (B9) -- +(-100:0.4);
\draw[ thick, color=cyan] (B9) -- +(-115:0.4);
\draw[thick, color=cyan] (B9) -- +(-130:0.4);
\draw[ thick, color=cyan] (B10) -- +(-100:0.4);
\draw[ thick, color=cyan] (B10) -- +(-115:0.4);
\draw[thick, color=cyan] (B10) -- +(-130:0.4);
\draw[ thick, color=cyan] (B11) -- +(-120:0.4);
\draw[ thick, color=cyan] (B11) -- +(-135:0.4);
\draw[thick, color=cyan] (B11) -- +(-150:0.4);
\draw[ thick, color=cyan] (B12) -- +(-115:0.4);
\draw[ thick, color=cyan] (B12) -- +(-130:0.4);
\draw[thick, color=cyan] (B12) -- +(-145:0.4);
\draw[ultra thick, color=orange] (2.7,0) -- +(100:.7) node (O6) {};
\draw[ultra thick, color=orange] (2.7,0) -- +(130:.7) node (O5) {};
\draw[ultra thick, color=orange] (3.2,.3) -- +(100:.7) node (O4) {};
\draw[ultra thick, color=orange] (3.2,.3) -- +(130:.7) node (O3) {};
\draw[ultra thick, color=orange] (3.7,.7) -- +(100:.7) node (O2) {};
\draw[ultra thick, color=orange] (3.7,.7) -- +(130:.7) node (O1) {};
\draw[ thick, color=orange] (O1) -- +(100:0.4);
\draw[ thick, color=orange] (O1) -- +(115:0.4);
\draw[thick, color=orange] (O1) -- +(130:0.4);
\draw[ thick, color=orange] (O2) -- +(100:0.4);
\draw[ thick, color=orange] (O2) -- +(115:0.4);
\draw[thick, color=orange] (O2) -- +(130:0.4);
\draw[ thick, color=orange] (O3) -- +(100:0.4);
\draw[ thick, color=orange] (O3) -- +(115:0.4);
\draw[thick, color=orange] (O3) -- +(130:0.4);
\draw[ thick, color=orange] (O4) -- +(100:0.4);
\draw[ thick, color=orange] (O4) -- +(115:0.4);
\draw[thick, color=orange] (O4) -- +(130:0.4);
\draw[ thick, color=orange] (O5) -- +(120:0.4);
\draw[ thick, color=orange] (O5) -- +(135:0.4);
\draw[thick, color=orange] (O5) -- +(150:0.4);
\draw[ thick, color=orange] (O6) -- +(115:0.4);
\draw[ thick, color=orange] (O6) -- +(130:0.4);
\draw[thick, color=orange] (O6) -- +(145:0.4);
\draw[ultra thick, color=orange] (-2.7,0) -- +(80:.7) node (O12) {};
\draw[ultra thick, color=orange] (-2.7,0) -- +(50:.7) node (O11) {};
\draw[ultra thick, color=orange] (-3.2,.3) -- +(80:.7) node (O10) {};
\draw[ultra thick, color=orange] (-3.2,.3) -- +(50:.7) node (O9) {};
\draw[ultra thick, color=orange] (-3.7,.7) -- +(80:.7) node (O8) {};
\draw[ultra thick, color=orange] (-3.7,.7) -- +(50:.7) node (O7) {};
\draw[ thick, color=orange] (O7) -- +(80:0.4);
\draw[ thick, color=orange] (O7) -- +(65:0.4);
\draw[thick, color=orange] (O7) -- +(50:0.4);
\draw[ thick, color=orange] (O8) -- +(80:0.4);
\draw[ thick, color=orange] (O8) -- +(65:0.4);
\draw[thick, color=orange] (O8) -- +(50:0.4);
\draw[ thick, color=orange] (O9) -- +(80:0.4);
\draw[ thick, color=orange] (O9) -- +(65:0.4);
\draw[thick, color=orange] (O9) -- +(50:0.4);
\draw[ thick, color=orange] (O10) -- +(80:0.4);
\draw[ thick, color=orange] (O10) -- +(65:0.4);
\draw[thick, color=orange] (O10) -- +(50:0.4);
\draw[ thick, color=orange] (O11) -- +(60:0.4);
\draw[ thick, color=orange] (O11) -- +(45:0.4);
\draw[thick, color=orange] (O11) -- +(30:0.4);
\draw[ thick, color=orange] (O12) -- +(50:0.4);
\draw[ thick, color=orange] (O12) -- +(35:0.4);
\draw[thick, color=orange] (O12) -- +(20:0.4);
\draw[ultra thick, color=brown] (0,.1) -- +(-110:1) node (R1) {};
\draw[ultra thick, color=brown] (0,.1) -- +(-70:1) node (R2) {};
\draw[ultra thick, color=brown] (1,0) -- +(110:1) node (R3) {};
\draw[ultra thick, color=brown] (1,0) -- +(70:1) node (R4) {};
\draw[ultra thick, color=brown] (-1,0) -- +(110:1) node (R5) {};
\draw[ultra thick, color=brown] (-1,0) -- +(70:1) node (R6) {};
\draw[thick, color=brown] (R1) -- +(-70:.6);
\draw[thick, color=brown] (R1) -- +(-90:.6);
\draw[thick, color=brown] (R1) -- +(-110:.6);
\draw[thick, color=brown] (R2) -- +(-70:.6);
\draw[thick, color=brown] (R2) -- +(-90:.6);
\draw[thick, color=brown] (R2) -- +(-110:.6);
\draw[thick, color=brown] (R3) -- +(70:.6);
\draw[thick, color=brown] (R3) -- +(90:.6);
\draw[thick, color=brown] (R3) -- +(110:.6);
\draw[thick, color=brown] (R4) -- +(70:.6);
\draw[thick, color=brown] (R4) -- +(90:.6);
\draw[thick, color=brown] (R4) -- +(110:.6);
\draw[thick, color=brown] (R5) -- +(70:.6);
\draw[thick, color=brown] (R5) -- +(90:.6);
\draw[thick, color=brown] (R5) -- +(110:.6);
\draw[thick, color=brown] (R6) -- +(70:.6);
\draw[thick, color=brown] (R6) -- +(90:.6);
\draw[thick, color=brown] (R6) -- +(110:.6);
\node[draw,circle,inner sep=2pt,fill=black, thick] at (0,0.1) {};
\node[draw,circle,inner sep=2pt,fill=black, thick] at (1,0) {};
\node[draw,circle,inner sep=2pt,fill=black, thick] at (-1,0) {};
\node[draw,circle,inner sep=2pt,fill=black, thick] at (2,-.3) {};
\node[draw,circle,inner sep=2pt,fill=black, thick] at (-2,-.3) {};
\node[draw,circle,inner sep=2pt,fill=green!50!olive!50, thick] at (3.5,-.3) {};
\node[draw,circle,inner sep=2pt,fill=green!50!olive!50, thick] at (-3.5,-.3) {};
\node[draw,circle,inner sep=2pt,fill=green!50!olive!50, thick] at (G1) {};
\node[draw,circle,inner sep=2pt,fill=green!50!olive!50, thick] at (G2) {};
\node[draw,circle,inner sep=2pt,fill=green!50!olive!50, thick] at (G3) {};
\node[draw,circle,inner sep=2pt,fill=green!50!olive!50, thick] at (G4) {};
\node[draw,circle,inner sep=2pt,fill=green!50!olive!50, thick] at (G5) {};
\node[draw,circle,inner sep=2pt,fill=green!50!olive!50, thick] at (G6) {};
\node[draw,circle,inner sep=2pt,fill=blue, thick] at (2.7,-.6) {};
\node[draw,circle,inner sep=2pt,fill=blue, thick] at (3.2,-.9) {};
\node[draw,circle,inner sep=2pt,fill=blue, thick] at (3.7,-1.3) {};
\node[draw,circle,inner sep=2pt,fill=blue, thick] at (-2.7,-.6) {};
\node[draw,circle,inner sep=2pt,fill=blue, thick] at (-3.2,-.9) {};
\node[draw,circle,inner sep=2pt,fill=blue, thick] at (-3.7,-1.3) {};
\node[draw,circle,inner sep=2pt,fill=red, thick] at (-2.7,0) {};
\node[draw,circle,inner sep=2pt,fill=red, thick] at (-3.2,.3) {};
\node[draw,circle,inner sep=2pt,fill=red, thick] at (-3.7,.7) {};
\node[draw,circle,inner sep=2pt,fill=red, thick] at (2.7,0) {};
\node[draw,circle,inner sep=2pt,fill=red, thick] at (3.2,.3) {};
\node[draw,circle,inner sep=2pt,fill=red, thick] at (3.7,.7) {};
\node[draw,circle,inner sep=2pt,fill=orange, thick] at (O1) {};
\node[draw,circle,inner sep=2pt,fill=orange, thick] at (O2) {};
\node[draw,circle,inner sep=2pt,fill=orange, thick] at (O3) {};
\node[draw,circle,inner sep=2pt,fill=orange, thick] at (O4) {};
\node[draw,circle,inner sep=2pt,fill=orange, thick] at (O5) {};
\node[draw,circle,inner sep=2pt,fill=orange, thick] at (O6) {};
\node[draw,circle,inner sep=2pt,fill=orange, thick] at (O7) {};
\node[draw,circle,inner sep=2pt,fill=orange, thick] at (O8) {};
\node[draw,circle,inner sep=2pt,fill=orange, thick] at (O9) {};
\node[draw,circle,inner sep=2pt,fill=orange, thick] at (O10) {};
\node[draw,circle,inner sep=2pt,fill=orange, thick] at (O11) {};
\node[draw,circle,inner sep=2pt,fill=orange, thick] at (O12) {};
\node[draw,circle,inner sep=2pt,fill=brown, thick] at (R1) {};
\node[draw,circle,inner sep=2pt,fill=brown, thick] at (R2) {};
\node[draw,circle,inner sep=2pt,fill=brown, thick] at (R3) {};
\node[draw,circle,inner sep=2pt,fill=brown, thick] at (R4) {};
\node[draw,circle,inner sep=2pt,fill=brown, thick] at (R5) {};
\node[draw,circle,inner sep=2pt,fill=brown, thick] at (R6) {};
\node[draw,circle,inner sep=2pt,fill=cyan, thick] at (B1) {};
\node[draw,circle,inner sep=2pt,fill=cyan, thick] at (B2) {};
\node[draw,circle,inner sep=2pt,fill=cyan, thick] at (B3) {};
\node[draw,circle,inner sep=2pt,fill=cyan, thick] at (B4) {};
\node[draw,circle,inner sep=2pt,fill=cyan, thick] at (B5) {};
\node[draw,circle,inner sep=2pt,fill=cyan, thick] at (B6) {};
\node[draw,circle,inner sep=2pt,fill=cyan, thick] at (B7) {};
\node[draw,circle,inner sep=2pt,fill=cyan, thick] at (B8) {};
\node[draw,circle,inner sep=2pt,fill=cyan, thick] at (B9) {};
\node[draw,circle,inner sep=2pt,fill=cyan, thick] at (B10) {};
\node[draw,circle,inner sep=2pt,fill=cyan, thick] at (B11) {};
\node[draw,circle,inner sep=2pt,fill=cyan, thick] at (B12) {};
\end{tikzpicture}\]
\end{center}

For $P_1$, each of those vertices are at distance $\beta$ from $A\cap B$, so we only need to count vertices at distance at most $\beta$ from the $r-2$ points of the inside of $A\cap B$. Note that each of them have $q-1$ neighbors going out of $A\cap B$, hence we get 
\[|P_1| = (r-2)(1+(q-1)(1+\cdots + q^{\beta-1})) = (r-2)q^\beta.\]

For $P_2$, we count each endpoints, and their $q-2$ neighbors going outside of both rivers, and we get every vertices in those directions up to distance $\beta$. 

We get 
\[|P_2| = 2 (1+(q-1)(1+q+\cdots + q^{\beta-1})) = 2\left(q^\beta-\frac{q^\beta-1}{q-1}\right).\]

For $P_3$ we move along vertices in $B\backslash (A\cap B)$ at distance at most $\alpha$ from $A\cap B$ and count how many blue vertices are closest to that point. If a vertex on $B$ is at distance $i\geq 1$ from $A\cap B$ 
then we include every neighbor going out of $B$ at distance at most $\min(\beta, \alpha-i)$, itself has $q-1$ vertices going out of $B$ and then each subsequent vertex has $q$ neighbors further away. So we have 
\begin{align*}|P_3| &= 2\sum_{i=1}^\alpha (1+(q-1)(1+\cdots q^{\min(\alpha-i,\beta)-1})) =2\sum_{i=1}^\alpha q^{\min(\alpha-i,\beta)}\\
&=2\left(\sum_{i=1}^{\alpha-\beta} q^{\beta}+\sum_{i=\alpha-\beta+1}^\alpha q^{\alpha-i}\right)=2\left((\alpha-\beta)q^\beta + \frac{q^\beta-1}{q-1}\right).\end{align*}
The same argument shows 
\[|P_4| =2\sum_{i=1}^\beta (1+(q-1)(1+\cdots q^{\min(\beta-i, \alpha)-1})) = 2\sum_{i=1}^\beta q^{\min(\beta-i,\alpha)} . \] 
Since $\beta\leq \alpha$, we can rewrite it as 
\begin{align*}|P_4| &= 2\sum_{i=1}^\beta q^{\beta-i} = 2(1+\cdots +q^{\beta-1}) = 2\frac{q^\beta-1}{q-1}\end{align*}

We can conclude that in that case 
\begin{align*}|X| &= |P_1|+|P_2|+|P_3|+|P_4| \\
&= (r-2)q^\beta+  2\left(q^\beta-\frac{q^\beta-1}{q-1}\right) + 2\left((\alpha-\beta)q^\beta + \frac{q^\beta-1}{q-1}\right) +  2\frac{q^\beta-1}{q-1}\\
&= [r+2(\alpha-\beta)]q^\beta + 2 \frac{q^\beta-1}{q-1}.\end{align*}
\end{proof}

\section{Orbits and stable orbits}
\label{sec:storbits}
In this article, all groups have simply connected derived subgroups, therefore we will not be making any distinction between $\bar{F}$\emph{-orbits} and \emph{stable orbits}.

Let $\mathbf{G}$ be a reductive group over $F$ with simply connected derived subgroup. Let $X$ denote an $F$-variety  with a distinguished closed point $x\in X(F)$ and a transitive right $\mathbf{G}$-action.

The $\emph{orbit}$ of $x$ is defined to be the set \[\mathrm{Orb}_G(x):= x\cdot \mathbf{G}(F) = \{y\in X(F) : \exists g\in \mathbf{G}(F)\ y = x\cdot g \}.\]

The \emph{stable orbit} of $x$ is the set \[\mathrm{Orb_{G}^{\mathrm{st}}}(x):= (x\cdot \mathbf{G}(\bar{F}))\cap X(F) = \{y\in X(F) : \exists g\in \mathbf{G}(\bar{F})\ y =  x\cdot g\}.\]

\begin{rem}\label{rem:orbs}Let $\mathbf{G}_x$ denote stabilizer of $x$. We can idendify  $\mathrm{Orb}_G(x)$ with $\mathbf{G}_x(F)\backslash\mathbf{G}(F)$ whereas $\mathrm{Orb_{G}^{\mathrm{st}}}(x)$ is identified with $(\mathbf{G}_x\backslash\mathbf{G})(F)$.
\end{rem}

\begin{rem}\label{rem:obs}
Although the choice of having a right-action can seem unintuitive, we want to make sure the action matches the one present in Kottwitz notes where $X= \mathbf{G}$ and $g$ act on $x\in \mathbf{G}(F)$ by conjugation 
$x\mapsto g^{-1}xg$. In this exposition we had to either consider right-actions or compose everything with the map $g\mapsto g^{-1}$ which felt arbitrary. 
\end{rem}

To decompose a stable orbit into a union of orbits, we look at the Galois cohomology of the sequence 
\[1\rightarrow \mathbf{G}_x(\bar{F}) \rightarrow \mathbf{G}(\bar{F})\rightarrow (x\cdot \mathbf{G}(\bar{F}))\rightarrow 1,\]
which gives us 
\[1\rightarrow \mathbf{G}_x({F})\rightarrow \mathbf{G}(F)\overset{\varphi}{\rightarrow} \mathrm{Orb_{G}^{\mathrm{st}}}(x)\rightarrow H^1(F,\mathbf{G}_x(\bar{F}) )\overset{\psi}{\rightarrow} H^1(F, \mathbf{G}(\bar{F})).\]
By remark \ref{rem:obs}, the number of orbits within the stable orbit of $x$ is the cardinality of 
\[\mathrm{Orb_{G}^{\mathrm{st}}}(x)/\mathrm{Im}(\varphi) = \mathrm{Coker}(\varphi)\cong \mathrm{Ker}(\psi).\]

In this article, we will only deal with groups which have trivial first Galois cohomology group so we will be able to count the number of orbits within a stable orbit by evaluating 
\[|H^1(F, \mathbf{G}_x)|.\]

\begin{lemma}\label{lem:orbits:GLn}
When $\mathbf{G} = \mathrm{GL}_n$ acting by conjugation on itself, if $\gamma\in \mathbf{G}(F)$ is regular semisimple then 
\[\mathrm{Orb}_{\mathrm{GL}_n}(\gamma) = \mathrm{Orb}_{\mathrm{GL}_n}^{\mathrm{st}}(\gamma)\]
\end{lemma}
\begin{proof}
Since $\gamma$ is regular semisimple, its centralizer is a quasi-split torus. Explicitly, let $F[\gamma] = F[t]/(\chi_\gamma(t))$ be the extension generated by $\gamma$, where $\chi_\gamma$ is the characteristic (and also minimal since $\gamma$ is regular) polynomial of $\gamma$. 
The extension $F[\gamma]$ is an \'etale algebra which we can decompose $F = F_1\oplus \cdots\oplus F_r$ where $F_i/F$ is a field extension. With this description we have 
\[\mathbf{G}_\gamma \cong \mathrm{R}_{F[\gamma]/F}\mathbb{G}_m \cong \prod_{i=1}^r \mathrm{R}_{F_i/F}\mathbb{G}_m.\]
In particular, $\mathbf{G}_\gamma(F) = \prod_{i=1}^r F_i^\times$. An immediate application of Shapiro's lemma and Hilbert Theorem 90 show that $H^1(F, \mathbf{G}_\gamma(\bar{F}))=1$ as desired. 
\end{proof}

\begin{lemma}\label{lemma:slnorbits}
Let $\mathbf{G} = \mathrm{SL}_n$ acting by conjugation on $\gamma \in \mathrm{GL}_n(F)$ where $\gamma$ is regular semisimple. Let $F[\gamma] = F[t]/(\chi_\gamma(t))\cong F_1\oplus \cdots \oplus F_r$ as in the proof of Lemma \ref{lem:orbits:GLn}. Then we have 
\[ \left|\mathrm{Orb}_{\mathrm{SL}_n}^{\mathrm{st}}(\gamma)/\mathrm{Orb}_{\mathrm{SL}_n}(\gamma)\right| = \left|F[\gamma]^\times/N_{F[\gamma]/F}(F[\gamma]^\times)\right|.\]
In particular, if $n=2$ and $F$ is a local field we have 
\[ \mathrm{Orb}_{\mathrm{SL}_n}^{\mathrm{st}}(\gamma)=\mathrm{Orb}_{\mathrm{SL}_n}(\gamma)\]
if the eigenvalues of $\gamma$ are in $F$, otherwise there are exactly two $\mathrm{SL}_n(F)$ orbits within a stable orbit.
\end{lemma}
\begin{proof} In this case, the centralizer of $\gamma$ corresponds to norm $1$ elements of $F[\gamma]$. 
The first part comes from the Galois cohomology of the $\bar{F}$-points of the sequence 
\[1\rightarrow\mathbf{G}_\gamma \rightarrow \mathrm{R}_{F[\gamma]/F}\mathbb{G}_m\overset{\prod_{i}N_{F[\gamma]/F}}{\rightarrow} \prod_{i=1}^r\mathbb{G}_m\rightarrow 1.\]
 The first cohomology group of the middle term vanishes by Hilbert Theorem 90, we get that $ H^1(F,\mathbf{G}_\gamma (\bar{F}) ) \cong F[\gamma]^\times/\mathrm{Im}(N_{F[\gamma]/F})$.

When $F$ is local and $F[\gamma]/F$ is a  cyclic field extension, we have 
\[F[\gamma]^\times/N_{F[\gamma]/F}(F[\gamma]^\times) \cong \mathbb{Z}/[F[\gamma]:F]\mathbb{Z}\]
 by Tate--Nakayama duality. 
 
 In particular, when $n=2$, the algebra  $F[\gamma]$ is isomorphic to either $F\oplus F$ or  a quadratic field extension. In the latter case, 
 the paragraph above yields $2$ rational orbits within the stable orbit. When $F[\gamma]=F\oplus F$, the map 
  $F^\times\times F^\times \rightarrow F^\times, (x,y)\mapsto xy$ is surjective, hence there is only one rational orbit within the stable orbit of $\gamma$.
\end{proof}

\begin{prop}\label{prop:orbitsgl22}
Assume $F$ is a local field. Let $\mathbf{G} = \mathrm{GL}_2\times_{\det}\mathrm{GL}_2$
 acting by conjugation on itself and let $\gamma=(\gamma_1,\gamma_2)\in \mathbf{G}(F)$ be a regular semisimple element.
Let $F[\gamma] = F[t]/(\chi_\gamma(t)) \cong F[\gamma_1]\oplus F[\gamma_2]$.

We have $\mathrm{Orb}_{\mathbf{G}}^{\mathrm{st}}(\gamma)=\mathrm{Orb}_{\mathbf{G}}(\gamma)$ in the following cases:
\begin{itemize}
\item[$(i)$] If $F[\gamma] \cong F^{\oplus 4}$ ($\gamma$ is hyperbolic).
\item[$(ii)$] If $F[\gamma]\cong F^{\oplus 2}\oplus K$, where $K/F$.
\item[$(iii)$] If $F[\gamma]\cong K^{\oplus 2}$ where $K$ is a quadratic field extensions of $F$. 
\end{itemize}
When $F[\gamma]\cong K_1\oplus K_2$ where $K_1,K_2$ are quadratic field extensions, then there are 
exactly two rational orbits within the stable orbit of $\gamma$.
\end{prop}
\begin{proof} Note that if 
 $\mathbf{G}_\gamma \cong \mathbb{G}_m^3$ (case $(i)$) or case $\mathbf{G}_\gamma\cong \mathbb{G}_m\times R_{K/F}\mathbb{G}_m$ (case $(ii)$) then 
 we may use Hilbert Theorem 90 to 
 see that $H^1(F, \mathbf{G}_\gamma(\bar{F})) =1$.\\

When $F[\gamma_1]$ and $F[\gamma_2]$ are fields, then we can use the code in \cite{tamatori} and Nakayama duality, we get that 
\[H^1(F, G_\gamma) \cong H^1(F, \mathbf{X}^\star(G_\gamma)) \cong \left\lbrace\begin{array}{ll}
0     & F[\gamma_1]\neq F[\gamma_2]  \\
\mathbb{Z}/2\mathbb{Z}     & F[\gamma_1]= F[\gamma_2]
\end{array}\right. .\]
\end{proof}
\begin{rem}
We use the same notations as in the previous proposition. 
When $F[\gamma_1]$ and $F[\gamma_2]$ are fields, the set of rational points of the centralizer of $\gamma$ is 
    \[\{(x,y)\in (F[\gamma_1]\times F[\gamma_2])^\times\ :\ N_{F[\gamma_1]/F}(x)=N_{F[\gamma_2]/F}(y)\}.\]

    There are noteworthy applications of tori of this form over a global field related to norm problems and a projective variant of the Hasse Norm Problem, see \cite{phnp-preprint}.

    Also note that the case $F[\gamma_1] = F[\gamma_2]$ we can see in this observation that  
$\mathbf{G}_\gamma \cong \mathbb{G}_m\times \mathrm{R}_{F[\gamma_1]/F}\mathbb{G}_m$ hence Hilbert Theorem 90
yields that the stable orbit consists of a unique rational orbit.

\end{rem}
\begin{ex}
Let $\gamma = \begin{pmatrix}0&-1\\1&0\end{pmatrix}$. We have $\chi_\gamma (t) = t^2+1$.

\begin{itemize}
\item If $F = \mathbb{R}$ then $F[\gamma] = \mathbb{C}$ and the norm map $N_{\mathbb{C}/\mathbb{R}}$ is surjective hence a direct application of the previous lemma tells us that the stable orbit of $\gamma$ is equal 
to its orbit. In other words, if $h\in \mathrm{GL}_2(\mathbb{R})$ is $\mathrm{GL}_2(\mathbb{C})$-conjugate to $\gamma$, then it is also $\mathrm{GL}_2(\mathbb{R})$-conjugate to $\gamma$.
\item If $F = \mathbb{Q}_p$ where $p\equiv 3\ (\mathrm{mod}\ 4)$ then $F[\gamma]$ is the unique unramified extension of $\mathbb{Q}_p$, let's call it $\mathbb{Q}_{p^2}$. Since the extension is unramified and the norm map 
is surjective for all finite fields (hence on the residue fields), then we also have the equality of orbit and stable orbit of $\gamma$. 
\item If $F = \mathbb{Q}_p$ where $p\equiv 1\ (\mathrm{mod}\ 4)$ then $F[\gamma]\cong F\oplus F$ and again, the equality of orbit and stable orbits follow.  
\item If $F = \mathbb{Q}_2$ then $F[\gamma]$ is a quadratic extension and the image of the norm has index $2$ in $\mathbb{Q}_2$ hence there are two orbits within the stable orbit of $\gamma$. One can see that the matrix 
\[\gamma'= \begin{pmatrix}0&1\\-1&0\end{pmatrix}\]
is conjugate to $\gamma$ in $\mathrm{GL}_2(\mathbb{Q}_2)$ since 
\[\gamma = \begin{pmatrix}0&1\\1&0\end{pmatrix}\gamma' \begin{pmatrix}0&1\\1&0\end{pmatrix},\]
but they are not conjugate in $\mathrm{SL}_2(\mathbb{Q}_2)$. Indeed, if they were to be conjugate in $\mathrm{SL}_2(\mathbb{Q}_2)$ then we'd have 
\[g\gamma = \gamma' g\]
where $g$ has to have the form $\begin{pmatrix} a&b\\b&-a\end{pmatrix}$ with $a^2+b^2=-1$. This cannot happen in $\mathbb{Q}_2$ by checking that square residues modulo $8$ are $0,1,4$ and therefore $-1$ is not a sum of two squares. 
\end{itemize}
\end{ex}

\begin{ex}Let us fix $F$ to be a local nonarchimedean field with uniformizer $\varpi$. The group $\mathrm{GL}_2$ has two conjugacy classes of unipotent elements, represented by 
\[\begin{pmatrix}1&0\\0&1\end{pmatrix},\ \text{and}\ \begin{pmatrix}1&1\\0&1\end{pmatrix}.\]
The former element is its own conjugacy class, whereas the centralizer of the latter is the set
\[\left\lbrace \pmat{a&b\\0&a}\right\rbrace\cong F^\times \times F.\]
The additive and multiplicative versions of Hilbert theorem 90 together tell us that the first cohomology group of this centralizer is trivial, so once again we can verify that the orbit 
of $\pmat{1&1\\0&1}$  is its own stable orbit. Obviously, knowing that there are only two stable orbit of unipotent elements, one of them being trivial, would tell us immediately that both orbits are stable orbits, but this 
gives us a way of checking that the orbit of any unipotent element in $\mathrm{GL}_2$ is stable directly. \\

The situation is slightly more complicated for $\mathrm{SL}_2$, however. Assume that the residue characteristic of $F$ is odd. This time, the centralizer of $\pmat{1&1\\0&1}$ still has the same description but we need $a^2 = 1$ hence the centralizer is 
\[\boldsymbol{\mu}_2(F)\times F= \{\pm 1\}\times F.\]
The first cohomology group of $F$ is still trivial, however we have 
\[H^1(F, \boldsymbol{\mu}_2(\bar{F})) = H^1(F, \mathbb{Z}/2\mathbb{Z}) = \mathrm{Hom}(\mathrm{Gal}(\bar{F}/F),\mathbb{Z}/2\mathbb{Z}).\]
By class field theory, this set is isomorphic to the set of degree $2$ \'etale algebras over $F$, hence there are $4$ possibilities:
\[F\oplus F,\ F(\sqrt{\varpi}),\ F(\sqrt{\alpha}),\ \text{and}\ F(\sqrt{\alpha\varpi}),\]
where $\alpha\in \mathcal{O}_F$ has a nonsquare residue. There $4$ extensions give us the $4$ orbits within the stable orbit of $\pmat{1&1\\0&1}$, which are represented by 
\[\pmat{1&1\\0&1},\pmat{1&\alpha \\0&1},\ \pmat{1&\varpi\\0&1},\ \text{and}\ \pmat{1&\alpha\varpi\\0&1}.\]
\end{ex}

\begin{ex}
This example is not as important but it is a classic visual example. For a simpler setting, we look at $X = \mathfrak{sl}_2$ the vector space of trace zero matrices. Pick $M = \pmat{x&y\\z&-x}\in \mathfrak{sl}_2(\mathbb{R})$. In order to be stably conjugate to $M$, 
a matrix of $\mathfrak{sl}_2(\mathbb{R})$ only needs to share the same determinant so the stable orbit can be visualized as trace zero matrices of determinant $yz-x^2$, which is a real equation of a hyperboloid in $\mathbb{R}^3$. 

There are therefore three possibilities: 

\begin{itemize}
\item $yz-x^2 = 0$: this is the equation of a cone (called the nilpotent cone). It is connected and the orbit of $M$ is stable. 
\item $yz-x^2<0$: we say that  $M$ is  hyperbolic, and we can see that its stable orbit is a hyperboloid of one sheet (called hyperbolic hyperboloid) which is again connected, the orbit of $M$ is once again stable. 
\item $yz-x^2>0$: we say that $M$ is elliptic. We get the last possible shape, which is a hyperboloid of two sheets (called elliptic hyperboloid). Each sheet is a connected component which corresponds to one orbit within the stable orbits. Matrices belonging to different 
sheets will be conjugate over $\mathbb{C}$  but not over $\mathbb{R}$.
\end{itemize} 

This makes sense geometrically, outside of the nilpotent cone, both types of hyperboloids are the same connected surface over $\mathbb{C}$, we can do an easy change of variables $(x, y, z)\mapsto (ix,iy,iz)$ to send the solutions of $yz-x^2 = d$ to the solutions of $yz-x^2=-d$, however the set of real solutions is not always connected.   

\begin{figure}[H]
\centering
\begin{tikzpicture}[el/.style args={#1,#2}{draw,ellipse,minimum width=#1, minimum height=#2},outer sep=0pt,>=latex']

\node(el-1) [el={.75cm,2cm}]at (0,0){};
\node(el-2) [el={.75cm,2cm}]at (3,0){};
\path [fill=white] (el-1.87)rectangle([shift={(1cm,-.1cm)}]el-1.-87);
\draw (el-1.87)to[bend right=20](el-2.93) (el-1.-87)to[bend left=20](el-2.-93);

\begin{scope}
\path[clip](1,-1)rectangle(1.5,1);
\end{scope}


\node(el-3) [el={.75cm,2cm}]at (-1,0){};
\node(el-4) [el={.75cm,2cm}]at (4,0){};
\path [fill=white] (el-3.87)rectangle([shift={(0.5cm,-.1cm)}]el-3.-87);
\draw[dashed] (-0.25,0.7) -- (2.75,-0.5);
\draw (-1,1) -- (-0.25,0.7);
\draw (4,-1) -- (2.75,-0.5);
\draw[dashed] (-0.25,-0.7) -- (2.75,0.5);
\draw (-1,-1) -- (-0.25,-0.7);
\draw (4,1) -- (2.75,0.5);

\node(el-5) [el={.75cm,2cm}]at (-2.5,0){};
\node(el-6) [el={.75cm,2cm}]at (5.5,0){};
\path [fill=white] (el-5.87)rectangle([shift={(0.5cm,-.1cm)}]el-5.-87);
\path [fill=white] (el-6.87)rectangle([shift={(-0.5cm,-.1cm)}]el-6.-87);
\draw (-2.5,-1) to[out=30,in=-90] (-1.5,0) to[out=90, in=-30] (-2.5,1);
\draw (5.5,1) to[out=-150,in=90] (4.5,0) to[out=-90, in=150] (5.5,-1);
\end{tikzpicture}
\caption{From left to right in appearance: elliptic orbit, nilpotent orbit, hyperbolic orbit}
\end{figure}
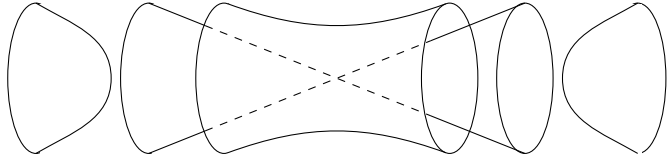

\end{ex}

\subsection{Orbital integrals}
Using the same notations as before, assuming a choice of measures $\mu_x$ (resp. $\mu^{\mathrm{st}}_s$) on the  orbit (resp. stable orbit) of $x\in X$ and integrate specific test functions. Such an integral would be called a  \emph{orbital integral} (resp. stable orbital integral). 

We will denote them respectively by

\[O(\mathbf{G},x, f) := \int_{\mathrm{Orb}_G(x)}f(y)\ \mathrm{d}\mu_x(y),\]
\[O^{\mathrm{st}}(\mathbf{G},x, f) := \int_{\mathrm{Orb}^{\mathrm{st}}_G(x)}f(y)\ \mathrm{d}\mu^{\mathrm{st}}_{x}(y).\]

\begin{rem}\textbf{Compatibility of measures.} Assume that $\mathrm{Orb}^{\mathrm{st}}_G(x)$ can be written in a disjoint union of orbits, with representatives $x=x_0, \cdots, x_n$. In this article we will pick measures so that  
\[O^{\mathrm{st}}(\mathbf{G},x,f) = \sum_{i=1}^n O(\mathbf{G},x_i,f|_{\mathrm{Orb}_G(x_i)}).\]
This is however not always the case in the literature hence some of the exposition may be written elsewhere up to some constant.\\

In particular, given $\gamma\in \mathrm{GL}_n$,  we have that  $\mathrm{Orb}_{\mathrm{GL}_n}(\gamma) = \mathrm{Orb}^{\mathrm{st}}_{\mathrm{SL}_n}(\gamma)$, so we will pick measures on orbits so that 
\[O(\mathrm{GL}_n,\gamma,f) = \sum_{i=1}^n O(\mathrm{SL}_n,\gamma_i,f|_{\mathrm{Orb}_{\mathrm{SL}_n}(\gamma_i)}),\]
where $\gamma_1,\cdots, \gamma_n$ are representatives of the $\mathrm{SL}_n$ orbits of $\gamma$ within its $\mathrm{GL}_n$-orbit.

\end{rem}

In practice, the term orbital integral usually only refers to the specific choice of an algebraic group acting on either itself or its Lie algebra. 

\section{Aside: Relative orbital integrals}
\label{sec:relative}

With the recent progress on the relative Langlands Programme, the computation of relative orbital integrals with respect to spherical pairs has also gained importance.

In this section, we will compute orbital integrals for the diagonal action of $G = \mathrm{GL}_2$ on $\mathrm{GL}_2\times \mathrm{GL}_2$. Similar computations for the unitary group will be part of another paper with Wei Zhang, with applications to the Arithmetic Transfer Conjecture.

Let $\gamma = (\gamma_1,\gamma_2)\in \mathrm{GL}_2(F)\times\mathrm{GL}_2(F)$ be a regular semisimple element with respect to the diagonal action of $\mathrm{GL}_2(F)$. 
We want to compute
\[\mathrm{Orb}(\gamma_1,\gamma_2) = \int_{\mathrm{Orb}_{g\in\mathrm{GL}_2(F)}(\gamma_1,\gamma_2)}\mathds{1}_{\mathrm{GL}_2\times\mathrm{GL}_2({\mathcal{O}_F})}(g)\ \mathrm{d}g\]

Similarly to the classic setting, note that 
\[g\cdot(\gamma_1,\gamma_1) = (g^{-1}\gamma_1 g,g^{-1}\gamma_2 g)\in \mathrm{GL}_2\times\mathrm{GL}_2({\mathcal{O}_F}) = \mathrm{Stab}_{\mathrm{GL}_2(F)}\Lambda_0\times \mathrm{Stab}_{\mathrm{GL}_2(F)}\Lambda_0\]
if and only if 
\[\gamma_1\text{ fixes }g\Lambda_0\text{ and}\ \gamma_1\text{ fixes }g\Lambda_0.\]
We immediately get
\begin{prop}
    If $X^{\gamma_1}\cap X^{\gamma_2}$ is finite, then $\mathrm{Stab}_{\mathrm{GL}_2(F)}(\gamma_1,\gamma_2)$ is compact modulo center, and  
    \[\mathrm{Orb}(\gamma_1,\gamma_2) = \mathrm{vol}\left(F^\times\backslash \mathrm{Stab}_{\mathrm{GL}_2(F)}(\gamma_1,\gamma_2)\right)^{-1}|X^{\gamma_1}\cap X^{\gamma_2}|.\]
\end{prop}
As in Kottwitz's work, we may choose a Haar measure on $\mathrm{Stab}_{\mathrm{GL}_2(F)}$ giving its quotient modulo center measure $1$ so that $\mathrm{Orb}(\gamma_1,\gamma_2) = |X^{\gamma_1}\cap X^{\gamma_2}|$.\\

Let us recall the structure of $X^{\gamma_i}$ depending on the type of $\gamma_i$ from \S\ref{sec:fixedpoints}.

\begin{itemize}
    \item If $\gamma_i$ is hyperbolic, then $X^{\gamma_i}$ is a tube containing all vertices within a certain distance to the appartment corresponding to the centralizer of $\gamma_i$.
    \item If $\gamma_i$ is elliptic and $F[\gamma_i]$  is unramified, then $X^{\gamma_i}$ is a ball centered at a vertex.
    \item If $\gamma_i$ is elliptic and $F[\gamma_i]$  is ramified, then $X^{\gamma_i}$ is a ball centered at a midpoint of an edge.
\end{itemize}

The only case where $X^{\gamma_1}\cap X^{\gamma_2}$ is infinite is when both $\gamma_1,\gamma_2$ are hyperbolic and $A_{\gamma_1}\cap A_{\gamma_2}$ is infinite. Let us treat the other cases first.

\subsection{Hyperbolic elements}

Assume that $\gamma_1,\gamma_2$ are hyperbolic, i.e. $G_{\gamma_1}$ and $G_{\gamma_2}$ are split tori. We have four possible relative configurations for $A_{\gamma_1}$ and $A_{\gamma_2}$:
\begin{itemize}
    \item (A) $A_{\gamma_1}\cap A_{\gamma_2} = \varnothing$.
    \item (B) $0<|A_{\gamma_1}\cap A_{\gamma_2}|<\infty$.
    \item (C) $A_{\gamma_1}= A_{\gamma_2}$.
    \item (D) $A_{\gamma_1}\neq A_{\gamma_2}$ and $|A_{\gamma_1}\cap A_{\gamma_2}|=\infty$.
\end{itemize}

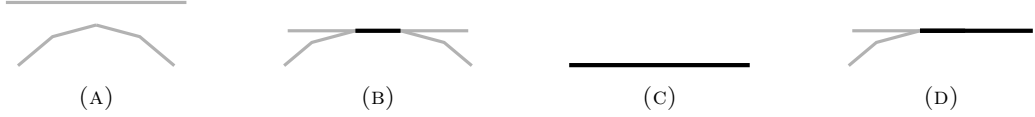
\begin{figure}[H]
\begin{subfigure}{.24\textwidth}
\centering
\begin{tikzpicture}[scale=0.6]
\draw[very thick, color=gray!60] (-2,0) -- (2,0);

\draw[very thick, color=gray!60] (0,-0.5) -- ++(-15:1) -- ++(-40:1);
\draw[very thick, color=gray!60] (0,-0.5) -- ++(195:1) -- ++(220:1);

\end{tikzpicture}
\caption{ }
\end{subfigure}
\begin{subfigure}{.24\textwidth}
\centering
\begin{tikzpicture}[scale=0.6]
\draw[very thick, color=gray!60] (-2,0) -- (2,0);

\draw[very thick, color=gray!60] (.5,0) -- ++(-15:1) -- ++(-40:0.8);
\draw[very thick, color=gray!60] (-.5,0) -- ++(195:1) -- ++(220:0.8);
\draw[ultra thick] (-.5,0) --(.5,0);
\end{tikzpicture}
\caption{}
\end{subfigure}
\begin{subfigure}{.24\textwidth}
\centering
\begin{tikzpicture}[scale=0.6]
\draw[ultra thick] (-2,0) -- (2,0);
\end{tikzpicture}
\caption{}
\end{subfigure}
\begin{subfigure}{.24\textwidth}
\centering
\begin{tikzpicture}[scale=0.6]
\draw[very thick, color=gray!60] (-2,0) -- (-.5,0);
\draw[ultra thick] (-.5,0) -- (2,0);
\draw[very thick, color=gray!60] (-.5,0) -- ++(195:1) -- ++(220:0.8);
\draw[ultra thick] (-.5,0) --(.5,0);
\end{tikzpicture}
\caption{ }
\end{subfigure}
\end{figure}

Cases (C) and (D) mean that $A_{\gamma_1}$ and $A_{\gamma_2}$ represent the same sector, i.e. point at infinity. Concretely speaking, this means that both tori $G_{\gamma_1}$ and $G_{\gamma_2}$ can be embedded in the same Borel subgroup. In case (C) those two subgroups are the same, but in case (D) we have $G_{\gamma_1}\cap G_{\gamma_2} = F^\times$, the center. As a result, the quotient of the stabilizer of $(\gamma_1,\gamma_2)$ by the center is trivial, hence the orbital integral is equal to the cardinality of $X^{\gamma_1}\cap X^{\gamma_2}$, which is infinite. We get that the orbital integral is not converging in case (D).

\begin{ex}
Let $\gamma_1 = \begin{pmatrix}
    1&0\\0&-1
\end{pmatrix}$ and $\gamma_2 =\begin{pmatrix}
    1&2\\0&-1
\end{pmatrix}  = \begin{pmatrix}
    1&-1\\0&1
\end{pmatrix}  \gamma_1 \begin{pmatrix}
    1&1\\0&1
\end{pmatrix}$. We have that $G_{\gamma_1}$ is the diagonal torus and $G_{\gamma_2} = \left\lbrace\begin{pmatrix}
        a&a-b\\0&b
    \end{pmatrix}: a,b\in F^\times\right\rbrace$.  We can verify that $G_{\gamma_1}\cap G_{\gamma_2}$ is the center.
\end{ex}

For case (C), however, the orbital integral is just an ordinary orbital integral over $\mathrm{GL}_2$, and therefore
  \[\mathrm{Orb}(\gamma_1,\gamma_2) = q^{\min(d_{\gamma_1}, d_{\gamma_2})}.\]

For case (A) and (C), the stabilizer of $(\gamma_1,\gamma_2)$ is compact modulo center, hence we have 
\[ \mathrm{Orb}(\gamma_1,\gamma_2) =|X^{\gamma_1}\cap X^{\gamma_2}|.\]

The explicit computations are carried out in \S \ref{sec:combin-tre}. Let us only mention the result here. In case (A), $\delta$ will denote the combinatorial distance between $A_{\gamma_1}$ and $A_{\gamma_2}$, which we will make explicit later.

We have the following.

\begin{itemize}
    \item Case (A). Up to permutation, assume that $d_{\gamma_1}\geq d_{\gamma_2}$.
    \[\mathrm{Orb}(\gamma_1,\gamma_2) =\left\lbrace \begin{array}{ll}0&\text{if } d_{\gamma_1}+d_{\gamma_2}< \delta\\ 
    &\\
1+(q+1)\frac{q^{(d_{\gamma_1}+d_{\gamma_2}-\delta)/2}-1}{q-1}&{\text{if }}\left\lbrace\begin{array}{l}|d_{\gamma_1}-d_{\gamma_2}|\leq \delta\leq d_{\gamma_1}+d_{\gamma_2}\\d_{\gamma_1}+d_{\gamma_2}-\delta\equiv 0\pmod2\\\end{array} \right.\\ &\\ 2\frac{q^{\frac{1}{2}(d_{\gamma_1}+d_{\gamma_2}-\delta+1)}-1}{q-1}&{\text{if }}\left\lbrace\begin{array}{l}|d_{\gamma_1}-d_{\gamma_2}|\leq \delta\leq d_{\gamma_1}+d_{\gamma_2}\\d_{\gamma_1}+d_{\gamma_2}-\delta\equiv 1\pmod2\\\end{array} \right.\\  &\\
(1+2(d_{\gamma_1}-\delta-d_{\gamma_2}))q^{d_{\gamma_2}} + 2\frac{q^{d_{\gamma_2}}-1}{q-1}&\text{if } \delta\leq |d_{\gamma_1}-d_{\gamma_2}|\\
\end{array}\right.\]

\item Case (B). Let $r = |A_{\gamma_1}\cap A_{\gamma_2}|$. We get 
\[\mathrm{Orb}(\gamma_1,\gamma_2) = (2(d_{\gamma_1}-d_{\gamma_2})+r)q^{d_{\gamma_2}}-2\frac{q^{d_{\gamma_2}}-1}{q-1}.
\]
\end{itemize}

\subsection{Elliptic elements}
Up to permutation of $\gamma_1$ and $\gamma_2$ we  now have 3 cases. 

\begin{itemize}
    \item (E) $F[\gamma_1]$, $F[\gamma_2]$ are unramified.
    \item (F) $F[\gamma_1]$, $F[\gamma_2]$ are ramified.
    \item (G) $F[\gamma_1]$ is unramified, $F[\gamma_2]$ is ramified.
\end{itemize}
In all cases we need to compute the number of vertices within the intersection of two balls on the building (which is itself a ball). In case (E) both balls are centered at a vertex, in case (F) both balls are centered at the midpoint of an edge, and in case (G) there is one of each type.

Recall that we will use the fact here that $F$ has odd residue characteristic. In the non-tamely ramified case, the set of Galois-fixed points on the enlarged building can be strictly bigger than the smaller building. In that case, even a Galois-invariant apartment in the big building need not intersect the smaller one, so the distance of the closest vertices of the small building to the apartment may be larger than $1/2$. This distance would then need to be subtracted from the radius. Explicitly, the same results would be true, where in the case of ramified $\gamma_i$ $d_{\gamma_i}$, we replace $d_{\gamma_i}$ by $\frac{1}{2}\sup\{\mathrm{val}_{F[\gamma_i]}(\gamma_i-a): a\in\mathcal{O}_F^\times\}$ as in \cite[p.\ 418]{kott_bible} with the minor difference that our $d_{\gamma_i}$ is the distance from the closest point to the apartment, whereas in Kottwitz's notes it is the distance from the edge containing that point, hence the difference of $1/2$.  \\

We will let $\delta$ be the distance between the centers of $X^{\gamma_1}$ and $X^{\gamma_2}$. Equivalently, it is also the distance between the apartments corresponding to the centralizers of $\gamma_1$ and $\gamma_2$ in the extended building.

We will use Lemma \ref{lem:interballs} to reduce each count to the number of vertices within a ball. In particular, the second case of the Lemma tells us where the center of the ball is in the only non trivial case. 
\begin{itemize}
    \item Case (E). In the non trivial case, we either have a center on a vertex or on a midpoint, depending on the parity of $d_{\gamma_1}-d_{\gamma_2}+\delta$. We get
    \[\mathrm{Orb}(\gamma_1,\gamma_2) = \left\lbrace \begin{array}{ll}
       0& \text{if }d_{\gamma_1}+d_{\gamma_2}<\delta    \\
       &\\
        1 + (1+q)\frac{q^{\frac{1}{2}(d_{\gamma_1}+d_{\gamma_2}-\delta)}-1}{q-1}    &\text{if }\left\lbrace\begin{array}{l}
            |d_{\gamma_1}-d_{\gamma_2}|<\delta  \\
              d_{\gamma_1}+d_{\gamma_2}+\delta \equiv0\pmod2
        \end{array}\right.\\ &\\
         2\frac{q^{\frac{1}{2}(d_{\gamma_1}+d_{\gamma_2}-\delta+1)}-1}{q-1}&\text{if }\left\lbrace\begin{array}{l}
            |d_{\gamma_1}-d_{\gamma_2}|<\delta  \\
              d_{\gamma_1}+d_{\gamma_2}+\delta \equiv1\pmod2
        \end{array}\right.  \\
        &\\  
     1 + (1+q)\frac{q^{\min(d_{\gamma_1},d_{\gamma_2})}-1}{q-1}    &\text{if }\delta\leq |d_{\gamma_1}-d_{\gamma_2}|
    \end{array} \right.\]

       \begin{figure}[H]
   \centering
    \begin{tikzpicture}
        \draw[thick] (0,0) -- (7,0);
        \draw[thick, dashed] (7,0) -- (8,0);
        \begin{scope}[shift=(0:4)]    
        \draw[thick] (0,0) -- (0,0.5);
        \begin{scope}[shift=(90:0.5)]
        \draw[thick] (0,0) -- (120:0.6);
        \draw[thick] (0,0) -- (60:0.6);
        \begin{scope}[shift=(120:0.6)]
        \draw[thick] (0,0) -- (120:0.5);
        \draw[thick] (0,0) -- (60:0.5);
        \end{scope}
        \begin{scope}[shift=(60:0.6)]
        \draw[thick] (0,0) -- (120:0.5);
        \draw[thick] (0,0) -- (60:0.5);
        \end{scope}
        \end{scope}
        \end{scope}
        \foreach \i in {0,1,...,7}{
        \node[draw,circle,inner sep=2pt,fill=black, thick] at (\i,0) {};
        }
        \foreach \i in {0,1,...,4}{
        \node[draw,circle,inner sep=3pt,color=red, thick] at (\i,0) {};
        }
        \node[draw,circle,inner sep=1pt,fill=blue, thick] at (3.5,0) {};
        \node[draw,circle,inner sep=2pt,fill=red, thick] at (1,0) {};
        \node[draw,circle,inner sep=2pt,fill=green!40!gray, thick] at (6,0) {};
        \foreach\j in{3,4,...,7}{
        \begin{scope}[shift=(0:\j),scale=0.5]
        \draw[very thick,color=green!40!gray] (-.2,-0.3) -- (.2,-0.3) to [out=0,in=-90] (.3,-0.2) -- (.3,0.2) to [out=90,in=0] (.2,0.3) 
-- (-.2,0.3) to[out=180,in=90] (-.3,0.2) -- (-.3,-0.2) to[out=-90,in=180] (-.2,-0.3);
        \end{scope}}
        \begin{scope}[shift={(4,0.6)},scale=0.5]
        \node[draw,circle,inner sep=2pt,fill=black, thick] at (0,0) {};
        \draw[very thick,color=green!40!gray] (-.2,-0.3) -- (.2,-0.3) to [out=0,in=-90] (.3,-0.2) -- (.3,0.2) to [out=90,in=0] (.2,0.3) 
-- (-.2,0.3) to[out=180,in=90] (-.3,0.2) -- (-.3,-0.2) to[out=-90,in=180] (-.2,-0.3);
        \end{scope}
    \end{tikzpicture}
    \caption{Case $d_{\gamma_1}=3,\ d_{\gamma_2}=3,\ \delta=5$. The blue node is the center of $X^{\gamma_1}\cap X^{\gamma_2}$}
    \end{figure}
       \begin{figure}[H]
   \centering
    \begin{tikzpicture}
        \draw[thick] (0,0) -- (7,0);
        \draw[thick, dashed] (7,0) -- (8,0);
        \begin{scope}[shift=(0:4)]    
        \draw[thick] (0,0) -- (0,0.5);
        \begin{scope}[shift=(90:0.5)]
        \draw[thick] (0,0) -- (120:0.6);
        \draw[thick] (0,0) -- (60:0.6);
        \begin{scope}[shift=(120:0.6)]
        \draw[thick] (0,0) -- (120:0.5);
        \draw[thick] (0,0) -- (60:0.5);
        \end{scope}
        \begin{scope}[shift=(60:0.6)]
        \draw[thick] (0,0) -- (120:0.5);
        \draw[thick] (0,0) -- (60:0.5);
        \end{scope}
        \end{scope}
        \end{scope}
        \begin{scope}[shift=(0:3),rotate=180]    
        \draw[thick] (0,0) -- (0,0.5);
        \begin{scope}[shift=(90:0.5)]
        \draw[thick] (0,0) -- (120:0.6);
        \draw[thick] (0,0) -- (60:0.6);
        \begin{scope}[shift=(120:0.6)]
        \draw[thick] (0,0) -- (120:0.5);
        \draw[thick] (0,0) -- (60:0.5);
        \end{scope}
        \begin{scope}[shift=(60:0.6)]
        \draw[thick] (0,0) -- (120:0.5);
        \draw[thick] (0,0) -- (60:0.5);
        \end{scope}
        \end{scope}
        \end{scope}
        \foreach \i in {0,1,...,7}{
        \node[draw,circle,inner sep=2pt,fill=black, thick] at (\i,0) {};
        }
        \foreach \i in {0,1,...,4}{
        \node[draw,circle,inner sep=3pt,color=red, thick] at (\i,0) {};
        }
        \node[draw,circle,inner sep=2pt,fill=blue, thick] at (3,0) {};
        \node[draw,circle,inner sep=2pt,fill=red, thick] at (1,0) {};
        \node[draw,circle,inner sep=2pt,fill=green!40!gray, thick] at (6,0) {};
        \foreach\j in{2,3,...,7}{
        \begin{scope}[shift=(0:\j),scale=0.5]
        \draw[very thick,color=green!40!gray] (-.2,-0.3) -- (.2,-0.3) to [out=0,in=-90] (.3,-0.2) -- (.3,0.2) to [out=90,in=0] (.2,0.3) 
-- (-.2,0.3) to[out=180,in=90] (-.3,0.2) -- (-.3,-0.2) to[out=-90,in=180] (-.2,-0.3);
        \end{scope}}
        \begin{scope}[shift={(4,0.6)},scale=0.5]
        \node[draw,circle,inner sep=2pt,fill=black, thick] at (0,0) {};
        \draw[very thick,color=green!40!gray] (-.2,-0.3) -- (.2,-0.3) to [out=0,in=-90] (.3,-0.2) -- (.3,0.2) to [out=90,in=0] (.2,0.3) 
-- (-.2,0.3) to[out=180,in=90] (-.3,0.2) -- (-.3,-0.2) to[out=-90,in=180] (-.2,-0.3);
        \end{scope}
        \begin{scope}[shift={(4,0.6)}]
        \begin{scope}[shift=(60:0.5),scale=0.5]
        \node[draw,circle,inner sep=2pt,fill=black, thick] at (0,0) {};
        \draw[very thick,color=green!40!gray] (-.2,-0.3) -- (.2,-0.3) to [out=0,in=-90] (.3,-0.2) -- (.3,0.2) to [out=90,in=0] (.2,0.3) 
-- (-.2,0.3) to[out=180,in=90] (-.3,0.2) -- (-.3,-0.2) to[out=-90,in=180] (-.2,-0.3);
        \end{scope}
        \end{scope}
         \begin{scope}[shift={(4,0.6)}]
        \begin{scope}[shift=(120:0.5),scale=0.5]
        \node[draw,circle,inner sep=2pt,fill=black, thick] at (0,0) {};
        \draw[very thick,color=green!40!gray] (-.2,-0.3) -- (.2,-0.3) to [out=0,in=-90] (.3,-0.2) -- (.3,0.2) to [out=90,in=0] (.2,0.3) 
-- (-.2,0.3) to[out=180,in=90] (-.3,0.2) -- (-.3,-0.2) to[out=-90,in=180] (-.2,-0.3);
        \end{scope}
        \end{scope}
        \begin{scope}[shift={(3,-0.6)},scale=0.5]
        \node[draw,circle,inner sep=3pt,color=red, thick] at (0,0) {};
        \node[draw,circle,inner sep=2pt,fill=black, thick] at (0,0) {};
        \draw[very thick,color=green!40!gray] (-.2,-0.3) -- (.2,-0.3) to [out=0,in=-90] (.3,-0.2) -- (.3,0.2) to [out=90,in=0] (.2,0.3) 
-- (-.2,0.3) to[out=180,in=90] (-.3,0.2) -- (-.3,-0.2) to[out=-90,in=180] (-.2,-0.3);
        \end{scope}
    \end{tikzpicture}
    \caption{Case $d_{\gamma_1}=3,\ d_{\gamma_2}=4,\ \delta=5$. The blue node is the center of $X^{\gamma_1}\cap X^{\gamma_2}$}
    \end{figure}
    \item Case (F). Here  both $d_{\gamma_1}$ and $d_{\gamma_2}$ are half-integers, whereas $\delta$ is an integer. Therefore, $d_{\gamma_1}-d_{\gamma_2}+r$ is still an integer, and depending on the parity, we get a vertex (if odd), or a half integer (if even). This is the same case as previously. The center of $X^{\gamma_1}\cap X^{\gamma_2}$ is also a half point if $d_{\gamma_1}-d_{\gamma_2}+\delta$ is even. Just note that by virtue of being half integers, the parity of the latter alternate sum is the opposite parity to $d_{\gamma_1}+d_{\gamma_2}+\delta$.
      \[\mathrm{Orb}(\gamma_1,\gamma_2) = \left\lbrace \begin{array}{ll}
       0& \text{if }d_{\gamma_1}+d_{\gamma_2}<\delta    \\
       &\\
       1 + (1+q)\frac{q^{\frac{1}{2}(d_{\gamma_1}+d_{\gamma_2}-\delta)}-1}{q-1}    &\text{if }\left\lbrace\begin{array}{l}
            |d_{\gamma_1}-d_{\gamma_2}|<\delta  \\
              d_{\gamma_1}+d_{\gamma_2}+\delta \equiv0\pmod2
        \end{array}\right.\\ &\\
       2\frac{q^{\frac{1}{2}(d_{\gamma_1}+d_{\gamma_2}-\delta+1)}-1}{q-1}&\text{if }\left\lbrace\begin{array}{l}
            |d_{\gamma_1}-d_{\gamma_2}|<\delta  \\
              d_{\gamma_1}+d_{\gamma_2}+\delta \equiv1\pmod2
        \end{array}\right.  \\
        &\\  
    2\frac{q^{\min(d_{\gamma_1},d_{\gamma_2})+1/2}-1}{q-1}    &\text{if }\delta\leq |d_{\gamma_1}-d_{\gamma_2}|
    \end{array} \right.\]

   \begin{figure}[H]
   \centering
    \begin{tikzpicture}
        \draw[thick] (0,0) -- (7,0);
        \draw[thick, dashed] (7,0) -- (8,0);
        \begin{scope}[shift=(0:4)]    
        \draw[thick] (0,0) -- (0,0.5);
        \begin{scope}[shift=(90:0.5)]
        \draw[thick] (0,0) -- (120:0.6);
        \draw[thick] (0,0) -- (60:0.6);
        \begin{scope}[shift=(120:0.6)]
        \draw[thick] (0,0) -- (120:0.5);
        \draw[thick] (0,0) -- (60:0.5);
        \end{scope}
        \begin{scope}[shift=(60:0.6)]
        \draw[thick] (0,0) -- (120:0.5);
        \draw[thick] (0,0) -- (60:0.5);
        \end{scope}
        \end{scope}
        \end{scope}
        \foreach \i in {0,1,...,7}{
        \node[draw,circle,inner sep=2pt,fill=black, thick] at (\i,0) {};
        }
        \foreach \i in {0,1,...,4}{
        \node[draw,circle,inner sep=3pt,color=red, thick] at (\i,0) {};
        }
        \node[draw,circle,inner sep=1pt,fill=blue, thick] at (3.5,0) {};
        \node[draw,circle,inner sep=1pt,fill=red, thick] at (0.5,0) {};
        \node[draw,circle,inner sep=1pt,fill=green!40!gray, thick] at (5.5,0) {};
        \foreach\j in{3,4,...,7}{
        \begin{scope}[shift=(0:\j),scale=0.5]
        \draw[very thick,color=green!40!gray] (-.2,-0.3) -- (.2,-0.3) to [out=0,in=-90] (.3,-0.2) -- (.3,0.2) to [out=90,in=0] (.2,0.3) 
-- (-.2,0.3) to[out=180,in=90] (-.3,0.2) -- (-.3,-0.2) to[out=-90,in=180] (-.2,-0.3);
        \end{scope}}
        \begin{scope}[shift={(4,0.6)},scale=0.5]
        \node[draw,circle,inner sep=2pt,fill=black, thick] at (0,0) {};
        \draw[very thick,color=green!40!gray] (-.2,-0.3) -- (.2,-0.3) to [out=0,in=-90] (.3,-0.2) -- (.3,0.2) to [out=90,in=0] (.2,0.3) 
-- (-.2,0.3) to[out=180,in=90] (-.3,0.2) -- (-.3,-0.2) to[out=-90,in=180] (-.2,-0.3);
        \end{scope}
    \end{tikzpicture}
    \caption{Case $d_{\gamma_1}=3.5,\ d_{\gamma_2}=2.5,\ \delta=5$. The blue node is the center of $X^{\gamma_1}\cap X^{\gamma_2}$}
    \end{figure}
       \begin{figure}[H]
   \centering
    \begin{tikzpicture}
        \draw[thick] (0,0) -- (7,0);
        \draw[thick, dashed] (7,0) -- (8,0);
        \begin{scope}[shift=(0:4)]    
        \draw[thick] (0,0) -- (0,0.5);
        \begin{scope}[shift=(90:0.5)]
        \draw[thick] (0,0) -- (120:0.6);
        \draw[thick] (0,0) -- (60:0.6);
        \begin{scope}[shift=(120:0.6)]
        \draw[thick] (0,0) -- (120:0.5);
        \draw[thick] (0,0) -- (60:0.5);
        \end{scope}
        \begin{scope}[shift=(60:0.6)]
        \draw[thick] (0,0) -- (120:0.5);
        \draw[thick] (0,0) -- (60:0.5);
        \end{scope}
        \end{scope}
        \end{scope}
        \begin{scope}[shift=(0:3),rotate=180]    
        \draw[thick] (0,0) -- (0,0.5);
        \begin{scope}[shift=(90:0.5)]
        \draw[thick] (0,0) -- (120:0.6);
        \draw[thick] (0,0) -- (60:0.6);
        \begin{scope}[shift=(120:0.6)]
        \draw[thick] (0,0) -- (120:0.5);
        \draw[thick] (0,0) -- (60:0.5);
        \end{scope}
        \begin{scope}[shift=(60:0.6)]
        \draw[thick] (0,0) -- (120:0.5);
        \draw[thick] (0,0) -- (60:0.5);
        \end{scope}
        \end{scope}
        \end{scope}
        \foreach \i in {0,1,...,7}{
        \node[draw,circle,inner sep=2pt,fill=black, thick] at (\i,0) {};
        }
        \foreach \i in {0,1,...,4}{
        \node[draw,circle,inner sep=3pt,color=red, thick] at (\i,0) {};
        }
        \node[draw,circle,inner sep=2pt,fill=blue, thick] at (3,0) {};
        \node[draw,circle,inner sep=1pt,fill=red, thick] at (0.5,0) {};
        \node[draw,circle,inner sep=1pt,fill=green!40!gray, thick] at (5.5,0) {};
        \foreach\j in{2,3,...,7}{
        \begin{scope}[shift=(0:\j),scale=0.5]
        \draw[very thick,color=green!40!gray] (-.2,-0.3) -- (.2,-0.3) to [out=0,in=-90] (.3,-0.2) -- (.3,0.2) to [out=90,in=0] (.2,0.3) 
-- (-.2,0.3) to[out=180,in=90] (-.3,0.2) -- (-.3,-0.2) to[out=-90,in=180] (-.2,-0.3);
        \end{scope}}
        \begin{scope}[shift={(4,0.6)},scale=0.5]
        \node[draw,circle,inner sep=3pt,color=red, thick] at (0,0) {};
        \node[draw,circle,inner sep=2pt,fill=black, thick] at (0,0) {};
        \draw[very thick,color=green!40!gray] (-.2,-0.3) -- (.2,-0.3) to [out=0,in=-90] (.3,-0.2) -- (.3,0.2) to [out=90,in=0] (.2,0.3) 
-- (-.2,0.3) to[out=180,in=90] (-.3,0.2) -- (-.3,-0.2) to[out=-90,in=180] (-.2,-0.3);
        \end{scope}
        \begin{scope}[shift={(4,0.6)}]
        \begin{scope}[shift=(60:0.5),scale=0.5]
        \node[draw,circle,inner sep=2pt,fill=black, thick] at (0,0) {};
        \draw[very thick,color=green!40!gray] (-.2,-0.3) -- (.2,-0.3) to [out=0,in=-90] (.3,-0.2) -- (.3,0.2) to [out=90,in=0] (.2,0.3) 
-- (-.2,0.3) to[out=180,in=90] (-.3,0.2) -- (-.3,-0.2) to[out=-90,in=180] (-.2,-0.3);
        \end{scope}
        \end{scope}
         \begin{scope}[shift={(4,0.6)}]
        \begin{scope}[shift=(120:0.5),scale=0.5]
        \node[draw,circle,inner sep=2pt,fill=black, thick] at (0,0) {};
        \draw[very thick,color=green!40!gray] (-.2,-0.3) -- (.2,-0.3) to [out=0,in=-90] (.3,-0.2) -- (.3,0.2) to [out=90,in=0] (.2,0.3) 
-- (-.2,0.3) to[out=180,in=90] (-.3,0.2) -- (-.3,-0.2) to[out=-90,in=180] (-.2,-0.3);
        \end{scope}
        \end{scope}
        \begin{scope}[shift={(3,-0.6)},scale=0.5]
        \node[draw,circle,inner sep=2pt,fill=black, thick] at (0,0) {};
        \draw[very thick,color=green!40!gray] (-.2,-0.3) -- (.2,-0.3) to [out=0,in=-90] (.3,-0.2) -- (.3,0.2) to [out=90,in=0] (.2,0.3) 
-- (-.2,0.3) to[out=180,in=90] (-.3,0.2) -- (-.3,-0.2) to[out=-90,in=180] (-.2,-0.3);
        \end{scope}
    \end{tikzpicture}
    \caption{Case $d_{\gamma_1}=3.5,\ d_{\gamma_2}=3.5,\ \delta=5$. The blue node is the center of $X^{\gamma_1}\cap X^{\gamma_2}$}
    \end{figure}
    \item Case (G). Now $d_{\gamma_1}$ is an integer and $d_{\gamma_2},\delta$ are  half-integers. The center of $X^{\gamma_1}\cap X^{\gamma_2}$ a vertex if $d_{\gamma_1}-d_{\gamma_2}+\delta$ is even and a midpoint if it is odd. Again, since $d_{\gamma_2}$ is a half-integer, this corresponds to the opposite parity for $d_{\gamma_1}+d_{\gamma_2}+\delta$. We therefore have
      \[\mathrm{Orb}(\gamma_1,\gamma_2) = \left\lbrace \begin{array}{ll}
       0& \text{if }d_{\gamma_1}+d_{\gamma_2}<\delta    \\
       &\\
       2\frac{q^{\frac{1}{2}(d_{\gamma_1}+d_{\gamma_2}-\delta+1)
       }-1}{q-1}  &\text{if }\left\lbrace\begin{array}{l}
            |d_{\gamma_1}-d_{\gamma_2}|<\delta  \\
              d_{\gamma_1}+d_{\gamma_2}+\delta \equiv0\pmod2
        \end{array}\right.\\ &\\
      1 + (1+q)\frac{q^{\frac{1}{2}(d_{\gamma_1}+d_{\gamma_2}-\delta)}-1}{q-1}  &\text{if }\left\lbrace\begin{array}{l}
            |d_{\gamma_1}-d_{\gamma_2}|<\delta  \\
              d_{\gamma_1}+d_{\gamma_2}+\delta \equiv1\pmod2
        \end{array}\right.  \\
        &\\  
    2\frac{q^{d_{\gamma_2}+1/2}-1}{q-1}    &\text{if }\left\lbrace\begin{array}{l}
            \delta\leq|d_{\gamma_1}-d_{\gamma_2}|\\
              d_{\gamma_1}> d_{\gamma_2}
        \end{array}\right. \\ &\\
         1+(1+q)\frac{q^{d_{\gamma_1}}-1}{q-1}    &\text{if }\left\lbrace\begin{array}{l}
            \delta\leq |d_{\gamma_1}-d_{\gamma_2}|\\
              d_{\gamma_1}< d_{\gamma_2}
        \end{array}\right. 
    \end{array} \right.\]

   \begin{figure}[H]
   \centering
    \begin{tikzpicture}
        \draw[thick] (0,0) -- (7,0);
        \draw[thick, dashed] (7,0) -- (8,0);
        \begin{scope}[shift=(0:4)]    
        \draw[thick] (0,0) -- (0,0.5);
        \begin{scope}[shift=(90:0.5)]
        \draw[thick] (0,0) -- (120:0.6);
        \draw[thick] (0,0) -- (60:0.6);
        \begin{scope}[shift=(120:0.6)]
        \draw[thick] (0,0) -- (120:0.5);
        \draw[thick] (0,0) -- (60:0.5);
        \end{scope}
        \begin{scope}[shift=(60:0.6)]
        \draw[thick] (0,0) -- (120:0.5);
        \draw[thick] (0,0) -- (60:0.5);
        \end{scope}
        \end{scope}
        \end{scope}
        \foreach \i in {0,1,...,7}{
        \node[draw,circle,inner sep=2pt,fill=black, thick] at (\i,0) {};
        }
        \foreach \i in {0,1,...,4}{
        \node[draw,circle,inner sep=3pt,color=red, thick] at (\i,0) {};
        }
        \node[draw,circle,inner sep=1pt,fill=blue, thick] at (3.5,0) {};
        \node[draw,circle,inner sep=2pt,fill=red, thick] at (0,0) {};
        \node[draw,circle,inner sep=1pt,fill=green!40!gray, thick] at (5.5,0) {};
        \foreach\j in{3,4,...,7}{
        \begin{scope}[shift=(0:\j),scale=0.5]
        \draw[very thick,color=green!40!gray] (-.2,-0.3) -- (.2,-0.3) to [out=0,in=-90] (.3,-0.2) -- (.3,0.2) to [out=90,in=0] (.2,0.3) 
-- (-.2,0.3) to[out=180,in=90] (-.3,0.2) -- (-.3,-0.2) to[out=-90,in=180] (-.2,-0.3);
        \end{scope}}
        \begin{scope}[shift={(4,0.6)},scale=0.5]
        \node[draw,circle,inner sep=2pt,fill=black, thick] at (0,0) {};
        \draw[very thick,color=green!40!gray] (-.2,-0.3) -- (.2,-0.3) to [out=0,in=-90] (.3,-0.2) -- (.3,0.2) to [out=90,in=0] (.2,0.3) 
-- (-.2,0.3) to[out=180,in=90] (-.3,0.2) -- (-.3,-0.2) to[out=-90,in=180] (-.2,-0.3);
        \end{scope}
    \end{tikzpicture}
    \caption{Case $d_{\gamma_1}=4,\ d_{\gamma_2}=2.5,\ \delta=5.5$. The blue node is the center of $X^{\gamma_1}\cap X^{\gamma_2}$}
    \end{figure}
       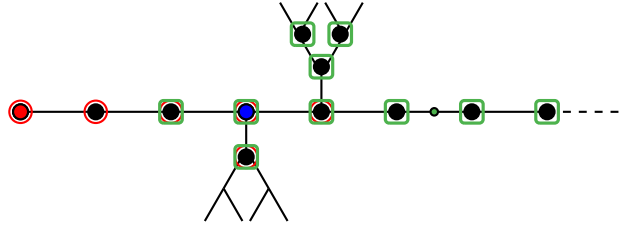
\begin{figure}[H]
   \centering
    \begin{tikzpicture}
        \draw[thick] (0,0) -- (7,0);
        \draw[thick, dashed] (7,0) -- (8,0);
        \begin{scope}[shift=(0:4)]    
        \draw[thick] (0,0) -- (0,0.5);
        \begin{scope}[shift=(90:0.5)]
        \draw[thick] (0,0) -- (120:0.6);
        \draw[thick] (0,0) -- (60:0.6);
        \begin{scope}[shift=(120:0.6)]
        \draw[thick] (0,0) -- (120:0.5);
        \draw[thick] (0,0) -- (60:0.5);
        \end{scope}
        \begin{scope}[shift=(60:0.6)]
        \draw[thick] (0,0) -- (120:0.5);
        \draw[thick] (0,0) -- (60:0.5);
        \end{scope}
        \end{scope}
        \end{scope}
        \begin{scope}[shift=(0:3),rotate=180]    
        \draw[thick] (0,0) -- (0,0.5);
        \begin{scope}[shift=(90:0.5)]
        \draw[thick] (0,0) -- (120:0.6);
        \draw[thick] (0,0) -- (60:0.6);
        \begin{scope}[shift=(120:0.6)]
        \draw[thick] (0,0) -- (120:0.5);
        \draw[thick] (0,0) -- (60:0.5);
        \end{scope}
        \begin{scope}[shift=(60:0.6)]
        \draw[thick] (0,0) -- (120:0.5);
        \draw[thick] (0,0) -- (60:0.5);
        \end{scope}
        \end{scope}
        \end{scope}
        \foreach \i in {0,1,...,7}{
        \node[draw,circle,inner sep=2pt,fill=black, thick] at (\i,0) {};
        }
        \foreach \i in {0,1,...,4}{
        \node[draw,circle,inner sep=3pt,color=red, thick] at (\i,0) {};
        }
        \node[draw,circle,inner sep=2pt,fill=blue, thick] at (3,0) {};
        \node[draw,circle,inner sep=2pt,fill=red, thick] at (0,0) {};
        \node[draw,circle,inner sep=1pt,fill=green!40!gray, thick] at (5.5,0) {};
        \foreach\j in{2,3,...,7}{
        \begin{scope}[shift=(0:\j),scale=0.5]
        \draw[very thick,color=green!40!gray] (-.2,-0.3) -- (.2,-0.3) to [out=0,in=-90] (.3,-0.2) -- (.3,0.2) to [out=90,in=0] (.2,0.3) 
-- (-.2,0.3) to[out=180,in=90] (-.3,0.2) -- (-.3,-0.2) to[out=-90,in=180] (-.2,-0.3);
        \end{scope}}
        \begin{scope}[shift={(4,0.6)},scale=0.5]
        \node[draw,circle,inner sep=2pt,fill=black, thick] at (0,0) {};
        \draw[very thick,color=green!40!gray] (-.2,-0.3) -- (.2,-0.3) to [out=0,in=-90] (.3,-0.2) -- (.3,0.2) to [out=90,in=0] (.2,0.3) 
-- (-.2,0.3) to[out=180,in=90] (-.3,0.2) -- (-.3,-0.2) to[out=-90,in=180] (-.2,-0.3);
        \end{scope}
        \begin{scope}[shift={(4,0.6)}]
        \begin{scope}[shift=(60:0.5),scale=0.5]
        \node[draw,circle,inner sep=2pt,fill=black, thick] at (0,0) {};
        \draw[very thick,color=green!40!gray] (-.2,-0.3) -- (.2,-0.3) to [out=0,in=-90] (.3,-0.2) -- (.3,0.2) to [out=90,in=0] (.2,0.3) 
-- (-.2,0.3) to[out=180,in=90] (-.3,0.2) -- (-.3,-0.2) to[out=-90,in=180] (-.2,-0.3);
        \end{scope}
        \end{scope}
         \begin{scope}[shift={(4,0.6)}]
        \begin{scope}[shift=(120:0.5),scale=0.5]
        \node[draw,circle,inner sep=2pt,fill=black, thick] at (0,0) {};
        \draw[very thick,color=green!40!gray] (-.2,-0.3) -- (.2,-0.3) to [out=0,in=-90] (.3,-0.2) -- (.3,0.2) to [out=90,in=0] (.2,0.3) 
-- (-.2,0.3) to[out=180,in=90] (-.3,0.2) -- (-.3,-0.2) to[out=-90,in=180] (-.2,-0.3);
        \end{scope}
        \end{scope}
        \begin{scope}[shift={(3,-0.6)},scale=0.5]
        \node[draw,circle,inner sep=3pt,color=red, thick] at (0,0) {};
        \node[draw,circle,inner sep=2pt,fill=black, thick] at (0,0) {};
        \draw[very thick,color=green!40!gray] (-.2,-0.3) -- (.2,-0.3) to [out=0,in=-90] (.3,-0.2) -- (.3,0.2) to [out=90,in=0] (.2,0.3) 
-- (-.2,0.3) to[out=180,in=90] (-.3,0.2) -- (-.3,-0.2) to[out=-90,in=180] (-.2,-0.3);
        \end{scope}
    \end{tikzpicture}
\caption{Case $d_{\gamma_1}=3.5,\ d_{\gamma_2}=3.5,\ \delta=5.5$. The blue node is the center of $X^{\gamma_1}\cap X^{\gamma_2}$}
    \end{figure}
\end{itemize}

\subsection{Mixed case}

We are only left with the case where --up to permutation-- $\gamma_1$ is hyperbolic and $\gamma_2$ is elliptic. The field extension $F[\gamma_2]/F$ may still be either ramified or unramified but surprisingly, as shown in Theorem \ref{thm:balltube}, the value of the integral does not depend on it. We get 

\begin{itemize}
    \item Case (h). If $\gamma_1$ is hyperbolic and $\gamma_2$ is elliptic, then 
    \[\mathrm{Orb}(\gamma_1,\gamma_2) = \left\lbrace \begin{array}{ll}
         0& \text{if }d_{\gamma_1}+d_{\gamma_2}<\delta  \\ & \\
         1 + (1+q)\frac{q^{\frac{1}{2}(d_{\gamma_1}+d_{\gamma_2}-\delta)}-1}{q-1}  &\text{if }\left\lbrace\begin{array}{l}
            |d_{\gamma_1}-d_{\gamma_2}|\leq \delta  \\
              d_{\gamma_1}+d_{\gamma_2}-\delta \equiv0\pmod2
        \end{array}\right.  \\
        &\\  
        2\frac{q^{\frac{1}{2}(d_{\gamma_1}+d_{\gamma_2}-\delta+1)
       }-1}{q-1}  &\text{if }\left\lbrace\begin{array}{l}
            |d_{\gamma_1}-d_{\gamma_2}|\leq\delta  \\
              d_{\gamma_1}+d_{\gamma_2}-\delta \equiv1\pmod2
        \end{array}\right.\\ &\\
          1+(1+q)\frac{q^{d_{\gamma_2}}-1}{q-1}    &\text{if }\left\lbrace\begin{array}{l}
            d_{\gamma_2}-d_{\gamma_1}<\delta< |d_{\gamma_1}-d_{\gamma_2}|\\
              d_{\gamma_2}\in \mathbb{Z}
        \end{array}\right. \\ &\\
            2\frac{q^{d_{\gamma_2}+1/2}-1}{q-1}    &\text{if }\left\lbrace\begin{array}{l}
            d_{\gamma_2}-d_{\gamma_1}<\delta< |d_{\gamma_1}-d_{\gamma_2}|\\
              d_{\gamma_2}\notin \mathbb{Z}
        \end{array}\right. \\ &\\
         {[}1+2(d_{\gamma_2}-d_{\gamma_1}-\delta)]q^{d_{\gamma_1} }+ 2\frac{q^{d_{\gamma_1}}-1}{q-1}&\text{if }0\leq \delta\leq d_{\gamma_2}-d_{\gamma_1}
    \end{array}\right.\]
    Note that we looked at the parity of $d_{\gamma_1}+d_{\gamma_2}-\delta$ as opposed to  $d_{\gamma_1}+d_{\gamma_2}+\delta$ used previously, since it allows us to match computations of cases (E) and (G) and avoid the distinction of whether $F[\gamma_2]$ is ramified or not. 
\end{itemize}
\bibliographystyle{amsalpha}
\bibliography{biblio}
\end{document}